\documentclass[a4paper,11pt]{article}
\usepackage[utf8]{inputenc}
\usepackage[T1]{fontenc}
\usepackage{lmodern}

\usepackage{amsthm,amsmath,amsfonts,amssymb,bbm,mathrsfs}
\usepackage{mathtools}
\usepackage{enumitem}
\usepackage{url}
\usepackage{dsfont}
\usepackage{appendix}
\usepackage{amsthm}
\usepackage[dvipsnames]{xcolor}
\usepackage{graphicx}
\usepackage{array} 

\usepackage[colorlinks=true, linkcolor=NavyBlue, urlcolor=black, citecolor=NavyBlue,pdfstartview=FitH]{hyperref}

\usepackage[english]{babel}

\usepackage{caption,tikz,subfigure}

\usepackage[top=2.7cm, bottom=2.7cm, left=2.5cm, right=2.5cm]{geometry}

\makeatletter

\@addtoreset{equation}{section}
\makeatother

\setlist[enumerate,1]{label=(\roman*), font = \normalfont} 

\let\originalleft\left
\let\originalright\right
\renewcommand{\left}{\mathopen{}\mathclose\bgroup\originalleft}
\renewcommand{\right}{\aftergroup\egroup\originalright}

\newlength{\bibitemsep}
\newlength{\bibparskip}
\let\oldthebibliography\thebibliography
\renewcommand\thebibliography[1]{\oldthebibliography{#1}
	\setlength{\parskip}{\bibitemsep}
	\setlength{\itemsep}{\bibparskip}}

\newcommand{\N}{\mathbb{N}}
\newcommand{\Z}{\mathbb{Z}}

\newcommand{\R}{\mathbb{R}}

\renewcommand{\P}{\mathbb{P}}
\newcommand{\E}{\mathbb{E}}

\newcommand{\1}{\mathbbm{1}}

\newcommand{\cB}{\mathcal{B}}
\newcommand{\cC}{\mathcal{C}}

\newcommand{\cF}{\mathcal{F}}
\newcommand{\cL}{\mathcal{L}}
\newcommand{\cN}{\mathcal{N}}

\newcommand{\Ec}[1]{\mathbb{E} \left[#1\right]}
\newcommand{\Pp}[1]{\mathbb{P} \left(#1\right)}
\newcommand{\Ecsq}[2]{\mathbb{E} \left[#1\middle|#2\right]}
\newcommand{\Ppsq}[2]{\mathbb{P} \left(#1\middle|#2\right)}

\newcommand{\diff}{\mathop{}\mathopen{}\mathrm{d}}

\newcommand{\abs}[1]{\left\lvert#1\right\rvert}

\newcommand{\norme}[1]{\left\lVert#1\right\rVert}

\newcommand{\ind}[1]{\mathbbm{1}_{\{#1\}}}

\newcommand{\eps}{\varepsilon}

\theoremstyle{plain}
\newtheorem{thm}{Theorem}[section]
\newtheorem{prop}[thm]{Proposition}
\newtheorem{lem}[thm]{Lemma}
\newtheorem{cor}[thm]{Corollary}

\theoremstyle{definition}

\theoremstyle{remark}
\newtheorem{rem}[thm]{Remark}

\title{Polynomial slowdown in an angle-dependent 2d branching Brownian motion}
\author{Julien Berestycki\footnote{Department of Statistics, University of Oxford, 24-29 St Giles, OX1 3LB Oxford, United Kingdom.} 
\and 
David Geldbach$^{\ast}$ 
\and 
Michel Pain\footnote{Institut de Mathématiques de Toulouse (UMR5219), Université de Toulouse, CNRS; UPS, F-31062 Toulouse Cedex 9, France.}}

\begin{document}

\maketitle

\begin{abstract}
    We consider a branching Brownian motion in $\R^2$ in which particles independently diffuse as standard Brownian motions and branch at an inhomogeneous rate $b(\theta)$ which depends only on the angle $\theta$ of the particle. We assume that $b$ is maximal when $\theta=0$, which is the preferred direction for breeding. Furthermore we assume that
    $
	b(\theta) = 1 - \beta \abs{\theta}^\alpha + O(\theta^2)$,
	as $\theta \to 0$, for $\alpha \in (2/3,2)$ and $\beta>0.$ 
    We show that if $M_t$ is the maximum distance to the origin at time $t$, then $(M_t-m(t))_{t\ge 1}$ is tight where
    \[
m(t) = \sqrt{2} t - \frac{\vartheta_1}{\sqrt{2}} t^{(2-\alpha)/(2+\alpha)} -  \left(\frac{3}{2\sqrt{2}} - \frac{\alpha}{2\sqrt{2}(2+\alpha)}\right) \log t.
\]
and $\vartheta_1$ is explicit in terms of the first eigenvalue of a certain operator.
\end{abstract}

{\noindent \bf 2010 Mathematics Subject Classification:} 
 60J80, 
 60J65, 
 35R99. 
 \\
{\bf Keywords:} Branching Brownian motion, maximal displacement, inhomogenous environment.

\begin{figure}[htbp]
    \centering
    \includegraphics[width=0.95\textwidth]{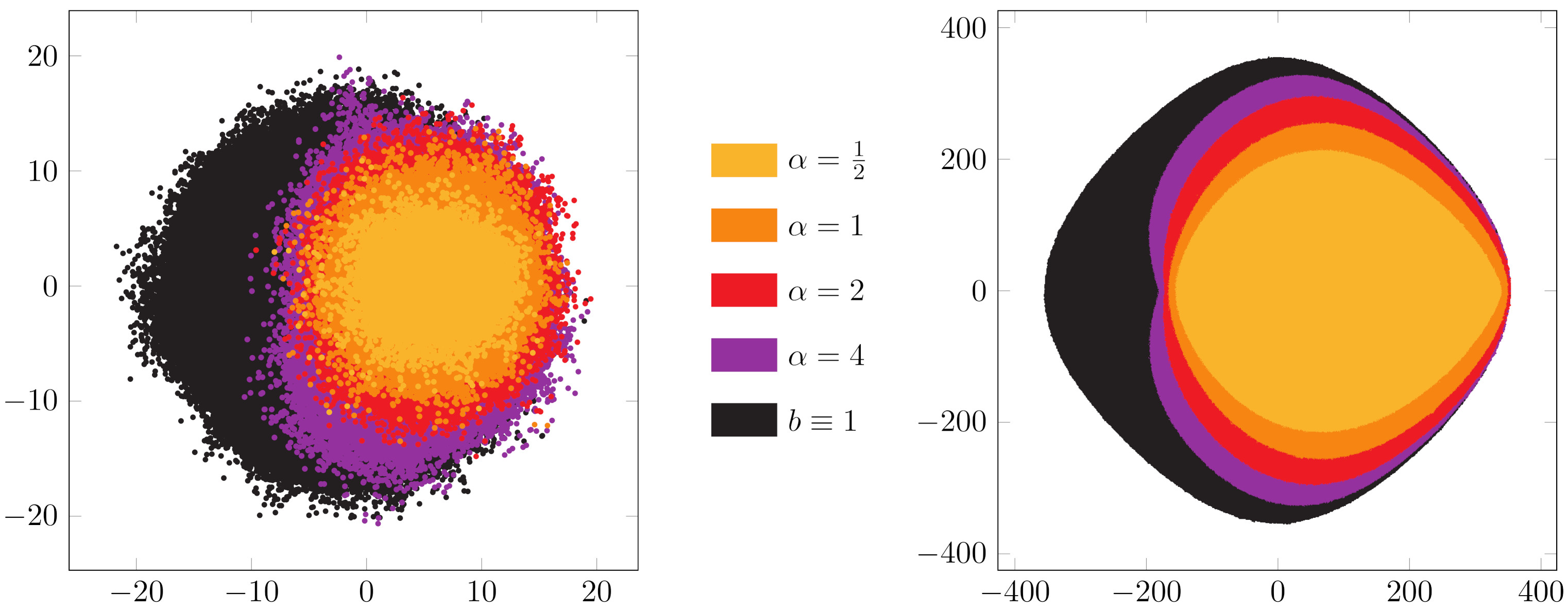} 
    \caption{On the left, a realisation at time $t=16$ of space--inhomogenous BBM with $\alpha\in\{1/2, 1, 2,4\}$ and branching rate $b(\theta) = 1-\abs{\sin(\theta/2)}^\alpha$, as well as homogenous BBM in black (which can be seen as $\alpha = \infty$). 
    The processes are coupled in such a way that every particle that exists for a smaller $\alpha$ also exists for a greater value of~$\alpha$. 
    On the right, a realisation of a version of the model with discrete time and positions in $\Z^2$. 
    In this model, each particle at time $n$ and with angular coordinate $\theta$ has 2 children at time $n+1$ with probability $b(\theta)$ and 1 child otherwise, and each child jumps to one of the four nearest neighbour sites, or stays on site with equal probability.
    This discretisation allows to push the numerical simulation to larger times, here $n=500$, but it changes the limiting shape of the process. 
    However, the behaviour of the maximal displacement, with the appearance of a polynomial slowdown for $\alpha<2$, should be universal.}
    \label{fig:simulation2}
\end{figure}

{
	\hypersetup{linkcolor=black}
	\tableofcontents
}

\section{Introduction}

\subsection{Model and main result}

The standard, one-dimensional Branching Brownian Motion (BBM) is a system of particles on~$\mathbb{R}$: at time $t=0$, there is a single particle at the origin. This particle performs a Brownian motion up to an exponential random time with mean one when the particle is replaced by two particles at the same position. Each daughter particle then starts an independent copy of the same process, and so on. This results in an exponentially growing cloud of branching, diffusing, particles $(X_u^{1\mathrm{d}}(t),u\in \mathcal{N}^{1\mathrm{d}}(t))_{t\geq0}$ where $X_u^{1\mathrm{d}}(t)$ is the position of particle $u$ at time $t$ and $\mathcal{N}^{1\mathrm{d}}(t)$ is the set of particles alive at time $t$. 
This is a well studied process with an extensive literature, including a particular focus on the study of the maximal displacement at time $t$,
\begin{equation*}
M^{1\mathrm{d}}_t \coloneqq \max_{u\in \mathcal{N}^{1\mathrm{d}}(t)} X_u^{1\mathrm{d}}(t). 
\end{equation*}
Classical results by \cite{Bra1978,Bra1983,LalSel1987} determined the asymptotic behavior of $M^{1\mathrm{d}}_t$, let $m^{1\mathrm{d}}(t) = \sqrt{2}t - \frac{3}{2\sqrt{2}}\log t$,
\begin{equation}\label{eq:1d_maximum}
\lim_{t \to \infty}\mathbb{P} \left( M^{1\mathrm{d}}_t - m^{1\mathrm{d}}(t) \leq x \right) = \mathbb{E}\left[e^{-cD_\infty e^{-\sqrt{2}x}}\right],
\end{equation}
where $c>0$ and $D_{\infty}>0$ is the limit of the so-called derivative martingale. 
Further results include for example the convergence of the extremal process \cite{AidBerBruShi2013,ArgBovKis2013} and a precise description of its limit \cite{CorHarLou2019,MytRoqRyz2022,BerBruGraMytRoqRyz2022,HarLouWu2024}.

Branching Brownian motion is an idealised model for how a growing diffusing population invades an environment. In reality, though, environments are rarely perfectly homogeneous and often more than one-dimensional. This has in part motivated the introduction of inhomogeneous branching particle systems. The inhomogeneity can be in time as in \cite{BovHar2014,BovHar2015,BovHar2020,BovKur2004}, or in space such as in \cite{BerBruHarHar2010,BerBruHarHarRob2015} or in \cite{OzEng2019} (see the references therein as well for more information on the branching Brownian motion in Poissonian obstacles). 
The behaviour of the maximal displacement for such variants of the BBM model has been a topic of intense research in the recent past and we discuss some of these results in the next subsection. 

In the present work, we study a spatially inhomogeneous version of BBM in $\mathbb{R}^2$, in which there is a preferred direction along which the environment is more favourable: we choose to make the branching rate of a particle depend on its angular coordinate, with the goal of capturing the effect of the inhomogeneity on the maximal displacement\footnote{Although a natural idea could be to make the branching rate $b$ a function of the $y$ coordinate which is maximum at $y=0$ so that the $x$ axis is the preferred direction, as we discuss below this makes the process very similar to \emph{branching Brownian motion in a strip}. By making $b$ depends on the angle, we uncover a much richer variety of behaviours. We further discuss this in Section \ref{subsec:open_questions}.}.  
Let $b \colon \R \to [0,\infty)$ be a $2\pi$-periodic function.
We start with one particle at $(0,0)$ at time $0$.
Each particle moves according to a 2d Brownian motion and a particle located at a point of polar coordinates $(r,\theta)$ splits into two particles at rate $b(\theta)$.
We denote by $\cN(t)$ the set of particles alive at time $t$ and, for $u \in \cN(t)$, by $(X_u(t),Y_u(t))$ the position of particle $u$ at time $t$.
For $s \leq t$, we also write $(X_u(s),Y_u(s))$ for the position of the ancestor of $u$ alive at time $s$.
Finally, we denote by $(R_u(t),\theta_u(t))$ the polar coordinates of $(X_u(t),Y_u(t))$ chosen such that $\theta_u(t) \in (-\pi,\pi]$.

We are interested in the asymptotic behaviour of 
\[
	M_t \coloneqq \max_{u\in\cN(t)} R_u(t)
\]
in the case where the function $b$ reaches its maximum on $[-\pi,\pi]$ at a single point, say $0$.
More precisely we assume 
\begin{enumerate}[label=(A\arabic*),leftmargin=35pt]
    \item $b$ is continuous and reaches its maximum on $[-\pi,\pi]$ at 0 only. \label{assumption1}
    \item there exist $\alpha \in (2/3,2)$ and $\beta > 0$ such that $b(\theta) = 1 - \beta \abs{\theta}^\alpha + O(\theta^2)$, as $\theta \to 0$, \label{assumption2}
\end{enumerate}
Instead of Assumption \ref{assumption1}, one can assume the weaker assumption
\begin{enumerate}[label=(A\arabic*'),leftmargin=35pt]
    \item $b \leq 1$ and, if $b(\theta_n) \to 1$ for some sequence $(\theta_n)_{n\geq0} \in (-\pi,\pi]^\N$, then $\theta_n \to 0$. \label{assumption1'}
\end{enumerate}
This assumption does not require continuity anymore, but ensures that $b$ can approach the value $1$ only at $0$ (and other multiples of $2\pi$). Note that by Assumption~\ref{assumption2} we still have continuity at $0$. 

To state our result, we introduce
\[
	\vartheta_1 \coloneqq \lambda_0 \cdot \frac{2+\alpha}{2-\alpha} \cdot \frac{\beta^{2/(2+\alpha)}}{2^{2\alpha/(2+\alpha)}},
\]
where $\lambda_0 >0$ is the smallest eigenvalue of the differential operator $\cL$ defined by $(\cL f)(x) \coloneqq -f''(x) + \abs{x}^\alpha f(x)$ (see Section \ref{sec:sturm-liouville} and Proposition~\ref{prop:sturm-liouville} for details),
and 
\[
m(t) \coloneqq \sqrt{2} t - \frac{\vartheta_1}{\sqrt{2}} t^{(2-\alpha)/(2+\alpha)} -  \left(\frac{3}{2\sqrt{2}} - \frac{\alpha}{2\sqrt{2}(2+\alpha)}\right) \log t.
\]

\begin{thm}\label{thm:main} Assume \ref{assumption1'} and \ref{assumption2}.
    Then, as $t \to \infty$, we have
    \[
    M_t  = m(t) + O_\P(1),
    \]
    which means that, for any $\varepsilon>0$, there exists $a\geq 0$ such that for $t$ large enough $\P(\abs{M_t-m(t)} \geq a) \leq \varepsilon$.
\end{thm}

Recalling that the $\theta=0$ direction is the one with maximal branching rate, it is reasonable to expect that particles with the maximal displacement are located in that direction.
More precisely, we can deduce the following result from our proof, which states that particles at a distance $m(t)+O(1)$ from the origin at time $t$ cannot have a too large $Y$-coordinate, so that their displacement is essentially due to the $X$-coordinate. 
The exponent $\alpha/(2+\alpha)$ appearing in the first part is optimal, see Remark~\ref{rem:optimal_exponent} for details.

\begin{prop}\label{porism} Assume \ref{assumption1'} and \ref{assumption2}.
	\begin{enumerate}
		\item\label{it:Y} For any $\varepsilon > 0$ and $a\geq 0$, we have
		\begin{equation*}
		\Pp{ \exists u \in \cN(t) : 
			R_u(t) \geq m(t)-a,
			\abs{Y_u(t)} \geq t^{\frac{\alpha}{2+\alpha} +\varepsilon} }
		\xrightarrow[t\to\infty]{} 0.
		\end{equation*}
		\item\label{it:angle} For any $\varepsilon > 0$ and $a\geq 0$, we have
		\begin{equation*}
		\Pp{ \exists u \in \cN(t) : 
			R_u(t) \geq m(t)-a,
			\abs{\theta_u(t)} \geq t^{-\frac{2}{2+\alpha} +\varepsilon} }
		\xrightarrow[t\to\infty]{} 0.
		\end{equation*}
		\item\label{it:X} The following convergence holds in probability:
		\begin{equation*}
			M_t - \max_{u\in\cN(t)} X_u(t) \xrightarrow[t\to\infty]{\P} 0.
		\end{equation*}
	\end{enumerate}
\end{prop}

Theorem~\ref{thm:main} has to be understood in view of equation \eqref{eq:1d_maximum}. 
Compare $m(t)$ and $m^{1\mathrm{d}}(t)$: introducing the inhomogeneous branching rate results in a strong slowdown of the maximal displacement proportional to $t^{(2-\alpha)/(2+\alpha)}$.
With our assumption of $\alpha \in (2/3, 2)$, the exponent takes values in $(0,1/2)$, see Figure \ref{fig:exponent} for an illustration. 
This range of~$\alpha$ corresponds to one of the phases of the model and we expect different behaviours outside of this range: for $\alpha \geq 2$, the polynomial correction should disappear and, for $\alpha \leq 2/3$, the fact that the exponent becomes larger than $1/2$ leads to new effects resulting in additional polynomial terms in $m(t)$. This is why we are focusing on the case $\alpha \in (2/3, 2)$ and we leave other cases as open questions which are discussed in Section~\ref{subsec:open_questions}.
Another difference between Theorem \ref{thm:main} and \eqref{eq:1d_maximum} is that we obtain tightness instead of convergence in distribution. We still believe that convergence is true for our model but the random variable playing the role of $D_\infty$ would be the limit of an approximate martingale, instead of an exact martingale, see Section~\ref{subsec:open_questions} for more details.

\begin{figure}[h]
	\centering
	\begin{tikzpicture}[scale=4]
	\draw [very thin, gray!30] (0,0) grid[step=.2] (2.15,1.15);
	\draw[->,>=latex] (0,0) -- (2.15,0) node[below right]{$\alpha$};
	\draw[->,>=latex] (0,0) -- (0,1.15);
	\draw (0,0) node[below left]{\small $0$};
	\draw (1,0) node[below]{\small $1$};
	\draw (2,0) node[below]{\small $2$};
	\draw (0,1) node[left]{\small $1$};
	\draw[dashed] (2/3,0) -- (2/3,1/2);
	\draw[dashed] (0,1/2) -- (2/3,1/2);
	\draw (2/3,0) node[below]{\small $\frac{2}{3}$};
	\draw (0,1/2) node[left]{\small $\frac{1}{2}$};
	\draw[NavyBlue,thick,domain=0:2] plot (\x, {(2-\x)/(2+\x)});
	\end{tikzpicture}
\caption{The exponent of the polynomial slowdown as function of $\alpha$, $\alpha \mapsto \frac{2-\alpha}{2+\alpha}$. Note that our assumption $\alpha \in (2/3, 2)$ corresponds to the exponent taking values in $(0,1/2)$.}
\label{fig:exponent}
\end{figure}
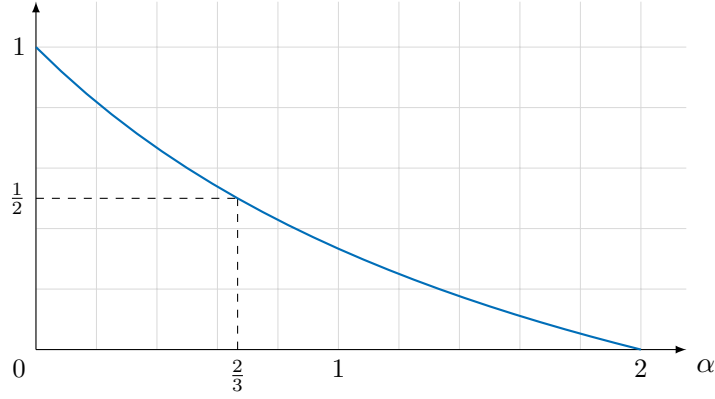

\subsection{Related work}
\label{sec:related_work}

Our result should be compared to other variants of branching Brownian motion where the maximal displacement has been studied.
Firstly, BBM with homogeneous branching rate, i.e.\@ $b\equiv 1$, has been studied in dimensions higher than 1. 
Convergence in distribution of the maximal displacement after recentering, as well as of the extremal process, has been obtained in a series of works \cite{Mal2015b,StaBerMal2021,kim_maximum_2023,BerKimLubMalZei2024}.
In particular, in dimension 2, the proper recentering is $m^{2\mathrm{d}}(t) =  \sqrt{2} t -\frac{1}{\sqrt2} \log t$: comparing this to \eqref{eq:1d_maximum}, we see that changing the dimension yields a change in the logarithmic correction. 

\paragraph{Space-inhomogeneous BBMs.}
Some models with a space-inhomogeneous branching rate have been studied in dimension 1.
In \cite{HarHar2009}, Harris and Harris considered a BBM on $\R$, where a particle located at $x$ branches at rate $\abs{x}^p$ for some $p \in (0,2]$. They proved that the maximal displacement grows like $t^{2/(2-p)}$ if $p<2$, and exponentially fast if $p=2$. Further results on the growth of the population along any path have been obtained in \cite{BerBruHarHar2010,BerBruHarHarRob2015}.
In \cite{RobSch2021,LiuSch2023}, the authors studied a BBM with space-inhomogeneous branching and death rate and a drift, which are tuned in a specific fashion to exhibit a Gaussian travelling wave phenomenon.

Another space-inhomogeneous version of branching Brownian motion is catalytic BBM. Here particles move as Brownian motions on $\mathbb{R}$ and branch at rate $\beta_0 \diff L_t$ where $L_t$ is the local time at $0$ of a given particle. 
In \cite{BocHar2014}, it has been proved that the maximum grows at speed $\beta_0/2$ a.s. and, in \cite{BocHar2016}, that after centering by $(\beta_0/2) t$ the maximum converges in distribution. Note the absence of logarithmic correction. 
Similar results had been obtained previously in \cite{CarHu2014} in the case of the catalytic branching random walk.
Finally, in \cite{BocWan2019}, a BBM with branching rate $\beta + \beta_0 \diff L_t$ is studied and the linear speed is obtained. 

\paragraph{Time-inhomogeneous BBMs.}
Another very related model is \textit{time-inhomogeneous} (or \textit{speed varying}) branching Brownian motion, and its discrete-time counterpart, which has been introduced by Bovier and Kurkova in \cite{BovKur2004}. In time-inhomogeneous BBM, one starts with a profile $\sigma \colon [0,1] \mapsto [0,\infty)$ and a time horizon $T>0$. Particles then branch at rate $1$ and move as Brownian motions on $\mathbb{R}$ with time-inhomogeneous instantaneous variance $\sigma^2(t/T)$ at time $t \leq T$. 
Note that, up to a deterministic time change, the time-inhomogeneous variance can be transformed into a time-inhomogeneous branching rate.
The linear speed of the maximal displacement at time $T$ has been obtained by \cite{BovKur2004} in the discrete-time case, for general profiles $\sigma$.

Of particular interest to us here is the case where $\sigma$ is strictly decreasing, because in that case the maximal displacement exhibits a polynomial correction of order $T^{1/3}$, as proved by \cite{FanZei2012b} for the BBM and \cite{Mal2015a} for the branching random walk.
Maillard and Zeitouni \cite{MaiZei2016} proved the tightness of the maximum after centering by $m(T) = v_\sigma T - w_\sigma T^{1/3} -\frac{\sigma(1)}{\sqrt{2}} \log T$ for some explicit constants $v_\sigma,w_\sigma$.
Hence, our space-inhomogeneous model provides another occurrence of such a polynomial slowdown for the maximal displacement. However, two differences can be noticed. First, our model is consistently defined for all time, it does not depend on the time horizon $T$. Second, all exponents in $(0,1)$ for the polynomial corrections can be obtained in our model, not only $1/3$: indeed, Theorem \ref{thm:main} includes all exponents in $(0,1/2)$ but the case $\alpha \in (0,2/3]$ should include the exponents in $[1/2,1)$ as explained in Section \ref{subsec:open_questions}.


Other cases for the choice of $\sigma$ have been studied, without any polynomial correction appearing in $m(T)$, but various possible coefficients for the logarithmic correction \cite{FanZei2012a,BovHar2014,BovHar2015,BovHar2020,alban_1_2025}. 
Similarly, we believe that, for our model in the case $\alpha \geq 2$, $m(t)$ should not include a polynomial correction but a logarithmic correction with any coefficient in $(1/\sqrt{2},\infty)$, see Section \ref{subsec:open_questions} for details.


\subsection{Heuristics and structure of the paper}
\label{sec:heuristics}

Theorem \ref{thm:main} says that the maximal displacement at time $t$ is found at a distance of $m(t) + O(1)$ from the origin. 
In this section, we discuss informally how to arrive at the right polynomial correction in $m(t)$ and describe the broad strategy of our proof.

We want to find $m(t)$ such that $\P(M_t \geq m(t))$ is strictly between 1 and 0. By a union bound 
\[
\P(M_t \geq m(t)) 
\le \E[\# \{u \in \cN(t): R_u(t) \geq m(t) \}],
\]
hence we want to find $m(t)$ so that this expectation is of order 1.
Note that this union bound is not sharp enough to get the logarithmic correction in $m(t)$, but is sufficient to predict the polynomial term.
By the many-to-one lemma, we have that  
\begin{equation} \label{eq:heuristics1}
    \E \left[\# \{u \in \cN(t): R_u(t) \geq m(t)  \} \right] = \E\left[  \exp\left({\int_0^t b(\theta_s) \diff s } \right)
    \ind{R_t \geq m(t)} \right] ,
\end{equation}
where $(R_t,\theta_t)_{t\geq 0}$ are the polar coordinates of a single two-dimensional Brownian motion of coordinates $(X_t,Y_t)_{t\geq 0}$ (so that $X$ and $Y$ are two independent standard Brownian motions).

Because $b$ reaches its maximum at 0 only, we expect that the maximal displacement will occur in the direction $\theta=0$.
This means that, on the event $\{R_t \geq m(t)\}$, we expect that $X_t \approx R_t \approx \sqrt{2} t$ while $Y_t = o(t)$.
The most efficient way for $(X,Y)$ to achieve this is to have $X_s \approx \sqrt{2} s$ and $Y_s = o(s)$ for all $s \in [s_0,t]$ with some large but fixed $s_0$, 
meaning that the angle $\theta_s$ can be approximated by $Y_s/(\sqrt{2} s)$.
Therefore, the expectation above should be close to the factorised expression
\begin{equation} \label{eq:heuristics:factorisation}
\E\left[ \exp\left({\int_0^t b(\theta_s) \diff s } \right) \ind{R_t \geq m(t)} \right] 
\approx \E\left[  \exp\left(\int_{s_0}^t b \left( \frac{Y_s}{\sqrt{2}s} \right) \diff s \right) \right]\P\Big(X_t \geq m(t)\Big). 
\end{equation}
Hence, using that $b(\theta) = 1 - \beta \abs{\theta}^\alpha + O(\theta^2)$ as $\theta \to 0$ by Assumption \ref{assumption2} and neglecting the error term, we should choose $m(t)$ so that
\begin{equation}
	e^t \cdot \Ec{\exp\left(-\beta  \int_{s_0}^t \abs{\frac{Y_s}{\sqrt{2} s}}^\alpha \diff s \right)} \cdot \P\Big( X_t > m(t) \Big) \approx 1 \label{eq:sc}.
\end{equation}
Gaussian tail estimates yield $\P(X_t > m(t)) = \exp(-m(t)^2/2t + O(\log t))$. On the other hand, we claim that
\begin{equation}\label{eq:don}
    \Ec{\exp\left(-\beta \int_{s_0}^t \left\vert \frac{Y_s}{\sqrt{2}s} \right\vert^\alpha \diff s\right)} = \exp \left({-c t^{(2-\alpha)/(2+\alpha)} } +O(\log t) \right),
\end{equation}
for some $c>0$. 
With this, \eqref{eq:sc} becomes
\begin{equation}\label{eq:appp}
    \exp\left(t-c t^{(2-\alpha)/(2+\alpha)}-\frac{m(t)^2}{2t}+O(\log t)\right)  \approx 1,
\end{equation}
and, solving for $m(t)$, we obtain $m(t)= \sqrt{2}t-\frac{c}{\sqrt{2}}t^{(2-\alpha)/(2+\alpha)}+O(\log t)$. 

Let us argue informally where \eqref{eq:don} comes from. 
We first briefly explain our proof strategy, which relies on a link with a PDE, and then give some alternative probabilistic heuristics. 
Changing variable with $r=s/t$ and using the scaling property of Brownian motion, we rewrite
\begin{equation}\label{eq:how_to_find_rho}
\beta \int_{s_0}^t \abs{\frac{Y_s}{\sqrt{2}s}}^\alpha \diff s
\overset{(\mathrm{d})}{=} \beta t^{1-\alpha/2} \int_{s_0/t}^1 \abs{\frac{Y_r}{\sqrt{2}r}}^\alpha \diff r
= \varrho \int_{s_0/t}^1 \abs{\frac{\sqrt{\varrho} Y_r}{r}}^\alpha \diff r,
\end{equation}
where $\varrho =c(\alpha,\beta)t^{(2-\alpha)/(2+\alpha)}$ is obtained by solving the last equality.
Using this rewriting and then the Feynman--Kac formula, the expectation in \eqref{eq:don} can therefore be expressed in terms of solutions at time $r=1-\frac{s_0}{t}$ of the following PDE on $[0,1)\times \R$:
\begin{equation}\label{eq:first_pde}
\partial_r u = \varrho \left( \partial_{xx}^2 u - (1-r)^{-\alpha} \abs{x}^{\alpha} u \right).
\end{equation}
The exponential decay of these solutions is governed by the largest eigenvalue of the operator on the right-hand side of \eqref{eq:first_pde}, which is proportional to $\varrho$ and thus leads to~\eqref{eq:don}. 
Our treatment of the time-inhomogeneous PDE to precisely estimate its fundamental solution is inspired in large part by the approach of Maillard and Zeitouni \cite{MaiZei2016}.


There is another probabilistic way of seeing why $\exp(-ct^{(2-\alpha)/(2+\alpha)})$ is the right decay, but this heuristics is not precise enough to predict the exact value of $c$. 
The expectation in \eqref{eq:don} is governed by a trade-off between the advantage for $Y$ to stay close to zero in order to minimize the integral in the exponential, and the cost of doing so.
We can expect that the optimal strategy is of the following form  $A_\gamma = \{\forall s \in [s_0,t], \abs{Y_s} \leq s^\gamma\}$ for some parameter $\gamma \in [0,1/2]$ which we need to optimize. 
Note that, by the scaling property of Brownian motion, the cost of satisfying the constraint $\abs{Y_s} \leq s^\gamma$ is roughly the same on any interval of the form $[r,r+r^{2\gamma}]$ for $r$ large enough. Moreover, the number of such intervals needed to cover $[s_0,t]$ is proportional to $t^{1-2\gamma}$, so we deduce that $\P(A_\gamma) \approx \exp(-c_1 t^{1-2\gamma})$. 
On the other hand, on the event $A_\gamma$,  
we have $\int_{s_0}^t \lvert \frac{Y_s}{\sqrt{2}s} \rvert^\alpha \diff s \approx c_2 t^{[1+\alpha(\gamma-1)] \vee 0} $. 
Therefore, we expect that
\begin{align*}
	\Ec{\exp\left(-\beta \int_{s_0}^t \left\vert \frac{Y_s}{\sqrt{2}s} \right\vert^\alpha \diff s\right)} 
	&\approx \sup_{\gamma \in [0,1/2]} 
	\exp\left(-c_1t^{1-2\gamma}- c_2t^{[1+\alpha(\gamma-1)] \vee 0} \right) \\
	&\approx \exp\left( -c t^{(2-\alpha)/(2+\alpha)}\right),
\end{align*}
where the supremum is achieved by $\gamma = \alpha/(2+\alpha)$ (note that is is consistent with the exponent appearing in Proposition~\ref{porism}.\ref{it:Y}). 
While we don't use this approach in our proofs, this heuristic tells us the scale of $Y_s$ that should dominate the integral, namely $s^{\alpha/(2+\alpha)}$. This will be reflected in our proofs, especially in Section \ref{sec:probabilistic_arguments}. 

The structure of this paper follows these heuristics. In Section~\ref{sec:BM_weighted} we make the connection between \eqref{eq:don} and the PDE \eqref{eq:first_pde} rigorous, and we show precise estimates on its solutions using Sturm--Liouville theory as well as methods inspired from \cite{MaiZei2016}. 
Then, in Section \ref{sec:probabilistic_arguments}, in order to incorporate the error term due to the approximation $b(\theta_s) \approx 1 - \lvert \frac{Y_s}{\sqrt{2}s} \rvert^\alpha $, we deal with more general expectations of the form
\begin{equation} \label{eq:general_expectations}
    \Ec{\exp\left(-\beta \int_{s_0}^t \left\vert \frac{Y_s}{\sqrt{2}s}  \right\vert^\alpha (1+f(Y_s,s)) \diff s\right)},
\end{equation}
where $f$ is thought of as a small error term. 
The PDE approach we use in Section \ref{sec:BM_weighted} does not allow to incorporate $f$ directly, so instead we use probabilistic methods to justify afterwards that the contribution of this error term is negligible.
The main idea is to localize the trajectory of $Y$ by showing that, on the event mainly contributing to this expectation, $Y_s$ stays of order $s^{\alpha/(2+\alpha)}$, so that we can bound the contribution of the error term deterministically and compare with the case $f=0$ covered in Section~\ref{sec:BM_weighted}. 
We briefly state the many-to-one and many-to-two lemmas in Section~\ref{sec:many-to-few-lemmas}, before turning to the proof of Theorem~\ref{thm:main}. 
In Section \ref{sec:upper_bound}, we show the upper bound in Theorem \ref{thm:main} relying essentially on first moment calculation.
The main difficulty is to localize the trajectory of a particle $u$ such that $R_u(t) \geq m(t)$, to be able to justify the approximations $\theta_u(s) \approx Y_u(s)/(\sqrt{2}s)$ and $R_u(t) \approx X_u(t)$ and then to add more precise barrier on the $X$-coordinate to determine the logarithmic corrections.
In Section \ref{sec:lower_bound}, we prove the lower bound in Theorem \ref{thm:main} by performing a first and second moment calculation on the number of particles such that $X_u(t) \geq m(t)$, while satisfying appropriate constraints on their trajectories.
Finally, in Section~\ref{sec:porism}, we prove Proposition~\ref{porism} and comment on the optimality of the exponents appearing in its statement.

\subsection{Open questions and remarks}
\label{subsec:open_questions}

\begin{enumerate}
\item In Theorem \ref{thm:main} we show tightness for $(M_t - m(t))_{t\geq 1}$. Can this be extended to convergence in distribution? 
For this, one needs a quantity replacing the limit of the derivative martingale in \eqref{eq:1d_maximum}.
The following process should work:
\[
W_t \coloneqq t^{-\alpha/[2(2+\alpha)]} 
\exp \left( \vartheta_1 t^{\frac{2-\alpha}{2+\alpha}} \right)
\sum_{u\in\cN(t)} (\sqrt{2}t-X_u(t)) e^{\sqrt{2}(X_u(t)-\sqrt{2} t)} 
\varphi_0 \left( \frac{\vartheta_2 Y_u(t)}{t^{\alpha/(2+\alpha)}} \right),
\]
where $\varphi_0$, $\vartheta_1$ and $\vartheta_2$ are introduced in Proposition \ref{prop:estimate_G}.
Note that the dependence in $X_u(t)$ is similar to the definition of the derivative martingale.
However, here $(W_t)_{t\geq 0}$ is not an exact martingale, but only satisfies $\E[W_t|\cF_s] = W_s + o(1)$ for $s$ and $t$ large.
We believe that $(W_t)_{t\geq 0}$ has a limit (in probability at least) and that it plays the role of $D_\infty$ in the convergence in distribution of $M_t-m(t)$.
To go further, we can also ask the question of the convergence of the extremal process.

\item In this work, we focus on the case $\alpha \in (2/3,2)$. We leave the other cases as open questions and briefly discuss here what behaviour we expect for $m(t)$.
\begin{enumerate}
\item[$(\alpha = 2)$]
This is a natural case as this corresponds most naturally to the assumption that $b$ is smooth at $\theta = 0$. Furthermore, as $\alpha \nearrow 2$ we have $\frac{2-\alpha}{2+\alpha} \to 0$, so we expect no polynomial correction. Here, one can show (for example using the methods of \cite{gao_laplace_2003}) that
\begin{equation} \label{eq:case_alpha=2}
    \E \left[\exp \left(- \beta \int_s^t \left( \frac{Y_r}{r} \right)^2 \diff r \right) \right] 
    \sim C(\beta) 
    \left( \frac{s}{t} \right)^{\frac{1}{4}(\sqrt{1+8\beta}-1) },
\end{equation}
as $t/s\to \infty$, $C(\beta)$ can be stated explicitly.
This strongly suggest that the appropriate correction for the maximum is
\begin{equation} \label{eq:conjecture_alpha=2}
    m(t)= \sqrt{2}t-\left(\frac{3}{2\sqrt{2}} + \frac{\sqrt{1+8\beta}-1}{4\sqrt{2}} \right)\log t.
\end{equation}
This guess relies on the fact that trajectories of extremal particles should satisfy $Y_t = O(\sqrt{t})$ and therefore $R_t = X_t + O(1)$. 
The coefficient of the logarithmic correction contains two terms: the first one comes from the fact that the $X$-trajectory has to stay below a straight line (as for the 1d BBM), whereas the second one comes from the constraint on the $Y$-trajectory and \eqref{eq:case_alpha=2}.
Note that this coefficient interpolates between $-3/(2\sqrt{2})$ as $\beta \to 0$ and $-\infty$ as $\beta \to \infty$.
\item[$(\alpha > 2)$] 
In this case, the constraint on the $Y$-trajectory of an extremal particle is weaker: if $\abs{Y_s}$ grows at most like $s^\gamma$ with $\gamma < 1-\frac{1}{\alpha}$, then the integral $\int_{s_0}^t \lvert \frac{Y_s}{s} \rvert^\alpha \diff s$ does not diverge and everything behaves as if the branching rate was equal to 1. 
Therefore, a good proxy for our model would be a 2d BBM with branching rate 1, but where only extremal particles with an angle $\theta_t = O(t^{-1/\alpha})$ are taken into account.
But heuristics coming from the study of 2d BBM \cite{Mal2015a} tell us that the behaviour of extremal particles in two disjoint cones of angle $t^{-1/2}$ is roughly independent. 
Altogether, this suggests that $m(t)$ should be the same as for the maximum of $t^{1/2-1/\alpha}$ independent 1d BBMs, which leads to
\begin{equation*}
     m(t) = \sqrt{2}t - \frac{1}{\sqrt{2}} \left( 1 + \frac{1}{\alpha} \right) \log t.
\end{equation*}
Note that the coefficient of the logarithmic correction interpolates between $-3/(2\sqrt{2})$ and $-1/\sqrt{2}$, that is between the 1d and the 2d BBM.
\item[$(\alpha \leq \frac{2}{3})$] 
Results in Sections \ref{sec:BM_weighted} and \ref{sec:probabilistic_arguments} concerning Brownian motion weighted by an integral are proved for any $\alpha \in (0,2)$, so in particular \eqref{eq:don} still holds for $\alpha \leq 2/3$.
Henceforth, we believe that (a weak version of) our methods should be enough to prove
\[
    M_t = \sqrt{2} t - \frac{\vartheta_1}{\sqrt{2}} t^{(2-\alpha)/(2+\alpha)} \left(1+o_\P(1)\right),
\]
where $o_\P(1)$ denotes a term tending to 0 in probability as $t \to \infty$.
However, getting the asymptotic expansion of $M_t$ up to $O(1)$ seems much harder and we do not have a precise conjecture for the right choice of $m(t)$ in that case.
More polynomial terms probably need to be added in the definition of $m(t)$, for several reasons listed here:
\begin{itemize}
    \item A simple reason already appears in the simplified heuristics that led to \eqref{eq:appp}. Indeed, solving \eqref{eq:appp} for $\alpha <2/3$ requires to add more and more polynomial terms to $m(t)$ as $\alpha$ decreases.  However, as explained in the following points, \eqref{eq:appp} is not the right choice for $m(t)$ anymore.
    \item Because $m(s)$ is roughly the position of extremal particles at time $s$, an extremal particle $u$ at time $t$ has to satisfy $X_u(s) \leq m(s)$ for any $s \leq t$. If $\alpha \in (2/3,2)$, the exponent $(2-\alpha)/(2+\alpha)$ is less than $1/2$ and staying below this barrier has the same cost as staying below a straight barrier with the same endpoints, that is a polynomial cost, which is responsible for the $-\frac{3}{2\sqrt{2}} \log t$  correction to $m(t)$.
    On the other hand, if $\alpha \in (0,2/3)$, this exponent becomes larger than $1/2$ and staying below this barrier has a stretched exponential cost, which results in a polynomial correction to $m(t)$.
    \item Lastly, the approximation $\theta_s \approx Y_s/X_s \approx Y_s/(\sqrt{2} s)$ used in the heuristics in Section \ref{sec:heuristics} is not good enough anymore. Indeed, we expect that $X_s - \sqrt{2} s$ is of order $s^{(2-\alpha)/(2+\alpha)}$, and therefore the resulting error term in the integral in \eqref{eq:don} is not negligible when $\alpha \leq 2/3$: instead, it should grow polynomially fast and therefore result in another polynomial correction to $m(t)$.
\end{itemize}
\end{enumerate}

\item We finally mention two models that can be seen as a limit as $\alpha  \to 0$ of our model.
One of them would be a 2d BBM in which a particle at $(x,y)$ branches at rate $b(y)$, where $b$ is continuous, $b(0)=1$ and $b(y) \to 0$ as $\lvert y \rvert \to \infty$. 
In that case, we expect that the linear order of $m(t)$ would depend on $b$ and would be smaller than $\sqrt{2}$, but there would be no polynomial correction, only a logarithmic one. We believe that a lot of the analysis would rely on understanding BBM in a strip $\R \times [-L,L] $ using the results of \cite{harris_branching_2016}. Another interesting model might be a catalytic 2d BBM, where a particle branches at rate $\beta + \beta_0 L_t^x$, where $L_t^x$ is the local time of the second coordinate of the particle at $0$.



\end{enumerate}

\subsection{Notation}

Throughout the paper, $C$ and $c$ denote positive constants that can change between occurrences. They can depend on some other parameters in a way which is made clear in the statement of each result; in the proof of such a result allowed dependencies are the same as in the statement.
We use standard Landau notation: for $f \colon [0,\infty) \to \R$ and $g \colon [0,\infty) \to (0,\infty)$, we say, as $t \to \infty$, 
that $f(t) = o(g(t))$ if $\lim_{t \to \infty} f(t)/g(t) = 0$, 
that $f(t) = O(g(t))$ if $\limsup_{t \to \infty} \abs{f(t)}/g(t) < \infty$, 
and that $f(t) \sim g(t)$ if $\lim_{t \to \infty} f(t)/g(t) =1$.

It is sometimes convenient to work under probability measures other than $\P$.
We denote by $\P_{(s,x,y)}$ the probability under which our inhomogeneous BBM starts at time $s$ from a single particle located at $(x,y)$. Moreover, under $\P_{(s,x,y)}$, $(X_t,Y_t)_{t\geq 0}$ denotes a 2d Brownian motion starting at time $s$ from $(x,y)$. If $s=0$, we drop the first index and simply write $\P_{(x,y)}$. If the event considered involves only $X$, then we keep only the starting point of $X$ as index; and similarly for $Y$.

\section{Brownian motion weighted by an integral via PDEs}
\label{sec:BM_weighted}

We define the kernel $G$ for $x,y \in \R$ and $0 \leq s < t$ by
\begin{equation} \label{eq:def_G}
G(s,x;t,y) = \frac{1}{\sqrt{2\pi(t-s)}} \exp \left( - \frac{(y-x)^2}{2(t-s)} \right)
\E_{(s,x)} \left[ \exp \left( - \beta \int_{s}^t \abs{\frac{B_r}{\sqrt{2}r}}^\alpha \diff r \right) \middle| B_t = y \right],
\end{equation}
where, under $\P_{(s,x)}$, $B$ is a Brownian motion starting from $x$ at time $s$.
Note that under $\P_{(s,x)}$ and given $B_t = y$, $(B_r)_{r\in[s,t]}$ is a Brownian bridge from $(s,x)$ to $(t,y)$.
The kernel $G$ satisfies, for any measurable function $f \colon \R \to \R_+$,
\begin{equation} \label{eq:int_G}
\int_\R f(y) G(s,x;t,y) \diff y 
= \E_{(s,x)} \left[ \exp \left( - \beta \int_{s}^t \abs{\frac{B_r}{\sqrt{2}r}}^\alpha \diff r \right) f(B_t) \right].
\end{equation}
The goal of this section is to prove the following result, namely a good approximation for $G$. Its proof is technical and done by PDE methods using first a connection via the Feynman--Kac formula. The reader, especially if from a probabilistic audience, might want to skip the remainder of this section (which is self-contained) on a first read of this paper. 

\begin{prop} \label{prop:estimate_G} 
	Let $\alpha \in (0,2)$, $\kappa \coloneqq 2\alpha/(2+\alpha) \in (0,1)$ and $\beta > 0$.
    There exist $K,C,c > 0$ such that, for any $x,y \in \R$, $s \geq K$ and $t \geq s + Ks^\kappa$, we have
    \begin{align*}
    & \abs{(st)^{\kappa/4} \exp\left(\vartheta_1 (t^{1-\kappa}-s^{1-\kappa})\right) G(s,x;t,y) 
    - \vartheta_2
	\varphi_0 \left( \frac{\vartheta_2 x}{s^{\kappa/2}} \right) 
	\varphi_0 \left( \frac{\vartheta_2 y}{t^{\kappa/2}} \right)}
    \nonumber \\
    & \leq C \left( s^{\kappa-1} + e^{-c (t-s)/t^\kappa} \right) \exp\left(-c \left( \frac{\abs{x}}{s^{\kappa/2}} \right)^{\alpha} \right)
    \exp\left(-c \left( \frac{\abs{y}}{t^{\kappa/2}} \right)^{\alpha} \right)
    \end{align*}
    where $\varphi_0$ is defined in Section~\ref{sec:sturm-liouville} and Proposition~\ref{prop:sturm-liouville} together with the constant $\lambda_0 > 0$ (and they only depend on $\alpha$) and we set $\vartheta_1 \coloneqq \lambda_0 \beta^{2/(2+\alpha)} 2^{-2\alpha/(2+\alpha)}/(1-\kappa)$ and $\vartheta_2 \coloneqq (2\beta)^{1/(2+\alpha)}$.
\end{prop}

As will be seen later in Proposition~\ref{prop:sturm-liouville}, $\varphi_0$ is an even positive function such that $\norme{\varphi_0}_2 = 1$.
See Figure~\ref{fig:phi0} for an illustration.

By integrating these bounds in the $y$ coordinate and bounding them uniformly in the $x$ coordinate we obtain the following corollary.

\begin{cor}\label{cor:estimate_G}
    Let $\alpha \in (0,2)$ and $\beta > 0$.
    There exist $K,C,c > 0$ such that the following holds for any $s \geq K$ and $t \geq s + Ks^\kappa$.
    \begin{enumerate}
    	\item\label{it:UB_integrated_G} For any $x \in \R$,
    	\begin{equation*}
    	\int_\R G(s,x;t,y) \diff y
    	\leq C (t/s)^{\kappa/4} \exp\left(\vartheta_1 (s^{1-\kappa}-t^{1-\kappa})\right).
    	\end{equation*}
    	\item\label{it:LB_integrated_G} For any $\abs{x} \leq s^{\kappa/2}$,
    	\begin{equation*}
    	\int_{\abs{y} \leq t^{\kappa/2}} G(s,x;t,y) \diff y
    	\geq c (t/s)^{\kappa/4} \exp\left(\vartheta_1 (s^{1-\kappa}-t^{1-\kappa})\right) .
    	\end{equation*}
    \end{enumerate}
\end{cor}

\begin{figure}[bth]
    \centering
    \includegraphics[width=0.7\linewidth]{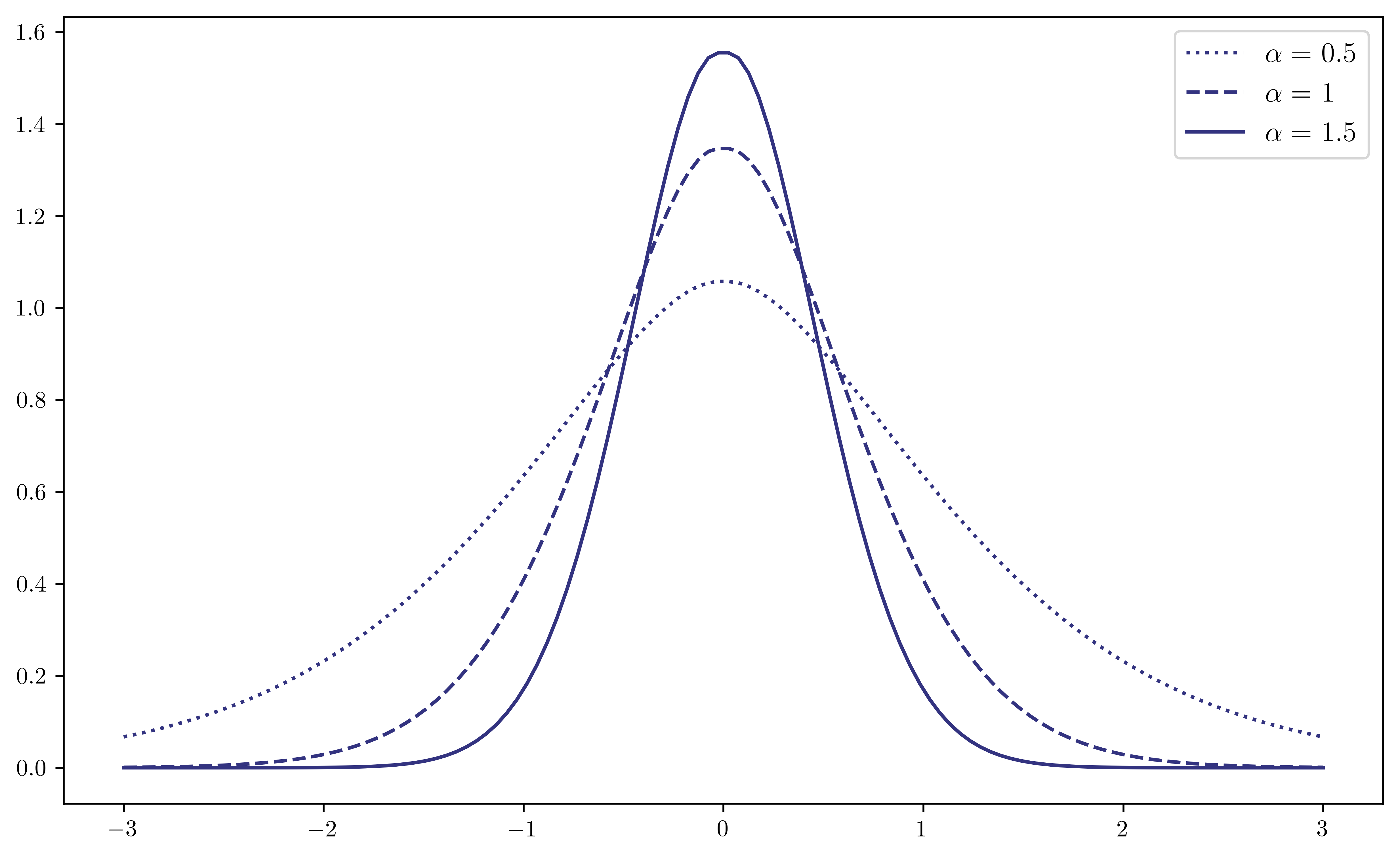}
    \caption{A plot of $\varphi_0$ for different values of $\alpha$.}
    \label{fig:phi0}
\end{figure}

\subsection{Connection with a PDE via Feynman--Kac formula} 

In this section, we first work in a slightly more general context and consider the following PDE on $[0,T] \times \R$, for some $T,\sigma>0$ and some function $k \colon [0,T] \times \R \to \R$,
\begin{equation} \label{eq:PDE_k}
	\partial_t u = \frac{\sigma^2}{2} \partial_{xx}^2 u + ku.
\end{equation}
We write $\cC^{1,2}(I \times \R)$ for the space of functions $u$ on $I\times \R$ such that $\partial_t u$, $\partial_x u$ and $\partial_{xx}^2 u$ exist and are continuous on $I\times \R$.
	A \textit{fundamental solution} $\Gamma$ of \eqref{eq:PDE_k} is a function of $(\tau,\xi;t,x)$ defined for any $0 \leq \tau < t \leq T$ and any $\xi,x\in\R$ such that, for any $\tau \in [0,T)$,
\begin{itemize}
	\item for any $\xi \in \R$, the function $\Gamma(\tau,\xi;\cdot,\cdot)$ is in $\cC^{1,2}((\tau,T] \times \R)$ and satisfies \eqref{eq:PDE_k} on $(\tau,T] \times \R$;
	\item for any $f \colon \R \to \R$ continuous with compact support and any $x_0 \in \R$,
	\begin{equation} \label{def:fund_sol}
		\lim_{(t,x) \to (\tau^+,x_0)} \int_\R \Gamma(\tau,\xi;t,x) f(\xi) \diff \xi = f(x_0).
	\end{equation}
\end{itemize}
The following result shows existence of a fundamental solution under appropriate conditions on the function $k$ and expresses it in terms of a Brownian motion.

\begin{lem} \label{lem:fundamental_solution}
	Assume $k \leq 0$ and $k$ is Hölder continuous on any compact subset of $[0,T] \times \R$.
	For any $0 \leq \tau < t \leq T$ and any $\xi,x\in\R$, let
	\begin{align*}
	\Gamma(\tau,\xi;t,x)
	& \coloneqq \frac{1}{\sqrt{2\pi \sigma^2 (t-\tau)}} 
	\exp \left( - \frac{(\xi-x)^2}{2 \sigma^2 (t-\tau)} \right) \\
        & \hspace{2cm}\times 
	\E_{(\tau,\frac{x}{\sigma})} \left[ \exp \left( \int_\tau^t k(t+\tau-r,\sigma B_r) \diff r \right) \middle| \sigma B_t = \xi \right] \\
	& = \frac{1}{\sqrt{2\pi \sigma^2 (t-\tau)}} 
	\exp \left( - \frac{(\xi-x)^2}{2 \sigma^2 (t-\tau)} \right)\\
    & \hspace{2cm}\times 
	\E_{(\tau,\frac{\xi}{\sigma})} \left[ \exp \left( \int_\tau^t k(r,\sigma B_r) \diff r \right) \middle| \sigma B_t = x \right].
	\end{align*}
	Then, $\Gamma$ is a fundamental solution of the PDE \eqref{eq:PDE_k}.
\end{lem}

\begin{proof}
	We first mention that the fact that the two expressions for $\Gamma(\tau,\xi;t,x)$ are equal follows from the fact that if $(P_s)_{s\in[\tau,t]}$ is a Brownian bridge from $(\tau,x/\sigma)$ to $(t,\xi/\sigma)$, then $(P_{t+\tau-s})_{s\in[\tau,t]}$ is a Brownian bridge from $(\tau,\xi/\sigma)$ to $(t,x/\sigma)$.
	
	Let $\tau \in [0,T)$ and $f \colon \R \to \R$ be continuous and bounded.
	Then, by \cite[Theorem III]{AroBes1967} applied with $\lambda = 0$ and $\alpha > 0$ small enough%
	\footnote{First note that their assumptions stated in \cite[Eq.\@ (2.1)-(2.3)]{AroBes1967} are satisfied with $\nu = k_1 = \sigma^2/2$, $\lambda = 0$ and any $k_2,k_3>0$. 
	Then, existence of $u$ is proved up to time $T_\alpha = \min(T,1/(2\beta(\alpha)))$ (see before Eq.\@ (2.4) there) and $\beta(\alpha)$ can be made arbitrarily large by choosing $\alpha,k_2,k_3$ small enough, so that we get $T_\alpha = T$.},
	there exists a solution $u \in \cC^0([\tau,T] \times \R) \cap \cC^{1,2}((\tau,T] \times \R)$ to the PDE \eqref{eq:PDE_k} with initial condition $u(\tau,x) = f(x)$ for $x \in \R$, and there exists $C>0$ such that
	\[
	\forall (t,x) \in [0,T] \times \R, \quad
	\abs{u(t,x)} \leq C \left( \abs{x}^2 + 1 \right).
	\]
	Therefore, we can apply Feynman--Kac formula \cite[Theorem 7.6]{KarShr1991} to $v(t,x) = u(T+\tau-t,x)$ to get, for any $t \in [\tau,T]$ and $x \in \R$,
	\begin{align}
	u(t,x) 
	& = \E_{(\tau,x/\sigma)} \left[ \exp \left( \int_\tau^t k(t+\tau-r,\sigma B_r) \diff r \right) f(\sigma B_t) \right] \nonumber  \\
	& = \int_\R \Gamma(\tau,\xi;t,x) f(\xi) \diff \xi, \label{eq:u_in_terms_of_Gamma}
	\end{align}
	by conditioning on $\sigma B_t = \xi$ and using the first expression for $\Gamma(\tau,\xi;t,x)$.
	This, together with $u \in \cC^0([\tau,T] \times \R))$ and $u(\tau,\cdot) = f$, proves that $\Gamma$ satisfies \eqref{def:fund_sol}.
	
	Now, fix some $(\tau,\xi) \in [0,T) \times \R$. It remains to check that $\Gamma(\tau,\xi;\cdot,\cdot)$ is in $\cC^{1,2}((\tau,T] \times \R)$ and satisfies \eqref{eq:PDE_k} on $(\tau,T] \times \R$.
	For that consider $s \in (\tau,t)$. As a consequence of Markov's property at time $s$ together with the second expression for $\Gamma(\tau,\xi;t,x)$, we have
	\[
	\Gamma(\tau,\xi;t,x) 
	= \int_\R \Gamma(\tau,\xi;s,y) \Gamma(s,y;t,x) \diff y,
	\]
	But, the function $y \mapsto \Gamma(\tau,\xi;s,y)$ is bounded and continuous: this is shown using $k\leq 0$, continuity of $k$ and the fact that if $(P_r)_{r\in[\tau,s]}$ is a Brownian bridge from $(\tau,\xi/\sigma)$ to $(s,0)$, then $(P_r+y(r-\tau)/(s-\tau))_{r\in[\tau,s]}$ is a Brownian bridge from $(\tau,\xi/\sigma)$ to $(s,y)$.
	Therefore, by \eqref{eq:u_in_terms_of_Gamma} and the discussion before, $\Gamma(\tau,\xi;\cdot,\cdot)$ is in $\cC^{1,2}((s,T] \times \R)$ and satisfies \eqref{eq:PDE_k} on $(s,T] \times \R$. Since this is true for any $s \in (\tau,T)$, this concludes the proof.
\end{proof}

Using this connection, the proof of Proposition \ref{prop:estimate_G} boils down to the study of the following PDE on $[0,1) \times \R$
\begin{equation} \label{eq:main_PDE}
\partial_t u = \varrho \left( \partial_{xx}^2 u - (1-t)^{-\alpha} \abs{x}^{\alpha} u \right),
\end{equation}
where $\varrho > 0$ is a parameter meant to be large.
More precisely, let $g$ denote the fundamental solution of \eqref{eq:main_PDE} given by Lemma \ref{lem:fundamental_solution} (it is defined consistently on $[0,T]$ for any $T < 1$).
The remainder of Section \ref{sec:BM_weighted} is dedicated to the proof of the following result.

\begin{prop} \label{prop:estimate_g} 
	Let $\alpha \in (0,2)$ and $\kappa \coloneqq 2\alpha/(2+\alpha) \in (0,1)$.
	There exist $C,c > 0$ such that, for any $\varrho \geq 40$, $x,\xi \in \R$, $T \in (0,1)$ and $\delta>0$ satisfying
	\[
	T \geq \frac{20}{\varrho}, \qquad 
	\varrho (1-T)^{1-\kappa} \geq 10, \qquad 
	\frac{1}{\varrho (1-T)^{1-\kappa}} \leq \delta \leq \frac{1}{10} \wedge \frac{T^2\varrho}{100},
	\]
	we have
	\begin{align*}
	& \abs{\exp \left( \lambda_0 \varrho \int_0^T \frac{\diff s}{(1-s)^{\kappa}} \right) g(0,\xi;T,x) 
		- \varphi_{0}(\xi) \cdot (1-T)^{-\kappa/4} \varphi_{0}\left( (1-T)^{-\kappa/2} x \right)} \\
	& \leq C \left( \delta + e^{-c\varrho T} \right)
	(1-T)^{-\kappa/4}
	\left( e^{-c \abs{\xi}^{(2+\alpha)/2}}
	+ e^{-c (\varrho \delta)^{1/2} (1+\abs{\xi}^{\alpha})} \right) \\
	&  {} \times
	\left( 
	\exp \left( -c \left( \frac{\abs{x}}{(1-T)^{\kappa/2}} \right)^{\frac{2+\alpha}{2}} \right)
	+ \exp\left(-c (\varrho(1-T)^{1-\kappa} \delta)^{1/2}
	\left( 1 + \left( \frac{\abs{x}}{(1-T)^{\kappa/2}} \right)^{\alpha} \right) \right)
	\right),
	\end{align*}
	where $\lambda_0$, $\varphi_0$ are defined in Proposition \ref{prop:sturm-liouville}.
\end{prop}

\begin{rem}
	If the constraints on $\varrho$ and $T$ are satisfied, then the condition on $\delta$ is non empty, that is one can check that 
	\[
	\frac{1}{\varrho (1-T)^{1-\kappa}} \leq \frac{1}{10} \wedge \frac{T^2\varrho}{100}.
	\]
	The fact that the left-hand side is at most $1/10$ is direct. To prove that it is less than $T^2\varrho/100$, one can distinguish two cases: if $(1-T)^{1-\kappa} > 1/2$, it follows from $T \geq 20/\varrho$ and, in the opposite case, one has necessarily $T\geq 1/2$ and so it follows from $\varrho (1-T)^{1-\kappa} \geq 10$ and $\varrho \geq 40$.
\end{rem}

As a consequence of this remark, it is always allowed to take $\delta = 1/(\varrho (1-T)^{1-\kappa})$ in the previous result, if the other assumptions on $\varrho$ and $T$ are satisfied. This gives the following simpler bound, where we do not try to optimize the tails in $\xi$ and $x$.

\begin{cor} \label{cor:estimate_g} 
	Let $\alpha \in (0,2)$ and $\kappa \coloneqq 2\alpha/(2+\alpha) \in (0,1)$.
	There exist $C,c > 0$ such that, for any $\varrho \geq 40$, $x,\xi \in \R$, $T \in (0,1)$ and $\delta>0$ satisfying $T \geq 20/ \varrho$ and $\varrho (1-T)^{1-\kappa} \geq 10$,
	we have
	\begin{align*}
	& \abs{\exp \left( \lambda_0 \varrho \int_0^T \frac{\diff s}{(1-s)^{\kappa}} \right) g(0,\xi;T,x) 
		- \varphi_{0}(\xi) \cdot (1-T)^{-\kappa/4} \varphi_{0}\left( (1-T)^{-\kappa/2} x \right)} \\
	& \leq C \left( \frac{1}{\varrho (1-T)^{1-\kappa}} + e^{-c\varrho T} \right)
	(1-T)^{-\kappa/4}
	e^{-c \abs{\xi}^\alpha} 
	\exp \left( -c \left( \frac{\abs{x}}{(1-T)^{\kappa/2}} \right)^\alpha \right).
	\end{align*}
\end{cor}

Next, we explain how to deduce Proposition~\ref{prop:estimate_G} from the corollary above.
\begin{proof}[Proof of Proposition \ref{prop:estimate_G}]
	By definition of $G$ in \eqref{eq:def_G} and the scaling property of Brownian motion, one has, for any $0 < s < t$ and $x,y \in \R$, 
	\begin{align*}
	 \sqrt{t} & G\left(s,x\sqrt{t};t,y\sqrt{t}\right) \\
	& = \frac{\sqrt{t}}{\sqrt{2\pi(t-s)}} \exp \left( - \frac{(y-x)^2t}{2(t-s)} \right)
	\E_{(s/t,x)} \left[ \exp \left( - \beta t^{1-\alpha/2} \int_{s/t}^1 \abs{\frac{B_r}{\sqrt{2}r}}^\alpha \diff r \right) \middle| B_1 = y \right] \\
	& = \frac{1}{\sqrt{2\pi(1-\tau)}} \exp \left( - \frac{(y-x)^2}{2(1-\tau)} \right)
	\E_{(\tau,x)} \left[ \exp \left(- \varrho \int_\tau^1 \abs{\frac{\sqrt{2\varrho} B_r}{r}}^\alpha \diff r \right) \middle| B_1 = y \right].
	\end{align*}
	setting $\tau = s/t$ and $\varrho$ such that $\beta t^{1-\alpha/2}2^{-\alpha/2} = \varrho (2\varrho)^{\alpha/2}$, i.e.\ $\varrho = \beta^{2/(2+\alpha)} 2^{-2\alpha/(2+\alpha)} t^{1-\kappa}$. 
	On the other hand, using the first expression for $g$ given by Lemma \ref{lem:fundamental_solution} with $\sigma = \sqrt{2\varrho}$, we have
	\begin{align*}
	& g(0,y;1-\tau,x) \\
	& = \frac{1}{\sqrt{2\pi \sigma^2 (1-\tau)}} 
	\exp \left( - \frac{(y-x)^2}{2 \sigma^2 (1-\tau)} \right)
	\E_{(0,x/\sigma)} \left[ \exp \left( - \varrho \int_0^{1-\tau} \abs{\frac{\sigma B_r}{\tau+r}}^\alpha \diff r \right) \middle| \sigma B_{1-\tau} = y \right].
	\end{align*}
	Doing a time shift by $\tau$ for the Brownian motion in the last expectation, we obtain the following relation
	\begin{equation} \label{eq:relation_G_g}
	G(s,x;t,y) 
	= \sqrt{\frac{2\varrho}{t}} g\left( 0,\sqrt{\frac{2\varrho}{t}} y; 1-\frac{s}{t}, \sqrt{\frac{2\varrho}{t}} x \right).
	\end{equation}
	Then, the result follows from Corollary \ref{cor:estimate_g} with $\varrho$ given above and $T = 1-\frac{s}{t}$.
\end{proof}

\subsection{Associated Sturm--Liouville operators}
\label{sec:sturm-liouville}

Fix $\alpha > 0$.
An important tool in the study of the PDE \eqref{eq:main_PDE} is the following family of Sturm--Liouville operators:
for any $q>0$, let $\cL_q$ be defined by
\[
(\cL_q f)(x) \coloneqq -f''(x) + q\abs{x}^\alpha f(x), \quad x\in \R,
\]
for any $f$ in the domain
\[
	D(\cL_q) \coloneqq \{ f \in L^2(\R) : f,f' \text{ absolutely continuous on } \R, \cL_q f \in L^2(\R) \}.
\]
These operators are densely defined on $L^2(\R)$ \cite[Lemma 9.4]{Tes2014} and self-adjoint \cite[Theorem 9.6]{Tes2014}.
In the case $q=1$, we simply write $\cL = \cL_1$.

\begin{prop} \label{prop:sturm-liouville}
	Let $\alpha > 0$.
	\begin{enumerate}
		\item\label{it:cL_eigenbasis} The spectrum of $\cL$ is discrete and consists of the eigenvalues $0<\lambda_0 < \lambda_1 < \lambda_2 < \cdots$ which all have multiplicity 1. 
		Let $\varphi_n$ be an eigenfunction associated to $\lambda_n$, then it has exactly $n$ zeroes, so it can be chosen uniquely such that $\norme{\varphi_n}_2 = 1$ and $\varphi_n(x) > 0$ for any $x$ greater than its largest zero.
		Moreover, $(\varphi_n)_{n\geq 0}$ is an orthonormal basis of $L^2(\R)$. 
		\item\label{it:link_cL_cL_q} Let $q > 0$. The previous point also holds for $\cL_q$ with eigenvalues $(\lambda_{q,n})_{n\geq 0}$ and eigenfunctions $(\varphi_{q,n})_{n\geq 0}$ such that, for any $n \geq 0$ and $x \in \R$,
	\begin{equation}\label{eq:eigenvalues_Lgamma}
		\lambda_{q,n} = q^{2/(2+\alpha)} \lambda_n
		\qquad \text{and} \qquad
		\varphi_{q,n}(x) = q^{1/[2(2+\alpha)]} \varphi_n(q^{1/(2+\alpha)} x).
	\end{equation}
		\item\label{it:asymp_lambda_n} Let $c_\alpha \coloneqq \frac{2}{\pi} \int_0^1 \sqrt{1-u^\alpha} \diff u$. Then, as $n \to \infty$,
		\begin{equation} \label{eq:asymp_lambda_n}
		\lambda_n \sim \left( \frac{n}{c_\alpha} \right)^{2\alpha/(\alpha+2)}.
		\end{equation}
		\item\label{it:parity} For any $n \geq 0$, the function $\varphi_n$ has the same parity as $n$.
		\item\label{it:tail} There exists $C >0$ depending on $\alpha$ such that, for any $n \geq 0$ and $x \in \R$, 
		\[
		\abs{\varphi_n(x)}
		\leq C (n+1)^3 \left[ 1 \wedge \exp \left( - \frac{1}{2+\alpha} \left( \abs{x}^{(2+\alpha)/2} -  Cn \right) \right) \right].
		\]
		\item\label{it:tail_derivative} There exist $C,c>0$ depending on $\alpha$ such that, for any $x \in \R$, 
		$\abs{\varphi_0'(x)}
		\leq C e^{-c \abs{x}^{(2+\alpha)/2}}$.
	\end{enumerate}
\end{prop}

\begin{proof}
	Part \ref{it:cL_eigenbasis}.
	The fact that the essential spectrum of $\cL$ is empty follows from \cite[Proposition 4.5.4]{BenBroWei2020} together with the fact that any self-adjoint realization of $-f'' + \abs{x}^\alpha f$ in $L^2((0,\infty))$ or $L^2((-\infty,0))$ has empty essential spectrum by \cite[Theorem 4.5.8]{BenBroWei2020}. 
	Then, it follows from \cite[Theorem 6.1.9]{BenBroWei2020} that the discrete spectrum is bounded from below so it can be written $\lambda_0 < \lambda_1 < \lambda_2 < \cdots$. Moreover, eigenvalues have multiplicity 1 (this is more clearly stated in \cite[Theorem 10.12.1.(8).(ii)]{Zet2005}) and an eigenfunction $\varphi_n$ associated with $\lambda_n$ has $n$ zeroes. 
	We now chose $\varphi_n$ uniquely as in the statement.
	The fact that $\lambda_0 > 0$ is obtained by contradiction: if $\lambda_0 \leq 0$, then $\varphi_0''(x) = (\abs{x}^\alpha-\lambda_0) \varphi_0(x) \geq 0$ for any $x \in \R$ (using $\varphi_0 > 0$ because it has no zero), so $\varphi_0$ is convex in $\R$, which contradicts $\varphi_0 \in L^2$. 
	Finally, $(\varphi_n)_{n\geq 0}$ is an orthonormal basis of $L^2(\R)$ by \cite[Corollary 4.2.3]{Tes2014}.
	
	Part \ref{it:link_cL_cL_q}. This follows from a direct verification.
	
	Part \ref{it:asymp_lambda_n}. This is a consequence of \cite[Theorem 6.1.12]{BenBroWei2020}, see \cite[Example 6.1.13, Part 5]{BenBroWei2020}.
	
	Part \ref{it:parity}. Let $n \geq 0$. Since $\cL$ preserves parity, the even and the odd parts of $\varphi_n$ are also eigenfunctions associated to $\lambda_n$, so one of them must be zero because $\lambda_n$ has multiplicity 1. So $\varphi_n$ is either even or odd. 
	If $\varphi_n$ is odd, then it has an odd number of zeroes.
	If $\varphi_n$ is even, then $\varphi_n'(0) = 0$ and therefore $\varphi_n(0) \neq 0$ (otherwise $\varphi_n=0$), 
	so $\varphi_n$ has an even number of zeroes.
	But, $\varphi_n$ has $n$ zeroes by Part \ref{it:cL_eigenbasis} so $\varphi_n$ has the same parity as $n$.
	
	Part \ref{it:tail}. This part relies on some preliminary results established in Lemma \ref{lem:sturm-liouville} below. Let $n \geq 0$ and $x \geq 0$ (by parity the result for $x \leq 0$ follows).
	Let $b \coloneqq \alpha \vee (2\lambda_n)^{1/\alpha}$.
	We consider the function $\psi_n$ defined on $[b,\infty)$ by $\cL \psi_n = \lambda_n \psi_n$ together with the initial conditions 
	$\psi_n(b) = 1$ and 
	$\psi_n'(b) = \frac{1}{2} b^{\alpha/2}$.
	Then, the Wronskian 
	\[
	W \coloneqq \varphi_n \psi_n' - \varphi_n' \psi_n,
	\]
	is a constant.
	Moreover, by Lemma \ref{lem:sturm-liouville}.\ref{it:convexity}, $\varphi_n > 0$ and $\varphi_n' < 0$ in $[b,\infty)$ and it follows from a similar argument that $\psi_n$ is convex in $[b,\infty)$ and therefore $\psi_n' > 0$ and $\psi_n \geq 0$.
	Therefore, for any $x \geq b$, 
	\begin{equation} \label{eq:tail1}
	\abs{\varphi_n(x)} 
	= \frac{W + \varphi_n'(x) \psi_n(x)}{\psi_n'(x)}
	\leq \frac{W}{\psi_n'(x)}.
	\end{equation}
	We first bound $W$ by taking its value at $b$: we have $W = \frac{1}{2} b^{\alpha/2} \varphi_n(b) - \varphi_n'(b)$.
	Moreover, using Lemma \ref{lem:sturm-liouville}.\ref{it:convexity},
	\[
	\abs{\varphi_n'(b)} 
	\leq \abs{\varphi_n'(\lambda_n^{1/\alpha})} 
	\leq \int_{y_0}^{\lambda_n^{1/\alpha}} \abs{\varphi_n''(x)} \diff x,
	\]
	where $y_0$ is a zero of $\varphi_n'$ in $[0,\lambda_n^{1/\alpha}]$ (one can see it exists using that either $\varphi_n'(0) = 0$ or $\varphi_n(0) = 0$ by Part \ref{it:parity}).
	Then, using $\varphi_n''(x) = (\abs{x}^\alpha-\lambda_n) \varphi_n(x)$, we get
	\[
	\abs{\varphi_n'(b)} 
	\leq \int_{y_0}^{\lambda_n^{1/\alpha}} \lambda_n \abs{\varphi_n(x)} \diff x
	\leq C \lambda_n^{(\alpha+1)/\alpha} (n+1)^{\alpha/[2(2+\alpha)]},
	\]
	applying Lemma \ref{lem:sturm-liouville}.\ref{it:norm_infinity}. 
	On the other hand, $b^{\alpha/2} \varphi_n(b) \leq C \lambda_n^{1/2} (n+1)^{\alpha/[2(2+\alpha)]}$ by Lemma \ref{lem:sturm-liouville}.\ref{it:norm_infinity} again.
	Hence, we get
	\begin{equation} \label{eq:tail2}
	W \leq C \lambda_n^{(\alpha+1)/\alpha} (n+1)^{\alpha/[2(2+\alpha)]} 
	\leq C (n+1)^3,
	\end{equation}
	using \eqref{eq:asymp_lambda_n}.
	Now, we aim at proving a lower bound for $\psi_n'$.
	For this, we introduce the following function
	\[
	g(x) \coloneqq \exp \left( \frac{1}{2+\alpha} \left( x^{(2+\alpha)/2} - b^{(2+\alpha)/2} \right) \right), \qquad x \geq b.
	\]
	Then, $g'(x) = \frac{1}{2} x^{\alpha/2} g(x)$
	and $g''(x) = (\frac{1}{4} x^{\alpha} + \frac{\alpha}{4} x^{(\alpha-2)/2}) g(x)$.
	In particular, for any $x \geq b$, we have $g''(x) \leq \frac{1}{2} x^{\alpha} g(x) \leq (x^\alpha - \lambda_n) g(x)$, where the first inequality uses $b \geq \alpha$ and the second one $b \geq (2\lambda_n)^{1/\alpha}$.
	On the other hand, $g(b) = \psi_n(b)$ and $g'(b) = \psi_n'(b)$.
	Hence, we get $g \leq \psi_n$ and $g' \leq \psi_n'$ on $[b,\infty)$. Combining this with \eqref{eq:tail1} and \eqref{eq:tail2}, we get, for any $x \geq b$,
	\[
	\abs{\varphi_n(x)} \leq C(n+1)^3 x^{-\alpha/2}  \exp \left( - \frac{1}{2+\alpha} \left( x^{(2+\alpha)/2} - b^{(2+\alpha)/2} \right) \right).
	\]
	The result follows by noting that $b^{(2+\alpha)/2} \leq C n$ by \eqref{eq:asymp_lambda_n} and using that we already now that $\norme{\varphi_n}_\infty \leq C (n+1)^3$ by Lemma \ref{lem:sturm-liouville}.\ref{it:norm_infinity}.
	
	Part \ref{it:tail_derivative}. By Part \ref{it:tail}, we have $\abs{\varphi_0(x)}
	\leq C e^{-c\abs{x}^{(2+\alpha)/2}}$ thus $\abs{\varphi_0''(x)}
	\leq C \abs{x}^\alpha e^{-c\abs{x}^{(2+\alpha)/2}}$ using $\lambda_0 >0$.
	Then, for $x \geq 0$, we have $\abs{\varphi'(x)} \leq \int_x^\infty \abs{\varphi_0''(y)} \diff y$ by Lemma \ref{lem:sturm-liouville}.\ref{it:tend_to_zero} and the result follows.
\end{proof}

\begin{lem} \label{lem:sturm-liouville}
	\begin{enumerate}
		\item\label{it:convexity} For any $n \geq 0$, $\varphi_n$ is convex and decreasing in $[\lambda_n^{1/\alpha},\infty)$;
		\item\label{it:tend_to_zero} For any $n \geq 0$, $\varphi_n$ and $\varphi_n'$ tend to zero at infinity;
		\item\label{it:norm_infinity} There exists $C = C(\alpha) >0$ such that, for any $n \geq 0$, $\norme{\varphi_n}_\infty \leq C (n+1)^{\alpha/[2(2+\alpha)]}$.
	\end{enumerate}
\end{lem}

\begin{proof}
	Part \ref{it:convexity}. Note that in $[\lambda_n^{1/\alpha},\infty)$, $\varphi_n''$ and $\varphi_n$ have the same sign (considering zero has both signs). 
	Therefore, $\abs{\varphi_n}$ is convex in $[\lambda_n^{1/\alpha},\infty)$. 
	By contradiction, assume $\varphi_n$ has a zero $x_0$ in $[\lambda_n^{1/\alpha},\infty)$. Then, $\varphi_n'(x_0) \neq 0$ (otherwise $\varphi_n = 0$). 
	Then, convexity of $\abs{\varphi_n}$ implies that $\abs{\varphi_n} \to \infty$, which contradicts $\varphi_n \in L^2$. 
	Therefore, $\varphi_n$ has no zero in $[\lambda_n^{1/\alpha},\infty)$ and hence it is positive (by our choice of $\varphi_n$) and so it is convex. 
	If $\varphi_n'$ takes a nonnegative value at $x_1 \in [\lambda_n^{1/\alpha},\infty)$, then $\varphi_n(x) \geq \varphi_n(x_1)$ for all $x \geq x_1$ and this contradicts $\varphi_n \in L^2$. 
	So $\varphi_n' < 0$ and $\varphi_n$ is decreasing in $[\lambda_n^{1/\alpha},\infty)$.
	
	Part \ref{it:tend_to_zero}. This is a consequence of Part \ref{it:convexity}, using again that $\varphi_n \in L^2$.
	
	Part \ref{it:norm_infinity}.
	The idea is to argue that, if $\varphi_n$ takes a large value at some point, then it creates some mass proportional to this value to its $L^2$ norm, which is fixed to 1, hence providing an upper bound for this value. 
	Since $\varphi_n$ is continuous and vanishes at infinity by Part~\ref{it:tend_to_zero}, there exists $y \in \R$ such that $\abs{\varphi_n(y)} = \norme{\varphi_n}_\infty$ and we can assume $y \geq 0$ by parity of $\varphi_n$. Then, $\varphi_n'(y) = 0$ so, for $t \geq 0$,
	\[
	\abs{\varphi_n(y+t)} 
	\geq \abs{\varphi_n(y)} - \frac{t^2}{2} \sup_{x \in [y,y+t]} \abs{\varphi_n''(x)}.
	\]
	Then, on the one hand, we have $y \leq \lambda_n^{1/\alpha}$ as a consequence of Part \ref{it:convexity}. 
	On the other hand, for $x \in [0,(2\lambda_n)^{1/\alpha}]$, $\abs{\varphi''(x)} \leq \abs{x^\alpha-\lambda_n} \cdot \abs{\varphi_n(x)} \leq \lambda_n \norme{\varphi_n}_\infty$.
	Therefore, for any $t \in [0,\lambda_n^{1/\alpha}]$,
	\[
	\abs{\varphi_n(y+t)} 
	\geq \norme{\varphi_n}_\infty - \frac{t^2}{2} \lambda_n \norme{\varphi_n}_\infty.
	\]
	There is a constant $c \in (0,1)$ depending only on $\alpha$ such that $c\lambda_n^{-1/2} \leq \lambda_n^{1/\alpha}$ for any $n \geq 0$.
	Then, we write
    \begin{align*}
	1 &= \norme{\varphi_n}_2^2
	\geq \int_0^{c\lambda_n^{-1/2}} \abs{\varphi_n(y+t)}^2 \diff t \\
	& \geq \norme{\varphi_n}_\infty^2 \int_0^{c\lambda_n^{-1/2}} \left( 1 - \frac{t^2}{2} \lambda_n \right) \diff t 
	= \norme{\varphi_n}_\infty^2 \lambda_n^{-1/2} \left( c - \frac{c^3}{6} \right),
    \end{align*}
	and therefore $\norme{\varphi_n}_\infty \leq C \lambda_n^{1/4}$ and the result follows from \eqref{eq:asymp_lambda_n}.
\end{proof}

\subsection{Proof of Proposition \ref{prop:estimate_g}}

In this section, we prove, up to some postponed lemmas, Proposition \ref{prop:estimate_g} concerning the fundamental solution estimate for the PDE \eqref{eq:main_PDE}.
We follow ideas from~\cite[Proposition A.2]{MaiZei2016}, which study the case $\alpha = 1$ and where the function $(1-t)^{-\alpha}$ is replaced by a $\cC^1$ function on $[0,1]$ (hence, with no explosion at 1). Moreover, we aim at getting better error terms, in particular with explicit tails in terms of $x$ and $y$.

Fix some horizon $T \in (0,1)$.
For $q \colon [0,T] \to (0,\infty)$ a Lipschitz continuous function, we consider the more general PDE on $[0,T] \times \R$
\begin{equation} \label{eq:PDE}
\partial_t u 
= \varrho \left( \partial_{xx}^2 u - q(t) \abs{x}^{\alpha} u \right),
\end{equation}
for some $\alpha > 0$ fixed and $\varrho > 0$ which is a parameter meant to be large. 

Fix some $\xi \in \R$. Recall our aim is to estimate $g(0,\xi;T,x)$.
To do this, we study the PDE~\eqref{eq:PDE} with $q = q_*$ or $q = q^*$, where $q_*$ and $q^*$ are well-chosen functions on $[0,T]$, given by Lemma \ref{construction:q_star} and illustrated by Figure \ref{fig:q}, which satisfy in particular $q_*(t) \leq (1-t)^{-\alpha} \leq q^*(t)$.
Then, if $g_q$ denotes the fundamental solution of the PDE~\eqref{eq:PDE} given by Lemma \ref{lem:fundamental_solution}, it follows from this probabilistic representation that
\begin{equation} \label{eq:comparison}
\forall t \in (0,T], \quad \forall x \in \R, \quad 
g_{q^*}(0,\xi;t,x) \leq g(0,\xi;t,x) \leq g_{q_*}(0,\xi;t,x),
\end{equation}
so it is enough to estimate $g_q$ for $q = q_*$ or $q = q^*$.

\begin{figure}[bt]
	\centering
	\begin{tikzpicture}[xscale=12,yscale=1.5]
	\draw[->,>=latex] (0,0) -- (.9,0);
	\draw (.9,-.04) node[below]{$t$};
	\draw[->,>=latex] (0,0) -- (0,5.2);
	\draw[dashed] (.15,-.04) -- (.15,{1/(1-.15)});
	\draw (.15,-.04) node[below]{\small$\eps_1$};
	\draw[dashed] (.3,-.04) -- (.3,{1/(1-.3)});
	\draw (.3,-.04) node[below]{\small$2\eps_1$};
	\draw[dashed] (.6,-.04) -- (.6,{1/(1-.6)});
	\draw (.6,-.04) node[below]{\small$T-2\eps_2$};
	\draw[dashed] (.7,-.04) -- (.7,{1/(1-.8)});
	\draw (.7,-.04) node[below]{\small$T-\eps_2$};
	\draw[dashed] (.8,-.04) -- (.8,{1/(1-.8)});
	\draw (.8,-.04) node[below]{\small$T$};
	\draw[very thick,NavyBlue] [domain=0:.8,samples=50] plot (\x,{1/(1-\x)});
	\draw[NavyBlue] (.83,4) node{$(1-t)^{-\alpha}$};
	\draw[very thick,BrickRed] (0,1) -- (.15,1) -- (.3,{1/(1-.3)});
	\draw[very thick,BrickRed,loosely dashed] [domain=0.3:.6] plot (\x,{1/(1-\x)});
	\draw[very thick,BrickRed,dashed] [domain=.6:.7] plot (\x,{1/(1-\x)});
	\draw[very thick,BrickRed] (.7,{1/(1-.7)}) -- (.8,{1/(1-.7)});
	\draw[BrickRed] (.76,3.1) node{$q_*(t)$};
	\draw[very thick,ForestGreen] (0,{1/(1-.15)}) -- (.15,{1/(1-.15)});
	\draw[very thick,ForestGreen,dashed] [domain=0.15:.3] plot (\x,{1/(1-\x)});
	\draw[very thick,ForestGreen,loosely dashed] [domain=0.31:.6] plot (\x,{1/(1-\x)});
	\draw[very thick,ForestGreen] (.6,{1/(1-.6)}) -- (.7,{1/(1-.8)}) -- (.8,{1/(1-.8)});
	\draw[ForestGreen] (.6,3.7) node{$q^*(t)$};
	\end{tikzpicture}
	\caption{Construction of $q_*$ and $q^*$ in terms of the two auxiliary parameters $\varepsilon_1$ and $\varepsilon_2$.}
	\label{fig:q}
\end{figure}
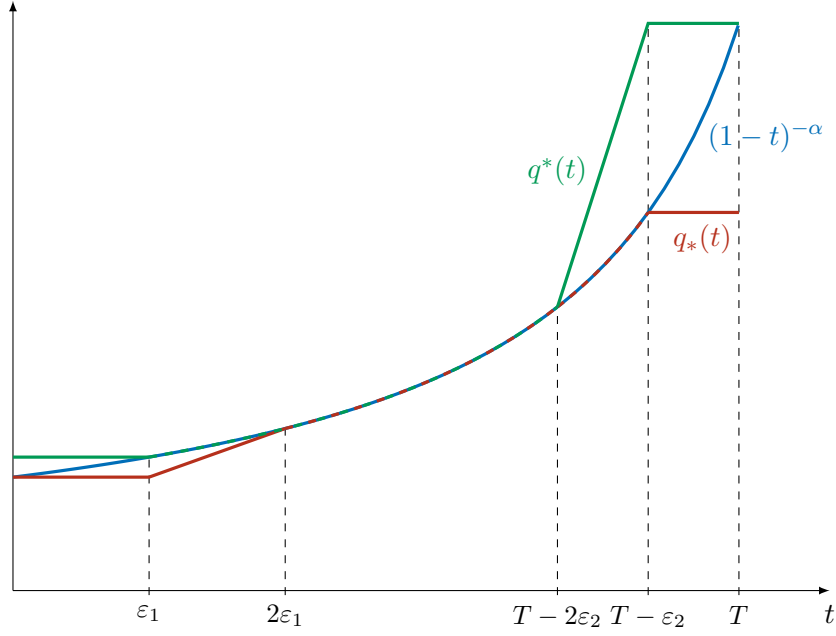

Fix some Lipschitz continuous function $q \colon [0,T] \to (0,\infty)$.
Omitting the dependence in $\xi$ which is fixed, we define, for any $t \in (0,T]$ and $x > 0$,
\begin{align} 
    W_t(x) 
    & \coloneqq g_q(0,\xi;t,x) \exp \left( \varrho \int_0^t \lambda_{q(s),0} \diff s \right) \nonumber \\
    & = g_q(0,\xi;t,x) \exp \left( \lambda_0 \varrho \int_0^t q(s)^{2/(2+\alpha)} \diff s \right), \label{eq:def_W}
\end{align}
where we used \eqref{eq:eigenvalues_Lgamma} in the second equality.
Moreover, for any $t > 0$, $g_q(0,\xi;t,\cdot) \in L^2(\R)$ (it is dominated by a Gaussian function by Lemma \ref{lem:fundamental_solution}) and so $W_t \in L^2(\R)$
and we can set, for any $n \geq 0$,
\begin{equation} \label{eq:def_c_n}
    c_n(t) 
    \coloneqq \langle \varphi_{q(t),n}, W_t \rangle
    = \langle \varphi_{q(t),n},g_q(0,\xi;t,\cdot) \rangle 
    \exp \left( \lambda_0 \varrho \int_0^t q(s)^{2/(2+\alpha)} \diff s \right).
\end{equation}
Since $(\varphi_{q(t),n})_{n\geq 0}$ is an orthonormal basis of $L^2(\R)$ by Proposition \ref{prop:sturm-liouville}.\ref{it:link_cL_cL_q}, we get
\begin{equation} \label{eq:decompo_W}
    \forall t \in (0,T], \quad \forall x \in \R, \quad 
    W_t(x) = \sum_{n\geq 0} c_n(t)\varphi_{q(t),n}(x).
\end{equation}
We also define 
\begin{equation} \label{eq:def_c_n(0)}
c_n(0) \coloneqq \varphi_{q(0),n}(\xi),
\end{equation}
which makes $c_n$ continuous at 0 as mentioned in the forthcoming Lemma \ref{lem:ode:c}. 

A key property of the functions $q_*$ and $q^*$ is that they are constant on intervals $[0,\eps_1]$ and $[T-\eps_2,T]$ for some parameters $\eps_1,\eps_2$. On these intervals, $c_0(t)$ stays constant, while $c_n(t)$ for $n \geq 1$ decays exponentially fast (see Corollary \ref{cor:if_q_is_constant}).
Filling the gap between times $\eps_1$ and $T-\eps_2$, we show in Lemma \ref{lem:precise_bounds} that $c_0(T) \simeq c_0(0)$ and that $c_n(T)$ is small for $n \geq 1$.
This lemma is a key tool in the proof of Proposition \ref{prop:estimate_g} below.

\begin{proof}[Proof of Proposition \ref{prop:estimate_g}]
	We first choose the parameters $\eps_1,\eps_2 > 0$. 
	The conditions we need in this proof are the following
	\begin{align}
	\eps_1,\eps_2 \leq T/10
	\qquad & \text{and} \qquad 
	\eps_2 \leq (1-T)/10, 
	\label{eq:requirements_eps_1} \\
	\varrho \eps_1 \geq 1
	\qquad & \text{and} \qquad 
	\varrho \varepsilon_2 (1-T)^{-\kappa} \geq 1,
	\label{eq:requirements_eps_2} \\
	\varrho \eps_1^2 \leq 1
	\qquad & \text{and} \qquad 
	\varrho \eps_2^2 (1-T)^{-\kappa-1} \leq 1.
	\label{eq:requirements_eps_3}
	\end{align}
	We set 
	\begin{equation} \label{eq:choice_eps}
	\varepsilon_1 \coloneqq \left( \frac{\delta}{\varrho} \right)^{1/2}
	\qquad \text{and} \qquad 
	\varepsilon_2 \coloneqq \left( \frac{\delta (1-T)^{1+\kappa}}{\varrho} \right)^{1/2}
	= (1-T) \left( \frac{\delta}{\varrho (1-T)^{1-\kappa}} \right)^{1/2}.
	\end{equation}
	We now check they satisfy the conditions listed above.
	The 1st part of \eqref{eq:requirements_eps_1} follows from $\varepsilon_2 \leq \varepsilon_1$ and $\delta \leq (T^2\varrho)/100$, 
	the 2nd one from $\varrho (1-T)^{1-\kappa} \geq 10$ and $\delta \leq 1/10$.
	The 1st part of \eqref{eq:requirements_eps_2} is obtained by noting that $\delta \geq 1/\varrho$ and the 2nd one follows from $\delta \geq 1/(\varrho(1-T)^{1-\kappa})$.
	Finally, \eqref{eq:requirements_eps_3} only requires $\delta \leq 1$ which is true.
	Throughout the proof, we use only the properties of $\eps_1$ and $\eps_2$ listed in \eqref{eq:requirements_eps_1}-\eqref{eq:requirements_eps_2}-\eqref{eq:requirements_eps_3} and express the bounds in terms of $\eps_1$ and $\eps_2$. Their precise choice is only used to deduce the proposition from these bounds. 
	This highlights the separated roles of $\eps_1$ and $\eps_2$ and hopefully can help to understand the precise choice which is made here.

	By \eqref{eq:requirements_eps_1}, we can apply Lemma \ref{construction:q_star} to consider functions $q_*$ and $q^*$ satisfying the properties listed there (see also Figure~\ref{fig:q}).
	Now, note that, by \eqref{eq:comparison}, it is enough to consider $q=q_*$ or $q = q^*$ and prove the bound for
	\begin{equation} \label{eq:goal}
	    \abs{\exp \left( \lambda_0 \varrho \int_0^T \frac{\diff s}{(1-s)^{\kappa}} \right) g_q(0,\xi;T,x) - \varphi_{0}(\xi) \varphi_{(1-T)^{-\alpha},0}(x)},
	\end{equation}
	where we rewrote $(1-T)^{-\kappa/4} \varphi_{0}( (1-T)^{-\kappa/2} x) = \varphi_{(1-T)^{-\alpha},0}(x)$ by Proposition \ref{prop:sturm-liouville}.\ref{it:link_cL_cL_q}.
	By the triangle inequality, we can bound \eqref{eq:goal} by $T_1 + T_2 + T_3$, where
	\begin{align*}
	    T_1 & \coloneqq \abs{\exp \left( \lambda_0 \varrho \int_0^T \frac{\diff s}{(1-s)^{\kappa}} \right) g_q(0,\xi;T,x) - W_T(x)}, \\
	    T_2 & \coloneqq \abs{W_T(x) - \varphi_{q(0),0}(\xi) \varphi_{q(T),0}(x)},  \\
	    T_3 & \coloneqq \abs{\varphi_{q(0),0}(\xi) \varphi_{q(T),0}(x) - \varphi_{0}(\xi) \varphi_{(1-T)^{-\alpha},0}(x)}.
	\end{align*}
	We start with $T_2$. By \eqref{eq:decompo_W} and recalling that $\varphi_{q(0),0}(\xi) = c_0(0)$, we have
	\begin{align*}
	    T_2 
	    & \leq \abs{c_0(T) - c_0(0)} \varphi_{q(T),0}(x)
	    + \sum_{n\geq1} \abs{c_n(T) \varphi_{q(T),n}(x)} \\
	    & \leq C q(T)^{1/[2(2+\alpha)]} 
	    \left( \frac{1}{\varrho(1-T)^{1-\kappa}} + e^{- c\varrho T} \right)
	    \left( 
	    e^{-c \abs{\xi}^{(2+\alpha)/2}}
	    + e^{-c \varrho \varepsilon_1 (1+\abs{\xi}^{\alpha})}
	    \right) \\
	    & \quad {} \times
	    \left( 
	    \varphi_0 \left( q(T)^{1/(2+\alpha)}x \right)
	    + \sum_{n\geq1} e^{-c \varrho \varepsilon_2 (1-T)^{-\kappa} n^{\kappa}}
	    \abs{\varphi_n \left( q(T)^{1/(2+\alpha)}x \right)}
	    \right), 
	\end{align*}
	by Lemma \ref{lem:precise_bounds} (using here $\varrho\varepsilon_1 \geq 1$ by \eqref{eq:requirements_eps_2}) and \eqref{eq:eigenvalues_Lgamma}.
	Using Proposition~\ref{prop:sturm-liouville}.\ref{it:tail} and that $q(T) \geq c (1-T)^{-\alpha}$, the series on the right-hand side of the last equation is at most
	\begin{align*}
		& C \sum_{n\geq1} e^{-c \varrho \varepsilon_2 (1-T)^{-\kappa} n^{\kappa}} \cdot
	    n^3 \left[ 1 \wedge \exp \left( - c \left( (1-T)^{-\alpha/2} \abs{x}^{(2+\alpha)/2} - Cn \right) \right) \right] \\
	    & \leq C e^{-c \varrho \varepsilon_2 (1-T)^{-\kappa}}
	    \left( 
	    \exp\left(-c (1-T)^{-\alpha/2} \abs{x}^{(2+\alpha)/2} \right)
	    + \exp\left(-c \varrho \varepsilon_2 (1-T)^{-\alpha} \abs{x}^{\alpha} \right)
	    \right),
	\end{align*}
	by Lemma \ref{lem:technical_series} with $u = \varrho \varepsilon_2 (1-T)^{-\kappa} \geq 1$ by \eqref{eq:requirements_eps_2} and $v = (1-T)^{-\alpha/2} \abs{x}^{(2+\alpha)/2}$.
	Using also Proposition~\ref{prop:sturm-liouville}.\ref{it:tail} to bound $\varphi_0$, we get
	\begin{align*}
	    T_2
	    & \leq C (1-T)^{-\kappa/4}
	    \left( \frac{1}{\varrho(1-T)^{1-\kappa}} + e^{- c\varrho T} \right)
	    \left( 
	    e^{-c \abs{\xi}^{(2+\alpha)/2}}
	    + e^{-c \varrho \varepsilon_1 (1+\abs{\xi}^{\alpha})}
	    \right) \\
	    & \quad {} \times
	    \left( 
	    \exp \left( -c \left( \frac{\abs{x}}{(1-T)^{\kappa/2}} \right)^{(2+\alpha)/2} \right)
	    + \exp\left(-c \frac{\varrho \varepsilon_2}{(1-T)^{\kappa}}
	    \left( 1 + \left( \frac{\abs{x}}{(1-T)^{\kappa/2}} \right)^{\alpha} \right) \right)
	    \right).
	\end{align*}
	With our choice of $\eps_1$ and $\eps_2$ in \eqref{eq:choice_eps} and using $\delta \geq 1/(\varrho(1-T)^{1-\kappa})$, this is smaller than the bound appearing in the statement of the proposition.
	
	We now bound $T_3$. For this, we first fix some $\theta > 1$ and note that, for any $q,p > 0$ such that $p/q \in [\theta^{-1},\theta]$ and any $x > 0$, using \eqref{eq:eigenvalues_Lgamma} and standard inequalities, with constants $C,c>0$ that depend only on $\theta$ and~$\alpha$,
	\begin{align*}
	& \abs{\varphi_{p,0}(x) - \varphi_{q,0}(x)} \\
	& \leq \abs{p^{1/[2(2+\alpha)]}-q^{1/[2(2+\alpha)]}} \varphi_0 \left( p^{1/(2+\alpha)}x \right)
	+ q^{1/[2(2+\alpha)]} \abs{\varphi_0 \left( p^{1/(2+\alpha)}x \right) - \varphi_0 \left( q^{1/(2+\alpha)}x \right)} \\
	& \leq C \left( 
	p^{1/[2(2+\alpha)]} \abs{1-\frac{q}{p}} \varphi_0 \left( p^{1/(2+\alpha)}x \right)
	+ p^{3/[2(2+\alpha)]} \abs{1-\frac{q}{p}} \cdot x \cdot \max_{t \geq (p\wedge q)^{1/(2+\alpha)}x} \abs{\varphi_0'(t)}
	\right) \\
	& \leq C p^{1/[2(2+\alpha)]} \abs{1-\frac{q}{p}}
	\exp \left( -c p^{1/2} \abs{x}^{(2+\alpha)/2} \right),
	\end{align*}
	using Proposition~\ref{prop:sturm-liouville}.\ref{it:tail}-\ref{it:tail_derivative} and that $(1+t^{2/(2+\alpha)})e^{-ct} \leq C e^{-ct}$ for $t>0$ up to a modification of $c$.
	By parity, the same inequality holds for $x <0$.
	Therefore, we get, using again Proposition~\ref{prop:sturm-liouville}.\ref{it:tail} together with Lemma \ref{construction:q_star}.\ref{it:encadrement_q},
	\begin{align*}
	T_3 
	& \leq \varphi_{q(0),0}(\xi) \abs{\varphi_{q(T),0}(x) - \varphi_{(1-T)^{-\alpha},0}(x)} 
	+ \varphi_{(1-T)^{-\alpha},0}(x) \abs{\varphi_{q(0),0}(\xi) - \varphi_{0}(\xi)} \\
	& \leq C (1-T)^{-\frac{\alpha}{2(2+\alpha)}} e^{-c \abs{\xi}^{(2+\alpha)/2}} 
	\exp \left( - c (1-T)^{-\alpha/2} \abs{x}^{(2+\alpha)/2} \right) \\
        & \hspace{6cm} \times
	\left( \abs{1-q(0)} + \abs{1-\frac{q(T)}{(1-T)^{-\alpha}}} \right) \\
	& \leq C (1-T)^{-\kappa/4} e^{-c \abs{\xi}^{(2+\alpha)/2}} 
	\exp \left( - c \left( \frac{\abs{x}}{(1-T)^{\kappa/2}} \right)^{(2+\alpha)/2} \right) 
	\left( \eps_1 + \frac{\eps_2}{1-T} \right).
	\end{align*}
	With our choice of $\eps_1$ and $\eps_2$ in \eqref{eq:choice_eps}, we have $\eps_1 \leq \varrho^{-1} + \delta$ and $\eps_2/(1-T) \leq (\varrho (1-T)^{1-\kappa})^{-1} + \delta$, so, recalling $\delta \geq (\varrho(1-T)^{1-\kappa})^{-1}$, this last bound is smaller than the bound appearing in the statement of the proposition.

	Finally, we bound $T_1$. By definition of $W_T(x)$, we have
	\[
	T_1 = \abs{\exp \left( \lambda_0 \varrho \int_0^T \left( (1-s)^{-2\alpha/(2+\alpha)} - q(s)^{2/(2+\alpha)} \right) \diff s \right) - 1} \cdot W_T(x).
	\]
	By Lemma \ref{construction:q_star}.\ref{it:integrale}, the quantity in the exponential above is bounded in absolute values by 
	$10 \lambda_0 \varrho (\eps_1^2 + \eps_2^2 (1-T)^{-\kappa-1})$, which is itself at most $20 \lambda_0$ by \eqref{eq:requirements_eps_3}.
	Hence, using that the exponential is Lipschitz continuous on $(-\infty,20 \lambda_0]$, we get
	\[
	T_1 \leq C \left( \varrho \eps_1^2 + \varrho \eps_2^2 (1-T)^{-\kappa-1} \right) W_T(x).
	\]
	Now, note that $W_T(x) \leq \varphi_{0}(\xi) \varphi_{(1-T)^{-\alpha},0}(x) + T_2 + T_3$, so using again \eqref{eq:eigenvalues_Lgamma} and Proposition~\ref{prop:sturm-liouville}.\ref{it:tail}, we get
	\begin{align*}
	T_1 & \leq C \left( \varrho \eps_1^2 + \varrho \eps_2^2 (1-T)^{-\kappa-1} \right) \\
	& \qquad {} \times
	\left( T_2 + T_3 +
	(1-T)^{-\kappa/4} e^{-c \abs{\xi}^{(2+\alpha)/2}} 
	\exp \left( - c (1-T)^{-\alpha/2} \abs{x}^{(2+\alpha)/2} \right)
	\right).
	\end{align*}
	With our choice of $\eps_1$ and $\eps_2$ in \eqref{eq:choice_eps}, we have $\varrho \eps_1^2 = \varrho \eps_2^2 (1-T)^{-\kappa-1} = \delta$, and combining the bounds for $T_1$, $T_2$ and $T_3$ concludes the proof. 
\end{proof}

\subsection{Approximating the inhomogeneity}

\begin{lem}\label{construction:q_star} 
	Let $\alpha \in (0,2)$.
	For any $T \in (0,1)$ and $\eps_1,\eps_2 \in (0,T/10)$ such that $\eps_2 \leq (1-T)/10$, there exist functions $q^*,q_*\colon [0,T]\to [0,\infty)$ that satisfy the following:
	\begin{enumerate}
		\item\label{it:prop_q} $q_*,q^*$ are non-decreasing, Lipschitz continuous on $[0,T]$, and differentiable on $[0,T]$ except at finitely many points;
	    \item\label{it:constant} $q_*,q^*$ are constant on $[0,\eps_1]$ and $[T-\eps_2,T]$;
	    \item\label{it:equal_part} For any $t\in [2\eps_1, T-2\eps_2]$, $q_*(t) = (1-t)^{-\alpha} = q^*(t)$;
	    \item\label{it:LB_UP} For any $t\in [0,T]$, $q_*(t)\leq (1-t)^{-\alpha} \leq q^*(t)$;
	    \item\label{it:encadrement_q_beginning} For any $t \in [0,2\eps_1]$,
	    \[
	    1 - 4 \varepsilon_1 \leq \frac{q_*(t)}{(1-t)^{-\alpha}} \leq \frac{q^*(t)}{(1-t)^{-\alpha}} \leq 1 + 3 \varepsilon_1.
	    \]
	    \item\label{it:encadrement_q_end} For any $t\in [T-2\eps_2,T]$, 
	    \[
	    1- \frac{2 \varepsilon_2}{1-T} \leq \frac{q_*(t)}{(1-t)^{-\alpha}} \leq \frac{q^*(t)}{(1-t)^{-\alpha}} \leq 1 + \frac{5 \varepsilon_2}{1-T}.
	    \]
	    \item\label{it:encadrement_q} For any, $t\in [0,T]$, 
	    \[
	    \frac{1}{2} \leq \frac{q_*(t)}{(1-t)^{-\alpha}} \leq \frac{q^*(t)}{(1-t)^{-\alpha}}\leq 2.
	    \]
	    \item\label{it:derivative} For any $t \in [0,T]$ and $q = q_*$ or $q^*$,
	    \[
	    \frac{q'(t)}{q(t)} \leq \frac{8}{(1-t)}.
	    \]
	    \item\label{it:integrale} For any $t \in [0,T]$ and $q = q_*$ or $q^*$,
	    \[
	    \int_0^T \abs{(1-t)^{-2\alpha/(2+\alpha)} - q(t)^{2/(2+\alpha)}} \diff t
	    \leq 10 \eps_1^2 + \frac{10 \eps_2^2}{(1-T)^{\kappa+1}}.
	    \]
	\end{enumerate}
\end{lem}

\begin{proof}
We define $q_*$ and $q^*$ as follows (see also Figure \ref{fig:q}):
\begin{itemize}
	\item for $t \in[0,\eps_1]$, $q_*(t)=1$ and $q^*(t) = (1-\eps_1)^{-\alpha}$,
	\item for $t \in [\eps_1,2\eps_1]$, $q^*(t) = (1-t)^{-\alpha}$ and $q_*$ is obtained by linear interpolation between the values at the endpoints.
    \item for $t \in[2\eps_1,T-2\eps_2]$, $q_*(t) = q^*(t) = (1-t)^{-\alpha}$,
    \item for $t\in [T-2\eps_2, T-\eps_2]$, $q_*(t) = (1-t)^{-\alpha}$ and $q^*$ is obtained by linear interpolation between the values at the endpoints.
    \item for $t \in [T-\eps_2,T]$, $q_*(t)=(1-T+\eps_2)^{-\alpha}$ and $q^*(t) = (1-T)^{-\alpha}$.
\end{itemize}

Parts \ref{it:prop_q}-\ref{it:constant}-\ref{it:equal_part} are clearly true.
For Part \ref{it:LB_UP}, the only non-trivial thing to check is that $q_*(t) \leq (1-t)^{-\alpha}$ for $t \in [\eps_1,2\eps_1]$. By convexity of $s \mapsto (1-s)^{-\alpha}$, it is enough to compare the left-derivatives at $2\eps_1$, which amounts to show that
\[
	\frac{(1-2\eps_1)^{-\alpha}-1}{\eps_1} \geq \alpha (1-2\eps_1)^{-\alpha-1}.
\]
We have $(1-2\eps_1)^{-\alpha}-1 \geq 2 \alpha \eps_1$, hence it is enough to show $(1-2\eps_1)^{-\alpha-1} \leq 2$, which is equivalent to $\eps_1 \leq \frac{1}{2}(1-2^{-1/(\alpha+1)})$. This is true by noting that $\eps_1 \leq 1/10 \leq \frac{1}{2}(1-2^{-1/3}) \leq \frac{1}{2}(1-2^{-1/(\alpha+1)})$. 
Hence, Part \ref{it:LB_UP} is proved.

For Part \ref{it:encadrement_q_beginning}, consider $t \in [0,2\eps_1]$. Then, 
\[
\frac{q_*(t)}{(1-t)^{-\alpha}} \geq \frac{1}{(1-2\eps_1)^{-\alpha}}
= (1-2\eps_1)^{\alpha} 
\geq \begin{cases}
1 - 2\alpha \eps_1, & \text{if } \alpha \geq 1, \\
1 - 2\eps_1, & \text{if } \alpha < 1.
\end{cases}
\]
This proves the lower bound. For the upper bound, using $(1-t)^{-\alpha} \geq 1$, we have
\begin{equation} \label{eq:UB_eps_1}
\frac{q_*(t)}{(1-t)^{-\alpha}}
\leq (1-\eps_1)^{-\alpha} 
\leq 1 + \eps_1 \alpha (1-\eps_1)^{-\alpha-1} 
\leq 1 + 2\eps_1 (9/10)^{-3},
\end{equation}
using $\eps_1 \leq 1/10$ and $\alpha \leq 2$. This proves Part \ref{it:encadrement_q_beginning}.

For Part \ref{it:encadrement_q_end}, consider $t\in [T-2\eps_2,T]$. Then, 
\[
\frac{q_*(t)}{(1-t)^{-\alpha}} 
\geq \frac{(1-T+\eps_2)^{-\alpha}}{(1-T)^{-\alpha}}
= \left( 1 + \frac{\eps_2}{1-T} \right)^{-\alpha} 
\geq 1 - \alpha \frac{\eps_2}{1-T},
\]
by convexity of $s \mapsto (1-s)^{-\alpha}$. On the other hand,
\begin{equation} \label{eq:UB_eps_2}
\frac{q^*(t)}{(1-t)^{-\alpha}}
\leq \frac{(1-T)^{-\alpha}}{(1-T+2\eps_2)^{-\alpha}}
\leq \left( 1 + \frac{2\eps_2}{1-T} \right)^2
= 1 + \frac{4\eps_2}{1-T} + \frac{4\eps_2^2}{(1-T)^2}
\leq 1 + \frac{5 \varepsilon_2}{1-T},
\end{equation}
using $\eps_2 \leq (1-T)/10$. This proves Part \ref{it:encadrement_q_end}.

Part \ref{it:encadrement_q} follows directly from Parts \ref{it:encadrement_q_beginning}-\ref{it:encadrement_q_end} and the assumptions on $\eps_1$ and $\eps_2$.

Part \ref{it:derivative} is immediate for $t \in [0,\eps_1] \cup [2\eps_1, T-2\eps_2] \cup [T-\eps_2,T]$.
For $t \in [\eps_1,2\eps_1]$, it has to be checked for $q_*$: using $q_*(t) \geq 1$ and then $(1-2\eps_1)^{-\alpha} \leq 1 + 4\eps_1 (8/10)^{-3}$ by proceeding as in \eqref{eq:UB_eps_1}, we get
\[
\frac{q_*'(t)}{q_*(t)} 
\leq \frac{(1-2\eps_1)^{-\alpha}-1}{\eps_1}
\leq 4 \cdot (8/10)^{-3}
\leq \frac{8}{(1-t)}.
\]
For $t \in [T-2\eps_2,T-\eps_2]$, it has to be checked for $q^*$: using $q^*(t) \geq (1-T+2\eps_2)^{-\alpha}$ and then proceeding as in \eqref{eq:UB_eps_2}, we get
\[
\frac{(q^*)'(t)}{q^*(t)} 
\leq \frac{(1-T)^{-\alpha}-(1-T+2\eps_2)^{-\alpha}}{\eps_2 (1-T+2\eps_2)^{-\alpha}}
\leq \frac{5}{1-T}
\leq \frac{6}{1-t},
\]
using that $1-t \leq 1-T+2\eps_2 \leq \frac{6}{5} (1-T)$ by assumption on $\eps_2$.
This proves Part \ref{it:derivative}.

Finally it remains to prove Part \ref{it:integrale}. Using $\lvert 1-x^{2/(2+\alpha)}\rvert \leq \abs{1-x}$ for any $x > 0$, we get
\begin{align*}
\int_0^T \abs{(1-t)^{-2\alpha/(2+\alpha)} - q(t)^{2/(2+\alpha)}} \diff t
 &\leq \int_0^T (1-t)^{-\kappa} \abs{1 - \frac{q(t)}{(1-t)^{-\alpha}}} \diff t \\
 \leq & \int_0^{2\eps_1} (1-2\eps_1)^{-\kappa} 4\eps_1 \diff t 
+ \int_{T-2\eps_2}^T (1-T)^{-\kappa} \frac{5 \eps_2}{1-T} \diff t,
\end{align*}
using Parts \ref{it:encadrement_q_beginning}-\ref{it:encadrement_q_end}.
Finally, using $(1-2\eps_1)^{-\kappa} \leq 5/4$ yields the result.
\end{proof}

\subsection{General properties of the coefficients \texorpdfstring{$c_n(t)$}{cn(t)}}

In this section, we establish several results concerning the coefficients $c_n(t)$ holding for general functions $q$.
Recall their definition in \eqref{eq:def_c_n} and \eqref{eq:def_c_n(0)} and that they implicitly depend on $\alpha$, $\varrho$ and $\xi$. 
Recall also the definition of the $\lambda_n$'s and $\varphi_n$'s in Proposition \ref{prop:sturm-liouville}.
\begin{lem} \label{lem:ode:c}
	Let $\alpha > 0$, $\varrho \geq 1$, $\xi\in\R$ and $T\in (0,1)$.
	Let $q \colon [0,T] \to (0,\infty)$ be a Lipschitz continuous function. Then, for any $n \geq 0$, $c_n$ is continuous on $[0,T]$.
	
	If moreover $q$ is differentiable at some $t \in (0,T]$, then $c_n$ is differentiable at $t$ and, writing $c(t)=(c_0(t),c_1(t),\ldots)$, we have $\partial_t c(t)= (-D(t) + A(t))c(t)$, where we set
	\[
	\begin{cases}
	D(t) = \varrho q(t)^{2/(2+\alpha)} D, \\
	D = \mathrm{Diag}((\lambda_i - \lambda_0)_{i\geq 0}),
	\end{cases}
	\quad \quad 
	\begin{cases}
	A(t) = \frac{q'(t)}{q(t)}A, \\
	A = \frac{1}{2(2+\alpha)} (\langle \varphi_{j}, x\varphi_{i}'\rangle - \langle x\varphi_j', \varphi_i \rangle)_{i,j \geq 0}.
	\end{cases}
	\]
\end{lem}

\begin{proof}
	Continuity of $c_n$ on $(0,T]$ follows from the expression for $g_q$ given in Lemma \ref{lem:fundamental_solution} and properties of $\varphi_{q,n}$ listed in Proposition \ref{prop:sturm-liouville}.
	Moreover, $c_n$ is continuous at 0 because, as $t\to 0^+$, $\varphi_{q(t),n}$ converges uniformly to $\varphi_{q(0),n}$ and $g_q(0,\xi;t,\cdot)$ converges weakly to $\delta_\xi$. 
	
	Now, we assume $q$ is differentiable at some $t \in (0,T]$.
    Recall $g_q(0,\xi;\cdot,\cdot)$ solves the PDE~\eqref{eq:PDE}, hence $\partial_t g_q(0,\xi;t,\cdot) = -\varrho \cL_{q(t)} g_q(0,\xi;t,\cdot)$. We use this together with the self-adjointness of $\cL_{q(t)}$, we get
	\begin{align*}
	\partial_t c_n(t) 
	& = \langle \partial_t \varphi_{q(t), n},W_t \rangle
	+ \varrho\langle \cL_{q(t)} \varphi_{q(t),n},W_t \rangle 
	+ \lambda_0 \varrho q(t)^{2/(2+\alpha)}  \langle \varphi_{q(t),n}, W_t \rangle \\
	&= -\varrho q(t)^{2/(2+\alpha)}(\lambda_{n}-\lambda_{0}) c_n(t) + \sum_{k\geq0} c_k(t) \langle \varphi_{q(t),k}, \partial_t \varphi_{q(t), n} \rangle,
	\end{align*}
	where we used that $\varphi_{q(t),n}$ is an eigenfunction of $\cL_{q(t)}$ with $\lambda_{q(t),n} = q(t)^{2/(2+\alpha)}\lambda_n$.
	Then, recalling from \eqref{eq:eigenvalues_Lgamma} that $\varphi_{q,n}(y) = q^{1/[2(2+\alpha)]} \varphi_n(q^{1/(2+\alpha)} y)$, we have
	\[
	\langle \varphi_{q,k}, \partial_q \varphi_{q, n} \rangle
	= \frac{1}{(2+\alpha)q}
	\left( \frac{1}{2} \langle \varphi_{k}, \varphi_{n} \rangle 
	+ \langle \varphi_k, y \varphi_n' \rangle \right),
	\]
	and therefore
	\begin{equation}
	\partial_t c_n(t) 
	= -\varrho q(t)^{2/(2+\alpha)}(\lambda_{n}-\lambda_{0})c_n(t)   + \sum_{k\geq0} c_k(t)\frac{1}{2+\alpha}\frac{q'(t)}{q(t)}\left( \frac{1}{2} \langle \varphi_{k},\varphi_{n}\rangle + \langle \varphi_{k}, y\varphi_{n}'\rangle \right). \label{eq:ode_c}
	\end{equation}
	By  integration by parts, using the decay at infinity of $\varphi_n$ (see Proposition~\ref{prop:sturm-liouville}.\ref{it:tail}), we have
	\begin{equation} \label{eq:IBP}
	\langle \varphi_{k},\varphi_{n}\rangle + 2 \langle \varphi_{k}, x\varphi_{n}'\rangle 
	= \langle \varphi_{k}, x\varphi_{n}'\rangle - \langle x\varphi_k', \varphi_n \rangle.
	\end{equation}
	Combining this with \eqref{eq:ode_c} proves the result.
\end{proof}

A key observation is that when $q$ is constant on some interval, then $A(t) = 0$ on this interval and the ODE satisfied by $c(t)$ can be explicitly solved as follows.

\begin{cor}\label{cor:if_q_is_constant}
	Let $0\leq s_1 < s_2 \leq T$.
	Assume $q$ is constant on $[s_1,s_2]$. 
	Then, for any $n \geq 0$, 
	\[
	c_n(s_2)=c_n(s_1)\exp\left(-\varrho(\lambda_n-\lambda_0)(s_2-s_1) q(s_1)^{2/(2+\alpha)} \right).
	\]
	In particular, $c_0(s_2) = c_0(s_1)$.
\end{cor}


Finally, we conclude this section with some rough bounds on the coefficients $c_n(t)$.

\begin{lem}\label{lemma:a_priori_estimates_1} 
	Let $\alpha > 0$, $\varrho \geq 1$, $\xi\in\R$ and $T\in (0,1)$.
	Let $q \colon [0,T] \to (0,\infty)$ be a function Lipschitz continuous on $[0,T]$, and differentiable on $[0,T]$ except at finitely many points.
	Then, the following holds.
	\begin{enumerate}
	    \item\label{it:L_2_norm_decreases} The function $t \in (0,T] \mapsto \norme{c(t)}_2$ is decreasing; 
	    \item\label{it:overline_c} Let $\overline{c}(t) \coloneqq (0,c_2(t),c_3(t),\ldots)$. There exists $C=C(\alpha) > 0$ such that, for $0 < t_0 \leq t \leq T$,
	    \begin{align*}
	        \norme{\overline{c}(t)}_2 
	        &\leq \norme{\overline{c}(t_0)}_2 \exp\left(-\varrho(\lambda_1-\lambda_0) \int_{t_0}^t q(s)^{2/(2+\alpha)} \diff s \right) 
	    \\
	    & \quad + C \norme{c(t_0)}_2 \int_{t_0}^t \frac{\abs{q'(s)}}{q(s)} \exp\left( -\varrho(\lambda_1-\lambda_0)\int_s^t q(r)^{2/(2+\alpha)} \diff r \right) \diff s.
	    \end{align*}
	\end{enumerate}
\end{lem}

\begin{proof} 
	Part \ref{it:L_2_norm_decreases}. We use Lemma \ref{lem:ode:c}, noting that $D$ is diagonal with $D_{00}=0$ and other entries positive and that $A$ is anti-symmetric, to get, for any $t \in (0,T]$ where $q$ is differentiable,
	\[
	    \partial_t \norme{c(t)}^2_2 
	    = c(t)^T (A(t)^T-D(t)^T+A(t)-D(t)) c(t) 
	    = - 2c(t)^T D(t) c(t) 
	    \leq 0.
	\]
	Together with the continuity of the function $t \in [0,T] \mapsto \norme{c(t)}^2_2$, this implies that it is a decreasing function and hence the same is true for $\norme{c(t)}_2$. 

	Part \ref{it:overline_c}. Using again Lemma \ref{lem:ode:c} and the fact that $A$ is antisymmetric, we have
	\begin{align*}
	    \partial_t \norme{\overline{c}(t)}_2 
	    & = \frac{1}{2\norme{\overline{c}(t)}_2}
	    \partial_t \norme{\overline{c}(t)}_2^2 
	    \\
	    & = -\frac{1}{\norme{\overline{c}(t)}_2} \overline{c}(t)^T D(t) \overline{c}(t) - \frac{1}{\norme{\overline{c} (t)}_2} c_0(t)\sum_{j=1}^\infty A_{0j}(t)c_j(t) 
	    \\
	    & \leq - \varrho q(t)^{2/(2+\alpha)}(\lambda_1-\lambda_0) \norme{\overline{c}(t)}_2
	    + \abs{c_0(t)} \frac{\abs{q'(t)}}{q(t)}\norme{(A_{0j})_{j\geq 1}}_2,
	\end{align*}
	using $\lambda_n \geq \lambda_1$ for the first term and Cauchy-Schwarz inequality in $\ell^2$ for the second one.
	By \eqref{eq:IBP}, we have $A_{0j} = \frac{1}{2(2+\alpha)} (\langle \varphi_j,\varphi_0\rangle + 2 \langle \varphi_j, x\varphi_0'\rangle)$ and, using that $(\varphi_j)_{j\geq 0}$ is an orthonormal basis, it follows that $\norme{(A_{0j})_{j\geq 1}}_2 \leq \frac{1}{(2+\alpha)} (1+2\norme{x \varphi_0'}_2)$, which is a finite constant depending only on $\alpha$.
    Now, Grönwall's inequality yields
	\begin{align*}
	    \norme{\overline{c}(t)}_2 
	    & \leq \norme{\overline{c}(t_0)}_2 \exp\left(-\varrho(\lambda_1-\lambda_0)\int_{t_0}^t q(s)^{2/(2+\alpha)} \diff s  \right) 
	    \\
	    & \quad + C \int_{t_0}^t \abs{c_0(s)} \frac{\abs{q'(s)}}{q(s)} \exp\left( -\varrho(\lambda_1-\lambda_0)\int_s^t q(r)^{2/(2+\alpha)} \diff r \right) \diff s,
	\end{align*}
	and the result follows using that $\abs{c_0(s)} \leq \norme{c(s)}_2 \leq \norme{c(t_0)}_2$ by Part \ref{it:L_2_norm_decreases}.
\end{proof}

\subsection{Precise estimates for the coefficients \texorpdfstring{$c_n(t)$}{cn(t)}}

In this section, we build upon the general bounds of Lemma \ref{lemma:a_priori_estimates_1} to prove more precise estimates for $c_n(T)$ in the case $q = q_*$ or $q=q^*$. In particular, we use that these functions $q$ are constant on $[0,\varepsilon_1]$ and $[T-\varepsilon_2,T]$ and rely on Corollary \ref{cor:if_q_is_constant} on these intervals.

\begin{lem} \label{lem:precise_bounds}
	Let $\alpha \in (0,2)$, $\varrho \geq 1$ and $\xi\in\R$.
	Let $T \in (0,1)$, $\eps_1,\eps_2 \in (0,T/10)$ such that $\eps_2 \leq (1-T)/10$.
	Let $q$ be either $q_*$ or $q^*$ given by Lemma \ref{construction:q_star}. 
	Assume $\varrho \varepsilon_1 \geq 1$ and $\varrho(1-T)^{1-\kappa} \geq 1$.
	Then, there exist $c=c(\alpha)>0$ and $C=C(\alpha)>0$ such that
	\[
	    \abs{c_0(T) - c_0(0)}
	    \leq \frac{C}{\varrho (1-T)^{1-\kappa}} 
	    \left( 
	    e^{-c \abs{\xi}^{(2+\alpha)/2}}
	    + e^{-c \varrho \varepsilon_1 (1+\abs{\xi}^{\alpha})}
	    \right),
	\]
	and, for any $n \geq 1$,
	\begin{align*}
	\abs{c_n(T)}
	\leq C \left( \frac{1}{\varrho(1-T)^{1-\kappa}} + e^{- c\varrho T} \right)
	\left( 
	e^{-c \abs{\xi}^{(2+\alpha)/2}}
	+ e^{-c \varrho \varepsilon_1 (1+\abs{\xi}^{\alpha})}
	\right) 
	e^{-c \varrho \varepsilon_2 (1-T)^{-\kappa} n^{\kappa}}.
	\end{align*}
\end{lem}

\begin{proof} 
	Throughout this proof, constants $c$ and $C$ depend only on $\alpha$.
	
    \textit{Step 1: from $0$ to $\varepsilon_1$.} Since $q$ is constant on $[0,\eps_1]$ by Lemma \ref{construction:q_star}.\ref{it:constant}, we have by Corollary~\ref{cor:if_q_is_constant}, for all $n\geq 0$,
    \begin{equation} \label{eq:0_to_eps1}
    c_n(\varepsilon_1)
    = c_n(0) \exp\left( - \varrho \varepsilon_1 (\lambda_n-\lambda_0) q(0)^{2/(2+\alpha)} \right).
    \end{equation}
    Since $c_n(0) = \varphi_{q(0),n} (\xi)$, $q(0) \geq 1$ and $\lambda_n - \lambda_0 \geq cn^{2\alpha/(2+\alpha)}$ by \eqref{eq:asymp_lambda_n}, we get
    \begin{align*}
        \sum_{n\geq 1} c_n(\varepsilon_1)^2 
        & \leq \sum_{n\geq 1} \varphi_{q(0),n}(\xi)^2 
        \exp\left(-c \varrho \varepsilon_1 n^{2\alpha/(2+\alpha)} \right) \\
        & \leq C \sum_{n\geq 1} n^6 \left[ 1 \wedge \exp \left( - \frac{2}{2+\alpha} \left( \abs{\xi}^{(2+\alpha)/2} -  Cn \right) \right) \right]
        \exp\left(-c \varrho \varepsilon_1 n^{2\alpha/(2+\alpha)}\right),
    \end{align*}
    where we have used \eqref{eq:eigenvalues_Lgamma}, Proposition~\ref{prop:sturm-liouville}.\ref{it:tail} and $1 \leq q(0) \leq C$.
    Applying Lemma \ref{lem:technical_series} (using here $\varrho \varepsilon_1 \geq 1$), we get
    \begin{align*}
        \sum_{n\geq 1} c_n(\varepsilon_1)^2 
        & \leq C e^{-c\varrho \varepsilon_1} \left( e^{-c\abs{\xi}^{(2+\alpha)/2}} + e^{-c\varrho \varepsilon_1 \abs{\xi}^{\alpha}} \right).
    \end{align*}
    On the other hand, $c_0(\varepsilon_1) = c_0(0) = \varphi_{q(0),0}(\xi)$.
    In particular, using again \eqref{eq:eigenvalues_Lgamma}, Proposition~\ref{prop:sturm-liouville}.\ref{it:tail} and $1 \leq q(0) \leq C$,  we get
    \begin{equation} \label{eq:conclusion_step_1}
    \norme{c(\varepsilon_1)}_2
    \leq C \left( 
    e^{-c \abs{\xi}^{(2+\alpha)/2}}
    + e^{-c \varrho \varepsilon_1 (1+\abs{\xi}^{\alpha})}
    \right).
    \end{equation}
    
    \noindent\textit{Step 2: from $\varepsilon_1$ to $T-\varepsilon_2$.} 
    By Lemma \ref{lem:ode:c}, recalling that $D_{00} = 0$, we have 
    $\partial_t c_0(t) = \sum_{j\geq 1} A_{0j}(t)c_j(t)$ as soon as $q'(t)$ exists.
    Therefore, 
    \begin{equation} \label{eq:step2}
    \abs{c_0(T-\varepsilon_2) - c_0(\varepsilon_1)}
    \leq \int_{\varepsilon_1}^{T-\varepsilon_2}
    \abs{\sum_{j\geq 1} A_{0j}(t)c_j(t)} \diff t
    \leq C \int_{\varepsilon_1}^{T-\varepsilon_2}
    \norme{\overline{c}(t)}_2 \frac{q'(t)}{q(t)} \diff t,
    \end{equation}
    using Cauchy-Schwarz inequality and $\norme{(A_{0j})_{j\geq 1}}_2 \leq C$ (like in the proof of Lemma~\ref{lemma:a_priori_estimates_1}).
    We now bound $\norme{\overline{c}(t)}_2$.
    By Lemma \ref{lemma:a_priori_estimates_1}.\ref{it:overline_c}, we have
    \begin{align*}
    \norme{\overline{c}(t)}_2 
    &\leq \norme{\overline{c}(\varepsilon_1)}_2 \exp\left(-\varrho(\lambda_1-\lambda_0) \int_{\varepsilon_1}^t q(s)^{2/(2+\alpha)} \diff s \right) 
    \\
    & \quad + C \norme{c(\varepsilon_1)}_2 \int_{\varepsilon_1}^t \exp\left( -\varrho(\lambda_1-\lambda_0)\int_s^t q(r)^{2/(2+\alpha)} \diff r \right) \frac{q'(s)}{q(s)} \diff s 
    \\
    &\leq C \norme{c(\varepsilon_1)}_2 \left( \exp\left(- c\varrho \int_{\varepsilon_1}^t \frac{\diff s}{(1-s)^{\kappa}} \right) 
    + \int_{1-t}^{1-\varepsilon_1} \exp\left( - c\varrho \int_{1-t}^s \frac{\diff r}{r^{\kappa}} \right)
    \frac{\diff s}{s} \right),
    \end{align*}
    using Lemma \ref{construction:q_star}.\ref{it:encadrement_q}-\ref{it:derivative}, replacing $r$ by $1-r$ and $s$ by $1-s$ in the second term, and recalling that $\kappa = 2\alpha/(2+\alpha)$. 
    Then, we use that, for $s \geq 1-t$,
    \[
    \int_{1-t}^s \frac{\diff r}{r^{\kappa}}
    = \frac{1}{1-\kappa} \left( s^{1-\kappa} - (1-t)^{1-\kappa} \right)
    \geq \begin{cases}
    c (1-t)^{-\kappa} (s-(1-t)), & \text{if } s \leq 2(1-t), \\
    cs^{1-\kappa}, & \text{otherwise},
    \end{cases}
    \]
    and simply bound $\int_{\varepsilon_1}^t \frac{\diff s}{(1-s)^{\kappa}} \geq (t-\varepsilon_1)$. This yields
    \begin{align}
    \norme{\overline{c}(t)}_2 
    &\leq C \norme{c(\varepsilon_1)}_2 
    \left( e^{- c\varrho (t-\varepsilon_1)}
    + \int_{1-t}^{2(1-t)} e^{- c \varrho (1-t)^{-\kappa} (s-(1-t))}
    \frac{\diff s}{1-t}  + \int_{2(1-t)}^{1-\varepsilon_1} e^{-c \varrho s^{1-\kappa}}
    \frac{\diff s}{s} \right) \nonumber \\
    & \leq C \norme{c(\varepsilon_1)}_2 
    \left( e^{- c\varrho (t-\varepsilon_1)}
    + \frac{1}{\varrho(1-t)^{1-\kappa}} \right), \label{eq:bound_ov_c}
    \end{align}
    bounding by the integral from $1-t$ to $\infty$ for the second term and using the inequality $e^{-x}\leq \frac{1}{x}$ for $x>0$ for the third term.
    Coming back to \eqref{eq:step2} and using again $q'(t)/q(t) \leq C (1-t)^{-1}$ by Lemma \ref{construction:q_star}.\ref{it:derivative}, we get 
    \begin{align} 
    \abs{c_0(T-\varepsilon_2) - c_0(\varepsilon_1)}
    & \leq C \norme{c(\varepsilon_1)}_2 
    \int_{\varepsilon_1}^{T-\varepsilon_2}
    \left( e^{- c\varrho (t-\varepsilon_1)}
    + \frac{1}{\varrho(1-t)^{1-\kappa}} \right)  \frac{\diff t}{1-t}. \label{eq:bound_c_0_1}
    \end{align}
   	Now, note that
   	\begin{align*} 
   	\int_{\varepsilon_1}^{T-\varepsilon_2}
   	e^{- c\varrho (t-\varepsilon_1)} \frac{\diff t}{1-t}
   	& \leq 2 \int_{\varepsilon_1}^{1/2}
   	e^{- c\varrho (t-\varepsilon_1)} \diff t
   	+  \int_{1/2}^{(T-\varepsilon_2) \wedge(1/2)} e^{- c\varrho}
   	\frac{\diff t}{1-t} \\
   	& \leq \frac{C}{\varrho} + e^{-c\varrho} \log \frac{1}{1-T} \\
   	& \leq \frac{C}{\varrho},
   	\end{align*}
    using here the assumptions $(1-T)^{-1} \leq \varrho^{1/(1-\kappa)}$ and $\varrho \geq 1$.
    Coming back to \eqref{eq:bound_c_0_1}, we get
    \begin{align} 
    \abs{c_0(T-\varepsilon_2) - c_0(\varepsilon_1)}
    & \leq \frac{C\norme{c(\eps_1)}_2}{\varrho (1-T)^{1-\kappa}}. \label{eq:bound_c_0}
    \end{align}
    
    \noindent\textit{Step 3: from $T-\varepsilon_2$ to $T$.} 
    Since $q$ is constant on $[T-\varepsilon_2,T]$, by Corollary \ref{cor:if_q_is_constant}, we have, for all $n\geq 0$,
    \[
    c_n(T) 
    = c_n(T-\varepsilon_2) 
    \exp\left(- \varrho \varepsilon_2 (\lambda_n-\lambda_0) q(T)^{2/(2+\alpha)} \right).
    \]
    On the one hand, combining this with \eqref{eq:0_to_eps1} and \eqref{eq:bound_c_0}, we get
    \[
    \abs{c_0(T) - c_0(0)} 
    = \abs{c_0(T-\varepsilon_2) - c_0(\varepsilon_1)}
    \leq \frac{C\norme{c(\varepsilon_1)}_2}{\varrho (1-T)^{1-\kappa}}.
    \]
    On the other hand, for $n \geq 1$, using \eqref{eq:asymp_lambda_n}, this yields
    \begin{align*}
    \abs{c_n(T)}
    & \leq \norme{\overline{c}(T-\varepsilon_2)}_2 
    \exp\left(-c n^{2\alpha/(2+\alpha)} \varrho \varepsilon_2 (1-T)^{-2\alpha/(2+\alpha)}\right) \\
    & \leq C \norme{c(\varepsilon_1)}_2 
    \left( e^{- c\varrho T} + \frac{1}{\varrho(1-T)^{1-\kappa}} \right)
    \exp\left(-c n^{2\alpha/(2+\alpha)} \varrho \varepsilon_2 (1-T)^{-2\alpha/(2+\alpha)}\right),
    \end{align*}
    using \eqref{eq:bound_ov_c} and the fact that $\varepsilon_1 \leq T/4$.
    Combining the two last inequalities with \eqref{eq:conclusion_step_1} yields the result.
\end{proof}

\subsection{Technical bound on a series}

\begin{lem} \label{lem:technical_series}
    For any $\kappa \in (0,1)$ and $a_1,a_2,a_3,a_4>0$, there exists $C,c>0$ such that, for any $u \geq 1$ and $v \geq 0$,
    \[
    \sum_{n\geq 1} n^{a_4} \left( 1 \wedge e^{-a_1(v-a_2n)} \right) e^{-a_3un^\kappa}
    \leq C e^{-cu} \left( e^{-cv} + e^{-cuv^\kappa} \right).
    \]
\end{lem}

\begin{proof} In this proof, constants $C$ and $c$ can depend on $\kappa,a_1,a_2,a_3,a_4$.
    First note that, bounding $e^{-a_3un^\kappa} \leq e^{-a_3 u/2} \cdot e^{-a_3un^\kappa/2}$, we can get a factor $e^{-a_3 u/2}$ in front of the series and therefore, it is enough to prove
    \begin{equation} \label{eq:new_goal}
    \sum_{n\geq 1} n^{a_4} \left( 1 \wedge e^{-a_1(v-a_2n)} \right) e^{-a_3un^\kappa}
    \leq C \left( e^{-cv} + e^{-cuv^\kappa} \right).
    \end{equation}
    Letting $n_0 \coloneqq \lceil v/a_2 \vee 2 \rceil$, we split the series into a part $n < n_0$ and a part $n \geq n_0$.
    
    The part $n \geq n_0$ equals
    \begin{equation} \label{eq:part_gretaer_n_0}
    \sum_{n\geq n_0} n^{a_4} e^{-a_3un^\kappa}
    \leq \int_{n_0-1}^\infty (x+1)^{a_4} e^{-a_3ux^\kappa} \diff x
    \leq C u^{-a_4/\kappa} \int_{u(n_0-1)^\kappa}^\infty y^{a_4/\kappa-1} e^{-a_3y} \diff y,
    \end{equation}
    bounding $x+1 \leq 2x$ because $x \geq n_0-1 \geq 1$ and changing variables with $y = ux^\kappa$.
    Noting that $u^{-a_4/\kappa} \leq 1$ together with $y^{a_4/\kappa-1} e^{-a_3y} \leq C e^{-a_3y/2}$ and $(n_0-1)^\kappa \geq (v/2a_2)^\kappa$, this shows the right-hand side of \eqref{eq:part_gretaer_n_0} is at most $Ce^{-cuv^\kappa}$.

    We now consider the part $n < n_0$. 
    If $v/a_2 < 2$, then this part of sum is smaller than $1$ and therefore than $Ce^{-v}$ with $C = e^{2a_2}$.
    We now consider the case $v/a_2 \geq 2$. 
    Then, this part equals
    \begin{equation} \label{eq:part_smaller_n_0}
        \sum_{n=1}^{n_0-1} n^{a_4} e^{-a_1(v-a_2n)} e^{-a_3un^\kappa}
        \leq (v/a_2)^{a_4+1} e^{-a_1v} \max_{x \in [1,v/a_2]} e^{a_1a_2 x-a_3u x^\kappa},
    \end{equation}
    where we bounded $n_0-1 \leq v/a_2$, which holds because $v/a_2 \geq 2$ so $n_0 = \lceil v/a_2 \rceil$.
    Now, note that the function $x \mapsto a_1a_2 x-a_3u x^\kappa$ is convex on $[1,v/a_2]$, so it achieves its maximum on the boundary of the interval.
    This shows that the right-hand side of \eqref{eq:part_smaller_n_0} is at most
    \[
    C (v/a_2)^{a_4+1} \left( 
    e^{-a_1v} e^{-a_3u} + e^{-a_3u (v/a_2)^\kappa}
    \right) 
    \leq C \left( e^{-cv} + e^{-cuv^\kappa} \right),
    \]
    using here again that $u \geq 1$ to argue that $y^{a_4+1} e^{-a_3u y^\kappa} \leq C e^{-a_3u y^\kappa/2}$ for any $y > 0$. This yields \eqref{eq:new_goal} and concludes the proof
\end{proof}

\section{Brownian motion weighted by an integral via probability}
\label{sec:probabilistic_arguments}

In this section, we continue studying the heat kernel of Brownian motion weighted by an integral, relying now on probabilistic arguments. 
Our two main goals are to generalise estimates of the previous section to cases including error terms in the integral weight and to obtain sharper bounds on the tail of the kernel.

As in the previous section, we fix some parameters $\alpha \in (0,2)$ and $\beta>0$.
For $f \colon \R\times (0,\infty) \to \R$ measurable, we define the kernel $\widetilde{G}$ for $x,y \in \R$ and $0 \leq s < t$ by
\begin{equation} \label{eq:G_tilde}
\widetilde{G}(s,x;t,y) =  \frac{e^{-(y-x)^2/[2(t-s)]}}{\sqrt{2\pi(t-s)}}
\E_{(s,x)} \left[ \exp \left( - \beta \int_{s}^t \abs{\frac{B_r}{\sqrt{2} r}}^\alpha \left( 1+f(B_r,r) \right) \diff r \right) \middle| B_t = y \right].
\end{equation}
The contribution of $f$ has to be thought as an error term: note that the case $f=0$ corresponds exactly to the kernel $G$ studied in Section~\ref{sec:BM_weighted}.
More precisely, for some $L,a,b>0$, we work under the assumption $f^-_{L,a,b} \leq f \leq f^+_{L,a,b}$, where we set
\begin{equation} \label{eq:def_f_-_+}
f^+_{L,a,b}(y,r) = \left[ L \left( \abs{\frac{y}{r}}^a + r^{-b} \right) \right] \wedge 1 
\quad \text{and} \quad 
f^-_{L,a,b}(y,r) = - \left( \left[ L \left( \abs{\frac{y}{r}}^a + r^{-b} \right) \right] \wedge \zeta \right),
\end{equation}
where $\zeta = \zeta(L,a,b,\alpha)$ is chosen as the largest possible real number in $(0,1/2]$ such that, for any $r \geq r_0:= (2L)^{1/b}$, the function $y \in [0,\infty) \mapsto (y/r)^\alpha (1+f^-_{L,a,b}(y,r))$ is non-decreasing. 
To see that such an $\zeta$ necessarily exist, use that $Lr^{-b} \leq 1/2$ and note that the function $u \in [0,u_0] \mapsto u^\alpha(\frac{1}{2}-Lu^a)$ is non-decreasing if $u_0$ is chosen small enough depending on $\alpha,a,L$.
Finally, note that we write $f^-_{L,a,b}$ even if this function actually also depends on $\alpha$ through $\zeta$.

\subsection{Comparing \texorpdfstring{$\widetilde{G}$}{Gtilde} to \texorpdfstring{$G$}{G}}

Recall from Proposition~\ref{prop:estimate_G} that, around time $t$, the natural time scale of $G$ is $t^\kappa$ and the natural space scale is $t^{\kappa/2}$, where $\kappa \coloneqq 2\alpha/(2+\alpha) \in (0,1)$: as soon as $t-s$ is much larger than $t^\kappa$, the density $G(s,x;t,\cdot)$ is close to its equilibrium shape, which is given by $\varphi_0$ rescaled by $\vartheta_2/t^{\kappa/2}$. 
It is helpful to keep these scales in mind throughout the section.

The first result of this section establishes that, if $a$ and $b$ are large enough in terms of $\alpha$, then $\widetilde{G}$ and $G$ are of the same order as long as we are on a time scale longer than the natural one and on a space scale not much longer than the natural one.

\begin{prop}\label{prop:approximating_G_with_G}
	Let $\alpha \in (0,2)$ and $\beta > 0$.
	Let $L>0$, $a> (1-\kappa)/(1-\frac{\kappa}{2}) = 1-\alpha/2$ and $b>1-\kappa$.
	Then, there exists $\eps_0 = \eps_0(\alpha,a,b) > 0$ such that for all $\eps \in (0,\eps_0)$, there are $C,c>0$ such that for $s$ large enough, for any $t\geq s + s^\kappa$, any $x,y\in \R$ with $\abs{x} \leq s^{(\kappa + \eps)/2}$ and $\abs{y} \leq t^{(\kappa + \eps)/2}$ and any function $f$ satisfying $f^-_{L,a,b} \leq f \leq f^+_{L,a,b}$, we have
	\begin{equation*}
	\widetilde{G}(s,x;t,y) = (1+o(1)) G(s,x;t,y),
	\end{equation*}
	where the $o(1)$ holds as $s\to\infty$, uniformly in $t,x,y,f$ (but depending on $\alpha,\beta,L,a,b,\varepsilon$).
\end{prop}


\subsubsection{Preliminary results}

We prove here two preliminary lemmas. The first one gives some rough estimates on the natural time scale.

\begin{lem} \label{lem:step_t^kappa}
	Let $\alpha \in (0,2)$ and $\beta > 0$. 
	There exist $C,c>0$ such that, for any $s\geq 1$, $t \in [s + s^\kappa,s + 3 s^\kappa]$ and any measurable function $f \colon \R\times (0,\infty) \to \R$, the following holds.
	\begin{enumerate}
		\item\label{it:LB_step_t^kappa} If $f \leq 1$, then, for any $M \geq 1$, $\abs{x} \leq Ms^{\kappa/2}$ and $\abs{y} \leq Mt^{\kappa/2}$, 
		\[
		\widetilde{G}(s,x;t,y) \geq \frac{ce^{-CM^2}}{t^{\kappa/2}}.
		\]
		\item\label{it:UB_step_t^kappa} If $f \geq -1/2$, then, for any $x,y\in \R$ and $\eta \geq 0$,
		\begin{align}
		& \frac{e^{-(y-x)^2/[2(t-s)]}}{\sqrt{2\pi(t-s)}} \nonumber\\
		& \qquad \times \E_{(s,x)} \left[ \exp \left( - \beta \int_{s}^t \abs{\frac{B_r}{\sqrt{2} r}}^\alpha \left( 1+f(B_r,r) \right) \diff r \right)
		\ind{\exists r \in [s,t], \abs{B_r} \geq r^{(\kappa + \eta)/2}}
		\middle| B_t = y \right] \nonumber \\
		& \leq C \exp \left( - c t^{\eta \alpha/2}
		- c \left( \frac{\abs{x}}{s^{\kappa/2}} \right)^{\alpha} 
		- c \left( \frac{\abs{y}}{t^{\kappa/2}} \right)^{\alpha}
		\right). \label{eq:UB_step_t^kappa} 
		\end{align}		
	\end{enumerate}
\end{lem}

\begin{proof}
	Part \ref{it:LB_step_t^kappa}. We use that $(y-x)^2/(t-s) \leq (2Mt^{\kappa/2})^2/s^\kappa \leq C M^2$ and that $f \leq 1$ to get
	\begin{align}
	\widetilde{G}(s,x;t,y) 
	&\geq \frac{e^{-CM^2}}{\sqrt{2\pi(t-s)}} 
	\E_{(s,x)} \left[ \exp \left( - 2 \beta \int_{s}^t \abs{\frac{B_r}{\sqrt{2} r}}^\alpha \diff r \right) \middle| B_t = y \right] \nonumber \\
	& \geq \frac{c e^{-CM^2}}{t^{\kappa/2}} 
	\P_{(s,x)} \left( \forall r \in [s,t], \abs{B_r} \leq 3 M t^{\kappa/2}\middle| B_t = y \right), \label{eq:LB_beginning}
	\end{align}
	where we restricted ourselves to the event where the Brownian bridge is bounded by $3 M t^{\kappa/2}$ so that we can bound $\int_{s}^t \lvert \frac{B_r}{r} \rvert^\alpha \diff r \leq C M^\alpha t^{\kappa +\alpha\kappa/2-\alpha} = C M^\alpha$ because $\kappa +\alpha\kappa/2-\alpha=0$.
	Now, recall the following formula (see e.g.\@ the proof of Lemma~2 in \cite{Bra1978}): for any $0<s<t$ and any $x,y<K$,
	\begin{equation} \label{eq:bound_Brownian_bridge}
	\P_{(s,x)} \left( \exists r \in [s,t], B_r \geq K \middle| B_t = y \right)
	= \exp \left( -\frac{2(K-x)(K-y)}{t-s} \right).
	\end{equation}
	Using this (and its symmetric version for the event $\{\exists r \in [s,t], B_r \leq -K\}$), we get that, for $s,t,x,y$ satisfying the assumptions of Part \ref{it:LB_step_t^kappa},
	\begin{align*}
	\P_{(s,x)} \left( \forall r \in [s,t], \abs{B_r} \leq 3 M t^{\kappa/2} \middle| B_t = y \right) 
	&\geq 1 - 2 \exp \left( - \frac{2\cdot 2Mt^{\kappa/2} \cdot 2Mt^{\kappa/2}}{t-s} \right) \\
	&\geq 1 - 2e^{-8/3} > 0,
	\end{align*}
	where we used $t-s \leq 3s^\kappa \le 3t^\kappa$ and $M \geq 1$.
	Coming back to \eqref{eq:LB_beginning}, this yields the result.

	Part \ref{it:UB_step_t^kappa}. Using $f \geq -1/2$ and $t-s \leq 3s^\kappa$, the left-hand side of \eqref{eq:UB_step_t^kappa} is at most 
	\begin{equation}
	e^{-(y-x)^2/(6s^\kappa)}
	\E_{(s,x)} \left[ \exp \left( - \frac{\beta}{2} \int_s^t \abs{\frac{B_r}{\sqrt{2} r}}^\alpha \diff r \right)
	\ind{\exists r \in [s,t], \abs{B_r} \geq s^{(\kappa + \eta)/2}}
	\middle| B_t = y \right], \label{eq:UB_step_t^kappa_2} 
	\end{equation}	
	which we now aim at bounding.
	We distinguish several cases. First, assume that $\abs{x} \vee \abs{y} \leq s^{(\kappa+\eta)/2}/2$. Then, keeping only the event in the indicator function and then applying \eqref{eq:bound_Brownian_bridge}, we get that \eqref{eq:UB_step_t^kappa_2} is at most
	\[
		\P_{(s,x)} \left( \exists r \in [s,t], \abs{B_r} \geq s^{(\kappa + \eta)/2}
		\middle| B_t = y \right)
		\leq 2 \exp \left( -\frac{s^{\kappa + \eta}}{2(t-s)} \right) \leq 2 \exp \left( -\frac{s^{\eta}}{6} \right),
	\]
	which is smaller than the right-hand side of \eqref{eq:UB_step_t^kappa} in this case (note that $s \geq t/4$ and $\alpha<2$).
	We can now assume that $\abs{x} \vee \abs{y} > s^{(\kappa+\eta)/2}/2$. We restrict ourselves to the case where $\abs{x} = \abs{x} \vee \abs{y}$ and $x>0$, the other cases being treated similarly.
	On the one hand, if $y\leq x/2$, then, bounding the expectation by 1, \eqref{eq:UB_step_t^kappa_2} is at most 
	\[
	e^{-(y-x)^2/(6s^\kappa)}
	\leq e^{-x^2/(24s^\kappa)}
    \leq \exp \left( - \frac{1}{24 s^{\kappa}} 
    \left( \frac{(s^{(\kappa+\eta)/2}/2)^2}{3} + \frac{x^2}{3} + \frac{y^2}{3} \right) 
    \right),
 	\]
	which is smaller than the right-hand side of \eqref{eq:UB_step_t^kappa}. 
	On the other hand, if $y > x/2$, then \eqref{eq:UB_step_t^kappa_2} is at most
	\begin{align}
	 \E_{(s,x)}& \left[ \exp \left( - \frac{\beta}{2} \int_s^t \abs{\frac{B_r}{\sqrt{2} r}}^\alpha \diff r \right) \middle| B_t = y \right] \nonumber \\
	& = \E_{(s,x)} \left[ \exp \left( - \frac{\beta}{2} \int_s^t \abs{\frac{B_r}{\sqrt{2} r}}^\alpha \diff r \right)	\ind{\forall r \in [s,t], B_r \geq x/4} \middle| B_t = y \right] \\
        & \hspace{2cm}
	+ \P_{(s,x)} \left(	\exists r \in [s,t], B_r \leq x/4 \middle| B_t = y \right) \nonumber \\
	& \leq \exp \left( - \frac{\beta}{2} (t-s) \abs{\frac{x}{4\sqrt{2} s}}^\alpha \right)
	+ \exp \left( -\frac{2(x/2)(x/4)}{t-s} \right), \label{eq:UB_step_t^kappa_3}
	\end{align}	
	using \eqref{eq:bound_Brownian_bridge} for the second term. Using $s^\kappa \leq t-s \leq 3s^\kappa$ and $\alpha-\kappa = \alpha\kappa/2$, the right-hand side of \eqref{eq:UB_step_t^kappa_3} is at most $\exp(- c x^\alpha/s^{\alpha\kappa/2} ) + \exp \left( - cx^2/s^\kappa \right) \leq 2 \exp(- c x^\alpha/s^{\alpha\kappa/2})$, which is again smaller than the right-hand side of \eqref{eq:UB_step_t^kappa} using $\abs{x} \geq \abs{y}$,  $\abs{x} \geq s^{(\kappa+\eta)/2}/2$ and $s \geq t/4$. 
	This concludes the proof.
\end{proof}


Before stating the next lemma, we define, for $0 < s < t$ and $x,y\geq 0$,
\begin{equation} \label{eq:def_G_tilde_||}
\widetilde{G}^{|\cdot|}(s,x;t,y) \coloneqq \widetilde{G}(s,x;t,y)+\widetilde{G}(s,x;t,-y).
\end{equation}
It can be rewritten in the following way:
\begin{align} 
\widetilde{G}^{|\cdot|}(s,x;t,y)
& = \frac{e^{-(y-x)^2/[2(t-s)]}+e^{-(y+x)^2/[2(t-s)]}}{\sqrt{2\pi(t-s)}} \nonumber \\
& \qquad \times \E_{(s,x)} \left[ \exp \left( - \beta \int_{s}^t \abs{\frac{B_r}{\sqrt{2} r}}^\alpha \left( 1+f(B_r,r) \right) \diff r \right) \middle| \abs{B_t} = y \right]. \label{eq:formula_G_tilde_||}
\end{align}
If $f(\cdot,r)$ is even for any $r \geq s$, then $f(B_r,r)$ can be replaced by $f(\abs{B_r},r)$, and $\widetilde{G}^{|\cdot|}$ can be seen as the heat kernel of reflected Brownian motion weighted by an integral.
Moreover under this additional assumption, for any $0<s<t$ and $x,y\geq 0$, we have
\begin{equation} \label{eq:other_expression_G_tilde_||}
\widetilde{G}^{|\cdot|}(s,x;t,y) = \widetilde{G}(s,x;t,y)+\widetilde{G}(s,-x;t,y),
\end{equation}
because $\widetilde{G}(s,-x;t,y) = \widetilde{G}(s,x;t,-y)$ by symmetry of Brownian motion and of $f$.
It is convenient to work with this new kernel in the following result for coupling reasons which appear in its proof.

\begin{lem} \label{lem:coupling}
	Let $\alpha \in (0,2)$ and $\beta > 0$. 
	Let $f \colon \R\times (0,\infty) \to \R$ measurable and $r_0 >0$ be such that, for any $r \geq r_0$, the function $f(\cdot,r)$ is even and the function $y \in [0,\infty) \mapsto (y/r)^\alpha (1+f(y,r))$ is non-decreasing.
	For any $r_0 \leq s < t$, $0 \leq x' \leq x$ and $0 \leq y' \leq y$,
	\begin{align*}
		\widetilde{G}^{|\cdot|}(s,x;t,y) \leq 2 \exp \left( \frac{(x'-y')^2 - (x-y)^2}{2(t-s)}\right) \widetilde{G}^{|\cdot|}(s,x';t,y').
	\end{align*}
\end{lem}

\begin{proof}
	We use expression \eqref{eq:formula_G_tilde_||} to compare $\widetilde{G}^{|\cdot|}(s,x;t,y)$ and $\widetilde{G}^{|\cdot|}(s,x';t,y')$.
	For the prefactor, we bound crudely
	\begin{equation*} 
	\frac{e^{-(y-x)^2/[2(t-s)]}+e^{-(y+x)^2/[2(t-s)]}}{e^{-(y'-x')^2/[2(t-s)]}+e^{-(y'+x')^2/[2(t-s)]}}
	\leq \frac{2 e^{-(y-x)^2/[2(t-s)]}}{e^{-(y'-x')^2/[2(t-s)]}}.
	\end{equation*}
	For the expectation, one can couple%
    \footnote{To do so, consider $\widetilde{X}^2$ a reflected Brownian motion bridge from $(s,x')$ to $(t,y')$ independent of $X^1$. On the event where $X^1$ and $\widetilde{X}^2$ do not intersect, we set $X^2 = \widetilde{X}^2$. On the complement of this event, let $S$ be the first hitting time of $X^1$ and $\widetilde{X}^2$ and $L$ the last hitting time: we set $X^2 = \widetilde{X}^2$ on $[s,t] \setminus [S,L]$ and $X^2 = X^1$ on $[S,L]$. By continuity of the paths, this ensures that $X^1 \geq X^2$. The fact that $X^2$ is a reflected Brownian motion bridge from $(s,x')$ to $(t,y')$ follows from the backward strong Markov property stated in \cite[Theorem 2]{ChaUri2011}. This coupling is often referred to as a Doeblin coupling.
    \label{footnote:Doeblin}}
    a reflected Brownian motion bridge $X^1$ from $(s,x)$ to $(t,y)$ with a reflected Brownian motion bridge $X^2$ from $(s,x')$ to $(t,y')$ such that $X^1_r \geq X^2_r$ for any $r \in [s,t]$. 
	By assumptions on $f$, we then get
	\begin{equation*} 
	\exp \left( - \beta \int_{s}^t \abs{\frac{X^1_r}{\sqrt{2} r}}^\alpha \left( 1+f(X^1_r,r) \right) \diff r \right)
	\leq \exp \left( - \beta \int_{s}^t \abs{\frac{X^2_r}{\sqrt{2} r}}^\alpha \left( 1+f(X^2_r,r) \right) \diff r \right).
	\end{equation*}
	Combining this gives the desired inequality.
\end{proof}

\subsubsection{Localizing typical paths}

The main ingredient of the proof of Proposition \ref{prop:approximating_G_with_G} is the following result localizing the paths contributing to $\widetilde{G}(s,y;t,z)$.

\begin{lem} \label{lem:localization_Gtilde}
	Let $\alpha \in (0,2)$ and $\beta > 0$. 
	For any $\eps \in (0,1-\kappa]$ and $\eta > 2\eps/\alpha$, there are $C,c>0$ such that the following holds.
	Let $f \colon \R\times (0,\infty) \to [-1/2,1]$ measurable and $r_0>0$ be such that, for any $r \geq r_0$, the function $f(\cdot,r)$ is even and the function $y \in [0,\infty) \mapsto (y/r)^\alpha (1+f(y,r))$ is non-decreasing.
	For any $s \geq r_0$, $t\geq s + s^\kappa$, $\abs{x} \leq s^{(\kappa + \eps)/2}$ and $\abs{y} \leq t^{(\kappa + \eps)/2}$, 
	\begin{align}
	& \frac{e^{-(y-x)^2/[2(t-s)]}}{\sqrt{2\pi(t-s)}}
	\E_{(s,x)} \left[ \exp \left( - \beta \int_{s}^t \abs{\frac{B_r}{\sqrt{2} r}}^\alpha \left( 1+f(B_r,r) \right) \diff r \right)
	\ind{\exists r \in [s,t], \abs{B_r} \geq r^{(\kappa + \eta)/2}}
	\middle| B_t = y \right] \nonumber \\
	& \leq C \exp \left( - c s^{\eta \alpha/2} \right) 
	\widetilde{G}(s,x;t,y). \label{eq:localization_Gtilde}
	\end{align}
\end{lem}

\begin{proof}
	We divide $[s,t]$ into shorter intervals by choosing $s=s_0<s_1<\dots<s_n = t$ such that for any $0 \leq k \leq n-1$, $s_{k+1} \in [s_k+s_k^\kappa,s_k+3s_k^\kappa]$.
	By a union bound, the left-hand side of \eqref{eq:localization_Gtilde} is at most
	\begin{align*}
	& \sum_{k=0}^{n-1} \frac{e^{-(y-x)^2/[2(t-s)]}}{\sqrt{2\pi(t-s)}} \\
        & \hspace{1cm}\times
	\E_{(s,x)} \left[ \exp \left( - \beta \int_{s}^t \abs{\frac{B_r}{\sqrt{2} r}}^\alpha \left( 1+f(B_r,r) \right) \diff r \right)
	\ind{\exists r \in [s_k,s_{k+1}], \abs{B_r} \geq r^{(\kappa + \eta)/2}}
	\middle| B_t = y \right]
	\end{align*}
	and we now aim at proving, for any $0 \leq k \leq n-1$,
	\begin{align}
	& \frac{e^{-(y-x)^2/[2(t-s)]}}{\sqrt{2\pi(t-s)}} \nonumber \\
    & \hspace{1cm}\times
	\E_{(s,x)} \left[ \exp \left( - \beta \int_{s}^t \abs{\frac{B_r}{\sqrt{2} r}}^\alpha \left( 1+f(B_r,r) \right) \diff r \right)
	\ind{\exists r \in [s_k,s_{k+1}], \abs{B_r} \geq r^{(\kappa + \eta)/2}}
	\middle| B_t = y \right] \nonumber \\
	& \leq C \exp \left( - c s_k^{\eta \alpha/2} \right) 
	\widetilde{G}(s,x;t,y). \label{eq:new_goal_subinterval}
	\end{align}
	The desired result then follows by summing these inequalities and using that 
    \begin{equation} \label{eq:sum_s_k}
        \sum_{k=0}^{n-1} \exp \left( - c s_k^{\eta \alpha/2} \right)  
        \le C \exp \left( - c s^{\eta \alpha/2} \right).
    \end{equation}
    Indeed, we have $s_{k+1} \geq s_k+s_k^\kappa \geq s_k + s^\kappa$, so $s_k \geq s + k s^\kappa$, and using this together with a sum-integral comparison proves \eqref{eq:sum_s_k} (up to a polynomial prefactor in $s$ which can be absorbed in the exponential by modifying $c$). 
    Hence, it remains to prove \eqref{eq:new_goal_subinterval}.
	
	We start with the case $1 \leq k \leq n-2$.
	Applying Markov's property at times $s_k$ and $s_{k+1}$ and then Lemma~\ref{lem:step_t^kappa}.\ref{it:LB_step_t^kappa} (with $M=1$) after having restricted the range of integration, we get 
	\begin{align}
	\widetilde{G}(s,x;t,y) 
	& = \int_{\R} \int_{\R}	\widetilde{G}(s,x;s_k,x_k) \widetilde{G}(s_k,x_k;s_{k+1},x_{k+1}) 
	\widetilde{G}(s_{k+1},x_{k+1};t,y) \diff x_k \diff x_{k+1} \nonumber \\
	& \geq \frac{c}{s_{k+1}^{\kappa/2}} \int_{\abs{x_{k+1}} \leq s_{k+1}^{\kappa/2}} \int_{\abs{x_k} \leq s_k^{\kappa/2}}
	\widetilde{G}(s,x;s_k,x_k) 
	\widetilde{G}(s_{k+1},x_{k+1};t,y) \diff x_k \diff x_{k+1} \nonumber \\
	& = \frac{c}{s_{k+1}^{\kappa/2}} 
	\Biggl( \int_0^{s_k^{\kappa/2}} \widetilde{G}^{|\cdot|}(s,x;s_k,x_k) \diff x_k \Biggr)
	\Biggl( \int_0^{s_{k+1}^{\kappa/2}} \widetilde{G}^{|\cdot|}(s_{k+1},x_{k+1};t,y) \diff x_{k+1} \Biggr), \label{eq:localization_Gtilde_1}
	\end{align}
	using the definition of $\widetilde{G}^{|\cdot|}$ in \eqref{eq:def_G_tilde_||} for the second integral in the last line, as well as \eqref{eq:other_expression_G_tilde_||} for the first integral.
	On the other hand, proceeding similarly but with Lemma~\ref{lem:step_t^kappa}.\ref{it:UB_step_t^kappa}, the left-hand side of \eqref{eq:new_goal_subinterval} is at most
	\begin{align}
	& C \exp \left( - c s_{k+1}^{\eta \alpha/2} \right) 
	\Biggl( \int_0^\infty 
	\widetilde{G}^{|\cdot|}(s,x;s_k,x_k) 
	\exp \Biggl( - c \left( \frac{\abs{x_k}}{s_k^{\kappa/2}} \right)^{\alpha} \Biggr)
	\diff x_k \Biggr) \nonumber \\
	& \hspace{2cm} \times 
	\Biggl( \int_0^\infty
	\widetilde{G}^{|\cdot|}(s_{k+1},x_{k+1};t,y)
	\exp \Biggl( - c \Biggl( \frac{\abs{x_{k+1}}}{s_{k+1}^{\kappa/2}} \Biggr)^{\alpha} \Biggr)
	\diff x_{k+1} \Biggr). \label{eq:localization_Gtilde_2}
	\end{align}
	Then cutting $[0,\infty)$ into intervals of length $s_k^{\kappa/2}$, we get
	\begin{align*}
	& \int_0^\infty 
	\widetilde{G}^{|\cdot|}(s,x;s_k,x_k) 
	\exp \Biggl( - c \left( \frac{\abs{x_k}}{s_k^{\kappa/2}} \right)^{\alpha} \Biggr)
	\diff x_k \\
	& \leq \sum_{\ell \geq 0} e^{- c \ell^\alpha}
	\int_{\ell s_k^{\kappa/2}}^{(\ell +1)s_k^{\kappa/2}}
	\widetilde{G}^{|\cdot|}(s,x;s_k,x_k) 
	\diff x_k \\
	& \leq C
	\int_{0}^{s_k^{\kappa/2}}
	\exp \left( \frac{(x-x_k')^2}{2(s_k-s)}\right)
	\widetilde{G}^{|\cdot|}(s,x;s_k,x_k') 
	\diff x_k',
	\end{align*}
	where we used Lemma \ref{lem:coupling} to compare $\widetilde{G}^{|\cdot|}(s,x;s_k,x_k)$ and $\widetilde{G}^{|\cdot|}(s,x;s_k,x_k')$ with $x_k' = x_k - \ell s_k^{\kappa/2}$ and bounded $\exp \{ (x-x'_k)^2-(x-x_k)^2/(2(s_k-s)) \}\le \exp \{ (x-x'_k)^2/(2(s_k-s)) \}$.
	Using here that $\abs{x} \leq s^{(\kappa + \eps)/2}$, we have $(x-x_k')^2 \leq 4s_k^{\kappa + \eps}$. Moreover, we bound $s_k-s \geq s_{k}-s_{k-1}\ge s_{k-1}^\kappa \ge (s_k/4)^\kappa$ and therefore we get
	\begin{equation} \label{eq:localization_Gtilde_3}
	\int_0^\infty 
	\widetilde{G}^{|\cdot|}(s,x;s_k,x_k) 
	\exp \Biggl( - c \left( \frac{\abs{x_k}}{s_k^{\kappa/2}} \right)^{\alpha} \Biggr)
	\diff x_k
	\leq C \exp \left( C s_k^\eps \right)
	\int_{0}^{s_k^{\kappa/2}} \widetilde{G}^{|\cdot|}(s,x;s_k,x_k) \diff x_k.
	\end{equation}
	Proceeding similarly, the second integral in \eqref{eq:localization_Gtilde_2} is at most 
    \begin{align} 
	& C \int_0^{s_{k+1}^{\kappa/2}} 
    \exp \left( \frac{(x_{k+1}-y)^2}{2(t-s_{k+1})} \right)
    \widetilde{G}^{|\cdot|}(s_{k+1},x_{k+1};t,y) \diff x_{k+1} \nonumber \\
    & \leq C \exp \left( \frac{2t^{\kappa+\eps}}{t-s_{k+1}} \right)
	\int_0^{s_{k+1}^{\kappa/2}} \widetilde{G}^{|\cdot|}(s_{k+1},x_{k+1};t,y) \diff x_{k+1} \nonumber \\
	& \leq C \exp \left(C s_{k+1}^\eps\right)
	\int_0^{s_{k+1}^{\kappa/2}} \widetilde{G}^{|\cdot|}(s_{k+1},x_{k+1};t,y) \diff x_{k+1}, 
	\label{eq:localization_Gtilde_4}
	\end{align}
    where, in the first inequality, we bound $(x'_{k+1}-y)^2 \le 4t^{\kappa+\epsilon}$ and, in the second one, we note that, if $s_{k+1} \geq t/2$, then $t^{\kappa+\eps}/(t-s_{k+1}) \leq (2s_{k+1})^{\kappa+\eps}/s_{k+1}^\kappa \leq 2 s_{k+1}^\eps$ and, if $s_{k+1} \leq t/2$, then $t^{\kappa+\eps}/(t-s_{k+1}) \leq 2t^{\kappa+\eps-1} \leq 2$ because $\eps \leq 1-\kappa$.
	Combining \eqref{eq:localization_Gtilde_1}, \eqref{eq:localization_Gtilde_2}, \eqref{eq:localization_Gtilde_3} and \eqref{eq:localization_Gtilde_4} proves \eqref{eq:new_goal_subinterval}, noting that the factors $s_{k+1}^{\kappa/2}$ and  $\exp(C s_{k+1}^\eps)$ can be absorbed in $\exp(-c s_{k+1}^{\eta \alpha/2})$ up to a modification of $c$ (recall that $\eta > 2\eps/\alpha$).
	
	We now consider the case $k=0$. Proceeding as in \eqref{eq:localization_Gtilde_1} and \eqref{eq:localization_Gtilde_2} but applying Markov's property only at time $s_{k+1} = s_1$ yields
	\begin{align}
	\widetilde{G}(s,x;t,y) 
	& \geq \frac{c}{s_1^{\kappa/2}} 
	\int_0^{s_1^{\kappa/2}} \widetilde{G}^{|\cdot|}(s_1,x_1;t,y) \diff x_1, \label{eq:localization_Gtilde_1_k=0}
	\end{align}
	and shows that the left-hand side of \eqref{eq:new_goal_subinterval} is at most
	\begin{align}
	& C \exp \left( - c s_1^{\eta \alpha/2} \right) 
	\int_0^\infty \widetilde{G}^{|\cdot|}(s_1,x_1;t,y)
	\exp \Biggl( - c \Biggl( \frac{\abs{x_1}}{s_1^{\kappa/2}} \Biggr)^{\alpha} \Biggr)
	\diff x_1. \label{eq:localization_Gtilde_2_k=0}
	\end{align}
	Then proceeding as in \eqref{eq:localization_Gtilde_4} proves \eqref{eq:new_goal_subinterval}.
	The case $k=n-1$ is covered similarly, but applying Markov's property only at time $s_k = s_{n-1}$.
	Finally note that this argument for cases $k=0$ and $k=n-1$ works only if $0<n-1$. But in the case $n=1$,  \eqref{eq:localization_Gtilde} is a direct consequence of Lemma~\ref{lem:step_t^kappa}.\ref{it:LB_step_t^kappa} (with $M=t^\eps$) and Lemma~\ref{lem:step_t^kappa}.\ref{it:UB_step_t^kappa}.
\end{proof}

\subsubsection{Proof of Proposition \ref{prop:approximating_G_with_G}}

\begin{proof}[Proof of Proposition \ref{prop:approximating_G_with_G}]
	We first note that increasing values of $f$ results in decreasing~$\widetilde{G}$. 
	Therefore, it is enough to prove the result for $f = f^-_{L,a,b}$ and $f=f^+_{L,a,b}$ defined in \eqref{eq:def_f_-_+}. 
	We assume we are in one of these two cases.
	Then, the function $f$ satisfies the assumptions of Lemma~\ref{lem:localization_Gtilde}: for $f=f^+_{L,a,b}$ this is direct with $r_0=1$ for example, and for $f^-_{L,a,b}$ this follows from our choice of $\zeta$ in its definition, with $r_0 = (2L)^{1/b}$.
	Therefore, if $\eps \leq 1-\kappa$ and $\eta > 2\eps/\alpha$, we get
	\begin{align*}
	\widetilde{G}(s,x;t,y)
	& = (1+o(1)) 
	\frac{e^{-(y-x)^2/[2(t-s)]}}{\sqrt{2\pi(t-s)}} \\
	& \quad \times \E_{(s,x)} \left[ \exp \left( - \beta \int_s^t \abs{\frac{B_r}{\sqrt{2} r}}^\alpha \left( 1+f(B_r,r) \right) \diff r \right)
	\ind{\forall r \in [s,t], \abs{B_r} < r^{(\kappa + \eta)/2}}
	\middle| B_t = y \right],
	\end{align*}
	where the $o(1)$ holds as $s\to\infty$, uniformly in $t,x,y,f$ (as required in the statement of the proposition).
	On the event $\{\forall r \in [s,t], \abs{B_r} < r^{(\kappa + \eta)/2}\}$, using also $\abs{f(y,r)} \leq L ( \rvert \frac{y}{r} \lvert^a + r^{-b})$, we can bound
	\begin{align}
	\int_s^t \abs{\frac{B_r}{r}}^\alpha f(B_r,r) \diff r 
	& \leq L \int_s^t r^{\alpha(\kappa + \eta)/2-\alpha} \left( r^{a(\kappa + \eta)/2-a} + r^{-b} \right) \diff r \nonumber  \\
	& = L \int_s^t \left( r^{-\kappa - a(1-\kappa/2) +(\alpha+a) \eta/2} + r^{-\kappa-b+\alpha \eta/2} \right) \diff r, \label{eq:integral_error_term}
	\end{align}
	using in particular that $\alpha\kappa/2 - \alpha = -\kappa$.
	The two exponent on the right-hand side of \eqref{eq:integral_error_term} can be made smaller than $-1$ by choosing
	\begin{equation}
	\eta < \eta_0 \coloneqq 2 \left( \frac{a(1-\frac{\kappa}{2})-(1-\kappa)}{a+\alpha}
	\wedge \frac{b-(1-\kappa)}{\alpha} \right).
	\end{equation}
	By our assumptions on $a$ and $b$, we have $\eta_0> 0$, so if we choose $\eps < \alpha \eta_0 / 2$, it is possible to choose such an $\eta$ which also satisfies the previous condition $\eta > 2\eps/\alpha$.
	Then, the right-hand side of \eqref{eq:integral_error_term} is a $o(1)$ in the same sense as before. Hence, we get that $\widetilde{G}(s,x;t,y)$ equals
	\begin{align*}
	& (1+o(1)) 
	\frac{e^{-(y-x)^2/[2(t-s)]}}{\sqrt{2\pi(t-s)}} \E_{(s,x)} \left[ \exp \left( - \beta \int_s^t \abs{\frac{B_r}{\sqrt{2} r}}^\alpha \diff r \right)
	\ind{\forall r \in [s,t], \abs{B_r} < r^{(\kappa + \eta)/2}}
	\middle| B_t = y \right] \\
	& = (1+o(1)) G(s,x;t,y),
	\end{align*}
	where the last equality follows from Lemma \ref{lem:localization_Gtilde} but with $f=0$ (for which $\widetilde{G}=G$). This concludes the proof. 
\end{proof}

\subsection{Bounding the total mass of \texorpdfstring{$\widetilde{G}$}{Gtilde}}

We prove here the following result, where we bound the total mass of $\widetilde{G}$, which can be written as
\begin{equation}\label{eq:formula_integrated_G_tilde}
	\int_\R \widetilde{G}(s,x;t,y) \diff y
	= \E_{(s,x)} \left[ \exp \left( - \beta \int_{s}^t \abs{\frac{B_r}{\sqrt{2} r}}^\alpha \left( 1+f(B_r,r) \right) \diff r \right) \right].
\end{equation}
This is analogous to Corollary \ref{cor:estimate_G} for $G$.

\begin{prop}\label{prop:bound_on_integrated_G_tilde}
	Let $\alpha \in (0,2)$ and $\beta>0$.
    Let $\vartheta_1 >0$ be defined as in Proposition \ref{prop:estimate_G}.
	Let $L>0$, $a> (1-\kappa)/(1-\frac{\kappa}{2})$ and $b>1-\kappa$.
    Let $\eta > 0$.
	There exist $K,C,c > 0$ such that the following holds for any $s \geq K$ and $t \geq s + Ks^\kappa$.
	\begin{enumerate}
		\item\label{it:UB_integrated_G_tilde} For any $x \in \R$ and any function $f$ satisfying $f \geq f^-_{L,a,b}$,
		\begin{equation}\label{eq:UB_integrated_G_tilde}
		\int_\R \widetilde{G}(s,x;t,y) \diff y
		\leq C (t/s)^{\kappa/4} \exp\left(\vartheta_1 (s^{1-\kappa}-t^{1-\kappa})\right).
		\end{equation}
		\item\label{it:LB_integrated_G_tilde} For any $\abs{x} \leq s^{\kappa/2}$ and any function $f$ satisfying $f \leq f^+_{L,a,b}$,
		\begin{align}
		& \E_{(s,x)} \left[ \exp \left( - \beta \int_{s}^t \abs{\frac{B_r}{\sqrt{2} r}}^\alpha \left( 1+f(B_r,r) \right) \diff r \right) 
        \ind{\abs{B_t} \leq t^{\kappa/2}}
        \ind{\forall r \in [s,t], \abs{B_r} \leq r^{(\kappa+\eta)/2}} 
        \right] \nonumber  \\
		& \geq c (t/s)^{\kappa/4} \exp\left(\vartheta_1 (s^{1-\kappa}-t^{1-\kappa})\right). 
        \label{eq:LB_integrated_G_tilde}
		\end{align}
	\end{enumerate}
\end{prop}

\begin{proof}
		Part \ref{it:UB_integrated_G_tilde}. By monotonicity, it is enough to prove the result for $f=f^-_{L,a,b}$. Then, $f(\cdot,r)$ is even so we can assume $x \geq 0$ by symmetry and write $f(B_r,r) = f(\abs{B_r},r)$ in \eqref{eq:formula_integrated_G_tilde} so that the expectation on the right-hand side of \eqref{eq:formula_integrated_G_tilde} can be seen as depending only on a reflected Brownian motion starting at $(s,x)$.
	But one can couple%
    \footnote{This is again a Doeblin coupling as in Footnote \ref{footnote:Doeblin}, but in the simpler context of Markov processes instead of Markov bridges. Consider $\widetilde{X}^2$ a reflected Brownian motion starting at $(s,0)$ independent of $X^1$, and define $X^2$ as $\widetilde{X}^2$ up to the first hitting time of $X^1$ and $\widetilde{X}^2$, and then as $X^1$. Here the fact that $X^2$ is a reflected Brownian motion starting at $(s,0)$ follows from the usual strong Markov property.}
    a reflected Brownian motion $X^1$ starting at $(s,x)$ with a reflected Brownian motion $X^2$ starting at $(s,0)$ such that $X^1_r \geq X^2_r$ for any $r \geq s$.
	Recalling that the function $y \in [0,\infty) \mapsto (y/r)^\alpha (1+f(y,r))$ is non-decreasing (by definition of $f^-_{L,a,b}$), this shows the expectation is maximal for $x=0$ so we can focus on this case.
	
	Now, we decompose
	\begin{equation} \label{eq:decompo}
	\int_\R \widetilde{G}(s,0;t,y) \diff y
	= \int_{\abs{y} \leq t^{\kappa/2}} \widetilde{G}(s,0;t,y) \diff y
	+ \int_{\abs{y} \geq t^{\kappa/2}} \widetilde{G}(s,0;t,y) \diff y.
	\end{equation}
	In the first term in \eqref{eq:decompo}, we can bound $\widetilde{G}(s,0;t,y)$ by $2G(s,0;t,y)$ by Proposition \ref{prop:approximating_G_with_G} and then get the desired bound by Corollary \ref{cor:estimate_G}.\ref{it:UB_integrated_G}.
	For the second term in \eqref{eq:decompo}, applying Markov's property at time $r \in (s,t)$ such that $r+r^\kappa = t$, it is at most
	\begin{align} 
	\int_{w \in \R} \widetilde{G}(s,0;r,w) 
	& \Biggl( 
	\int_{\abs{y} \geq t^{\kappa/2}}  \widetilde{G}(r,w;t,y) \diff y \Biggr) \diff w \nonumber \\
	& \leq C 
	\int_{w \in \R} \widetilde{G}(s,0;r,w) \exp \left(- c \left( \frac{\abs{w}}{r^{\kappa/2}} \right)^{\alpha} \right) \diff w, \label{eq:bound_2nd_integral}
	\end{align}
	by Lemma \ref{lem:step_t^kappa}.\ref{it:UB_step_t^kappa} (with $\eta = 0$, note that the indicator is automatically satisfied because $\abs{y} \geq t^{\kappa/2}$).
	Then, by definition of $\widetilde{G}^{|\cdot|}$ and cutting the integral, the right-hand side of \eqref{eq:bound_2nd_integral} is at most
	\begin{equation}
	C \sum_{\ell \geq 0} e^{-c\ell^\alpha}
	\int_{\ell r^{\kappa/2}}^{(\ell+1) r^{\kappa/2}}
	\widetilde{G}^{|\cdot|}(s,0;r,w) \diff w
	\leq C \int_0^{r^{\kappa/2}} \widetilde{G}^{|\cdot|}(s,0;r,w) \diff w,
	\label{eq:bound_2nd_integral_2}
	\end{equation}
	using that $\widetilde{G}^{|\cdot|}(s,0;r,w) \leq 2 \widetilde{G}^{|\cdot|}(s,0;r,w')$ for any $0 \leq w' \leq w$ by Lemma \ref{lem:coupling}.
	Finally, on the right-hand side of \eqref{eq:bound_2nd_integral_2}, we can bound $\widetilde{G}^{|\cdot|}(s,0;r,w)$ by $2G^{|\cdot|}(s,0;r,w)$ by Proposition \ref{prop:approximating_G_with_G} and then get the desired bound by Corollary \ref{cor:estimate_G}.\ref{it:UB_integrated_G} (note that $r^{1-\kappa} \geq t^{1-\kappa} -1$). Note that we have to take the constant $K$ larger than in Corollary \ref{cor:estimate_G} to ensure that we can apply Corollary \ref{cor:estimate_G} between times $s$ and $r$.
	
	Part \ref{it:LB_integrated_G_tilde}. By monotonicity again, it is enough to prove the result for $f=f^+_{L,a,b}$. 
    We first claim that, for any $\abs{x} \leq s^{\kappa/2}$,
	\begin{equation}\label{eq:LB_integrated_G_tilde2}
		\int_{\abs{y} \leq t^{\kappa/2}} \widetilde{G}(s,x;t,y) \diff y
		\geq c (t/s)^{\kappa/4} \exp\left(\vartheta_1 (s^{1-\kappa}-t^{1-\kappa})\right).
	\end{equation}
    Indeed, we can apply Proposition \ref{prop:approximating_G_with_G} to write $\widetilde{G}(s,x;t,y) \geq G(s,x;t,y)/2$ on the right-hand side, and then Corollary \ref{cor:estimate_G}.\ref{it:LB_integrated_G} concludes the proof of \eqref{eq:LB_integrated_G_tilde2}.
    Then, note that the difference between the left-hand side of \eqref{eq:LB_integrated_G_tilde2} and the left-hand side of \eqref{eq:LB_integrated_G_tilde} equals
    \begin{align*}
		& \E_{(s,x)} \left[ \exp \left( - \beta \int_{s}^t \abs{\frac{B_r}{\sqrt{2} r}}^\alpha \left( 1+f(B_r,r) \right) \diff r \right) 
        \ind{\abs{B_t} \leq t^{\kappa/2}}
        \ind{\exists r \in [s,t], \abs{B_r} > r^{(\kappa+\eta)/2}} 
        \right] \nonumber  \\
		& \leq C \exp \left( - c s^{\eta \alpha/2} \right) 
	    \int_{\abs{y} \leq t^{\kappa/2}} \widetilde{G}(s,x;t,y) \diff y
	\end{align*}
    by Lemma \ref{lem:localization_Gtilde}. Choosing $K$ large enough the prefactor in front of the last integral is less than $1/2$ and we get the desired result.
\end{proof}
\subsection{Bounding the tail of \texorpdfstring{$\widetilde{G}$}{Gtilde}}

We prove here the following bound on the tail of $\widetilde{G}(s,x;t,y)$ for $y$ in some time-dependent window. 
	
\begin{prop} \label{prop:tail_Gtilde} 
	Let $\alpha \in (0,2)$ and $\beta>0$.
	Let $L>0$, $a> (1-\kappa)/(1-\frac{\kappa}{2})$ and $b>1-\kappa$.
	There exist $C,c > 0$ such that, for $s$ large enough, for any $t\geq 2s$, $x\in\R$, $\abs{y} \leq t$, and any function $f$ satisfying $f \geq f^-_{L,a,b}$,
	\[
	\widetilde{G}(s,x;t,y) 
	\leq \frac{C}{(st)^{\kappa/4}} \exp\left(\vartheta_1 (s^{1-\kappa}-t^{1-\kappa})\right) 
	\exp\left(-c \left( \frac{\abs{y}}{t^{\kappa/2}} \right)^{(2+\alpha)/2} \right).
	\]
\end{prop}

Note that, up to constant factors, this bound is better than the upper bound given by Proposition \ref{prop:estimate_G} for $G$: the exponent for the tail in $y$ would be $\alpha$ there (because of the error term), whereas it is $(2+\alpha)/2$ here. 
Moreover, $(2+\alpha)/2$ is the exponent appearing in the tail of $\varphi_0$ (see Proposition \ref{prop:sturm-liouville}.\ref{it:tail}) and it reappears here via a probabilistic argument.

\begin{proof}
	By monotonicity, it is enough to deal with the case $f = f^-_{L,a,b}$. 
	Then, by symmetry of $f(\cdot,r)$, we can assume w.l.o.g.\@ that $y\geq 0$.
	Let $r \in [3t/4,t-t^\kappa]$ which will be chosen in terms of $y$ later.
	Applying Markov's property at time $r$, we get
	\begin{equation} \label{eq:proof_bound_Gtilde_1}
	\widetilde{G}(s,x;t,y) 
	= \int_{\R} \widetilde{G}(s,x;r,w) \widetilde{G}(r,w;t,y) \diff w.
	\end{equation} 
	To bound $\widetilde{G}(r,w;t,y)$, we use $f \geq -1/2$ and then distinguish according to whether the event $E = \{ \forall q \in [r,t], B_q \geq y/2 \}$ holds or not: this yields
	\begin{align*}
	\widetilde{G}(r,w;t,y)
	& \leq \frac{e^{-(y-w)^2/[2(t-r)]}}{\sqrt{2\pi(t-r)}}
	\E_{(r,w)} \left[ \exp \left( - \frac{\beta}{2} \int_r^t \abs{\frac{B_q}{\sqrt{2} q}}^\alpha \diff q \right)
	\middle| B_t = y \right]  \\
	& \leq \frac{e^{-(y-w)^2/[2(t-r)]}}{\sqrt{2\pi(t-r)}}
	\left( \exp \left( - \frac{\beta}{2} \int_r^t \abs{\frac{y/2}{\sqrt{2} q}}^\alpha \diff q \right)
	+ \P_{(r,w)} \left( E^c \middle| B_t = y \right) \right).
	\end{align*}
	If $w \geq 3y/4$, then $\P_{(r,w)}(E^c|B_t = y) \leq e^{-y^2/[4(t-r)]}$ by \eqref{eq:bound_Brownian_bridge}. On the other hand, if $w < 3y/4$ then $e^{-(y-w)^2/[2(t-r)]} \leq e^{-y^2/[8(t-r)]}$.
	Combining this and using $q\geq r \geq 3t/4$ yields
	\begin{equation} \label{eq:proof_bound_Gtilde_2}
	\widetilde{G}(r,w;t,y)
	\leq \frac{1}{\sqrt{t-r}}
	\left( e^{-c(t-r)y^\alpha/t^\alpha} + e^{-y^2/[8(t-r)]} \right).
	\end{equation}
	Note that the right-hand side is optimized when $r$ is chosen such that $t-r$ is of the same order as $y^{(2-\alpha)/2} t^{\alpha/2} = (y/t^{\kappa/2})^{(2-\alpha)/2} t^\kappa$ recalling $\kappa = 2\alpha/(2+\alpha)$.
	This motivates the following choice, which also fulfils the constraint $r \in [3t/4,t-t^\kappa]$ (recall $y \in [0,t]$):
	\[
		t-r \coloneqq 
		\left( \frac{1}{4} \left( \frac{y}{t^{\kappa/2}} \right)^{(2-\alpha)/2} \vee 1 \right) t^\kappa.
	\]
	With this choice of $r$, using that $\sqrt{t-r} \geq t^{\kappa/2}$, \eqref{eq:proof_bound_Gtilde_2} becomes 
	\begin{equation} \label{eq:proof_bound_Gtilde_3}
	\widetilde{G}(r,w;t,y)
	\leq \frac{C}{t^{\kappa/2}} \exp\left(-c \left( \frac{y}{t^{\kappa/2}} \right)^{(2+\alpha)/2} \right),
	\end{equation}
	where we used in the case $t-r =t^\kappa$ (or equivalently $y \leq 4^{2/(2-\alpha)} t^{\kappa/2}$) that the exponential on the right-hand side is lower bounded by a positive constant.
	Plugging this into \eqref{eq:proof_bound_Gtilde_1} and applying Proposition \ref{prop:bound_on_integrated_G_tilde}.\ref{it:UB_integrated_G} yields
	\begin{equation} \label{eq:proof_bound_Gtilde_4}
	\widetilde{G}(s,x;t,y) 
	\leq C (r/s)^{\kappa/4} \exp\left(\vartheta_1 (s^{1-\kappa}-r^{1-\kappa})\right) 
	\cdot \frac{1}{t^{\kappa/2}} \exp\left(-c \left( \frac{y}{t^{\kappa/2}} \right)^{(2+\alpha)/2} \right).
	\end{equation}
	Finally, using that $(1-x)^{1-\kappa} \geq 1-Cx$ for $x \in [0,1/4]$, we get
	\begin{equation*} 
	t^{1-\kappa}-r^{1-\kappa}
	= t^{1-\kappa} \left( 1-  \left( 1 - \frac{t-r}{t} \right)^{1-\kappa}\right)
	\leq C t^{1-\kappa} \frac{t-r}{t}
	= C \left( \frac{1}{4} \left( \frac{y}{t^{\kappa/2}} \right)^{(2-\alpha)/2} \vee 1 \right).
	\end{equation*}
	Note that the exponent $(2-\alpha)/2$ appearing here is smaller than the exponent $(2+\alpha)/2$ appearing in the last exponential in \eqref{eq:proof_bound_Gtilde_4}.
	Therefore, when $y \geq Kt^{\kappa/2}$ with $K$ large enough chosen in terms of the previous constants $c$ and $C$, the error made when replacing $r^{1-\kappa}$ by $t^{1-\kappa}$ on the right-hand side of \eqref{eq:proof_bound_Gtilde_4} can be included in the last exponential factor (up to replacing $c$ by $c/2$). 
	On the other hand, when $y \leq Kt^{\kappa/2}$, the error can simply be bounded by a constant factor.
	This gives the desired result.
\end{proof}

\section{Many-to-few lemmas}
\label{sec:many-to-few-lemmas}
Two typical tools from the study of branching processes are the many-to-one and many-to-two lemmas. They allow us to reduce the certain expectations of the branching process to expectations of just one or two particles. 


%


\begin{lem}[Many-to-one] \label{lem:many_to_one}
	Let $t \geq s_0 \geq 0$ and $f \colon \cC([s_0,t],\R^2) \to [0,\infty)$ be a measurable functional. Then, for any $x,y \in \R^2$,
	\begin{align*}
	    & \E_{(s_0,x,y)} \left[\sum_{u \in \mathcal{N}(t)} f\left((X_u(s),Y_u(s))_{s\in [s_0,t]} \right) \right] \\
	    & = \E_{(s_0,x,y)} \left[f\left((X_s,Y_s)_{s\in [s_0,t]} \right) \exp\left( \int_{s_0}^t b(\theta_s) \diff s \right) \right],
	\end{align*}
	where $(X_s,Y_s)_{s\in [s_0,t]}$ is a Brownian motion on $\R^2$ starting from $(x,y)$ at time $s_0$ under $\P_{(s_0,x,y)}$, and $(R_s,\theta_s)_{s\in [s_0,t]}$ is a representation of its polar coordinates. 
\end{lem}

For a proof see \cite[Theorem 8.5]{hardy_spine_2009}, they deal with the one-dimensional case but the proof does not change. Similarly, we need a version of the many-to-two lemma for our process. Again we do not prove this as a proof is messy, a more general result is the many-to-few lemma \cite[Lemma 1]{harris_many--few_2017} which also incorporates the inhomogeneous branching. To state the many-to-two lemma, we need to introduce two random processes $\xi^{1}$ and $\xi^{2,r}$.
Under $\P_{(s_0,x,y)}$, $\xi^{1}$ is a Brownian motion on $\R^2$  starting from $(x,y)$ at time $s_0$, and $\xi^{2,r}$ coincide with $\xi^{1}$ on $[s_0,r]$ and then move as a Brownian motion independently of $\xi^{1}$.

\begin{lem} [Many-to-two] \label{lem:many_to_two}
	Let $t \geq s_0 \geq 0$ and $f,g \colon \cC([s_0,t],\R^2) \to [0,\infty)$ be measurable functionals. Then, for any $x,y \in \R^2$,
	\begin{align}
	    &\E_{(s_0,x,y)} \left[ \sum_{u,v \in \mathcal{N}(t), u \neq v} f\left((X_u(s),Y_u(s))_{s\in [s_0,t]} \right) g\left((X_v(s),Y_v(s))_{s\in [s_0,t]} \right)\right] \nonumber \\ 
	    &= \int_{s_0}^t \E_{(s_0,x,y)} \left[ f\left((\xi^1_s)_{s\in [s_0,t]} \right) g\left((\xi^{2,r}_s)_{s\in [s_0,t]} \right)
	    b(\xi_r^1)
	    \exp\left( \int_{s_0}^t b(\xi_s^1) \diff s 
	    + \int_r^t b(\xi_s^{2,r}) \diff s \right) \right] 2\diff r,
	    \label{eq:many-to-two}
	\end{align}
	where, for $\xi \in \R^2$, we write $b(\xi) = b(\theta)$ where $\theta$ is the angular coordinate of $\xi$.
\end{lem}

\begin{proof}
	We prove the case $s_0 = 0$, the general case is a consequence by a time shift.
    We derive this expression from \cite[Lemma 1]{harris_many--few_2017}. There, we have a stopping time $T$ which is given by 
    \begin{equation*}
        \inf \left\{t\geq 0: \int_0^t b(\xi^1_s) \diff s > \hat{T} \right\},
    \end{equation*}
    where $\hat{T}$ is an exponential random variable of rate $2$ which is independent of $\xi^1$ and $\xi^2$. Observe that conditional on $(\xi_s^1)_{s \geq 0}$ we have
    \begin{equation*}
        \P(T > r \vert \xi^1) = \exp\left( -2\int_0^r b(\xi_s^1) \diff s \right),
    \end{equation*}
    and therefore
    \begin{equation*}
        \P(T \in \diff r \vert \xi^1 ) = 2 b(\xi_r^1) \exp\left( -2\int_0^r b(\xi_s^1) \diff s \right) \diff r.
    \end{equation*}
    Then \cite[Lemma 1]{harris_many--few_2017} states that the left-hand side of \eqref{eq:many-to-two} equals
    \begin{align*}
        &\E_{(x,y)} \Bigg[ \ind{T\leq t} 
        f\left((\xi^1_s)_{s\in [0,t]} \right) 
        g\left((\xi^{2,T}_s)_{s\in [0,t]} \right) \\
        & \hspace{2cm} \times
        \exp\left( \int_0^T b(\xi_s^1) \diff s 
        + \int_0^t b(\xi_s^1) \diff s 
        + \int_0^t b(\xi_s^{2,T}) \diff s \right) \Bigg].
    \end{align*}
    Conditioning on $\xi^1$ and $\xi^2$, and using the density of $T$ yields our version of the many-to-two lemma.
\end{proof}

\section{Upper bound for the maximum}
\label{sec:upper_bound}

The goal of this section is to show that with high probability there are no particles with a modulus much greater than $m(t)$ in the sense that 
$$\limsup_{a \to \infty} \limsup_{t\to\infty} \Pp{M_t \geq m(t)+ a} = 0.$$ 
This shows the upper half of the tightness claimed in Theorem \ref{thm:main}.

\subsection{An upper bound on the branching rate}

We start this subsection by introducing a useful event. First note that, dominating by a one-dimensional BBM with branching rate 1, we have (see e.g.\@ \cite[Eq.\@ (20)]{LalSel1987})
\[
\max_{u\in\cN(s)} X_u(s) - \sqrt{2} s 
\xrightarrow[s\to\infty]{\text{a.s.}} - \infty. 
\]
Therefore, setting
\begin{equation} \label{eq:def_A}
A_{s_0} \coloneqq \left\{ \forall s \geq s_0, \max_{u\in\cN(s)} X_u(s) \leq \sqrt{2} s -1 \right\},
\end{equation}
we have 
\begin{equation} \label{eq:prob_A}
\Pp{A_{s_0}} \xrightarrow[s_0 \to \infty]{} 1.
\end{equation}
Hence, in this section, we can and will regularly restrict our study to the event $A_{s_0}$ with $s_0$ large enough.
It is also sometimes useful to consider the larger event
\begin{equation} \label{eq:def_Abis}
A_{s_0,t} \coloneqq \left\{ \forall s \in [s_0,t], \max_{u\in\cN(s)} X_u(s) \leq \sqrt{2} s -1 \right\}.
\end{equation}

In particular, working on these events is useful to get the following upper bound on the branching rate.
Recall $(X_s,Y_s)_{s\in [0,t]}$ denotes a Brownian motion on $\R^2$ and $(R_s,\theta_s)_{s\in [0,t]}$ denotes a representation of its polar coordinates. 

\begin{lem}\label{lem:branching_rate_upper_bound}
	There exists $\sigma,L > 0$ such that, for any $0 \leq s_0 \leq t$,  on the event $\{ \forall s \in [s_0,t], X_s \leq \sqrt{2} s \text{ and } \abs{\theta_s} < \sigma \}$ as well as on the event $\{ \forall s \in [s_0,t], X_s \leq \sqrt{2} s \text{ and } \abs{Y_s} < \sigma s \}$, we have
	\begin{equation*}
	\forall s \in [s_0,t], \quad b(\theta_s) \leq 1 - \beta \abs{\frac{Y_s}{\sqrt{2} s}}^\alpha (1+f(Y_s,s)),
	\end{equation*}
    where $f = f^-_{L,2-\alpha,1} = - \bigl( \bigl[ L \bigl( \abs{\frac{y}{s}}^{2-\alpha} + s^{-1} \bigr) \bigr] \wedge \zeta \bigr)$ where $\zeta \in (0,1/2]$ is chosen in \eqref{eq:def_f_-_+}.
\end{lem}

Note that the parameters $a=2-\alpha$ and $b=1$ appearing here in $f^-_{L,a,b}$ satisfy the assumptions $a> 1-\frac{\alpha}{2}$ and $b > 1-\kappa$ of the propositions of Section \ref{sec:probabilistic_arguments}.

\begin{proof}
	By Assumption \ref{assumption2}, $b(\theta) = 1 - \beta \abs{\theta}^\alpha + O(\theta^2)$ as $\theta \to 0$, so there exist $K> 0$ and $\sigma_0 \in (0,\pi/2)$ such that, for any $\abs{\theta} \leq \sigma_0$, $b(\theta) \leq 1 - \beta \abs{\theta}^\alpha + K\theta^2$.
	Moreover, one can choose $\sigma_0$ small enough such that the function $\theta \in [0,\sigma_0] \mapsto 1 - \beta \abs{\theta}^\alpha + K\theta^2$ is decreasing.
	On the other hand, by Assumption \ref{assumption1} or \ref{assumption1'}, $\sup_{[-\pi,\pi]\setminus[-\sigma_0,\sigma_0]} b <1$. Hence, choosing $\sigma_0$ small enough, we get
	\[
	b(\theta) \leq \overline{b}(\theta) 
	\coloneqq \begin{cases}
	1 - \beta \abs{\theta}^\alpha + K\theta^2, 
	& \text{if } \theta \in [-\sigma_0,\sigma_0], \\
	1 - \beta \abs{\sigma_0}^\alpha + K\sigma_0^2,
	& \text{if } \theta \in [-\pi,\pi]\setminus[-\sigma_0,\sigma_0],
	\end{cases}
	\]
	and the function $\overline{b}$ is non-increasing on $[0,\pi]$.
	
	We now work on $E_1 \coloneqq \{ \forall s \in [s_0,t], X_s \leq \sqrt{2} s \text{ and } \abs{\theta_s} < \sigma \}$ or $E_2 \coloneqq \{ \forall s \in [s_0,t], X_s \leq \sqrt{2} s \text{ and } \abs{Y_s} < \sigma s \}$.
	Then, for any $s \in [s_0,t]$ such that $\theta_s \in [0,\pi/2)$, we have 
	\[
	\theta_s = \arctan \left( \frac{Y_s}{X_s} \right) 
	\geq \arctan \left( \frac{Y_s}{\sqrt{2} s} \right) 
	\geq \frac{Y_s}{\sqrt{2} s} 
	- \frac{1}{3} \left( \frac{Y_s}{\sqrt{2} s} \right)^3,
	\]	
	and therefore
	\[
	b(\theta_s) 
	\leq \overline{b} (\theta_s)
	\leq \overline{b} \left( \frac{Y_s}{\sqrt{2} s} 
	- \frac{1}{3} \left( \frac{Y_s}{\sqrt{2} s} \right)^3\right).
	\]
	Then, one can choose $\sigma$ small enough such that the last argument of $\overline{b}$ necessarily belongs to $[0,\sigma_0]$ both on $E_1$ and on $E_2$, and therefore we get, for some $L = L(K,\sigma_0,\alpha,\beta) > 0$,
	\begin{equation} \label{eq:UB_b}
	b(\theta_s) 
	\leq 1 - \beta \abs{\frac{Y_s}{\sqrt{2} s}}^\alpha \left( 1 - L \abs{\frac{Y_s}{s} }^{2-\alpha} \right),
	\end{equation}
	and the same holds if $\theta_s \in (-\pi/2,0]$ by the same argument.
	On $E_1$, it is enough to consider the case $\theta_s \in (-\pi/2,\pi/2)$ by choosing $\sigma < \pi/2$.
	On $E_2$, we could also have $\theta_s \in [-\pi,\pi] \setminus (-\pi/2,\pi/2)$, but then 
	$b(\theta_s) \leq 1 - \beta \abs{\sigma_0}^\alpha + K\sigma_0^2$ and the inequality \eqref{eq:UB_b} stays true if $\sigma$ is chosen small enough. 
	Finally, $\sigma$ can be chosen small enough such that, on $E_1$ or on $E_2$, $L \lvert Y_s/s \rvert^{2-\alpha} \leq \zeta$, where $\zeta$ appears in the definition of $f^-_{L,2-\alpha,1}$ in \eqref{eq:def_f_-_+}, so that the result follows from \eqref{eq:UB_b}.
\end{proof}

\subsection{The pseudo-derivative martingale}

For $t \geq 0$, let 
\begin{equation}\label{eq:def_Zt}
Z_t \coloneqq t^{-\kappa/4} \exp \left( \vartheta_1 t^{1-\kappa} \right)
\sum_{u\in\cN(t)} (\sqrt{2}t-X_u(t)) e^{\sqrt{2}(X_u(t)-\sqrt{2} t)}.
\end{equation}
This quantity naturally appears in the upper bound argument when computing conditional first moment give $\cF_t$, so we need to bound its first moment.
Note that this is not a martingale, but it is defined analogously to the derivative martingale for the standard BBM introduced by Lalley and Sellke \cite{LalSel1987}.
The goal of this section is to prove the following result.
\begin{lem} \label{lem:Z}
	For any $s_0 > 0$, there exists $C > 0$ such that, for any $t \geq s_0$,
	\[
	\Ec{Z_t \1_{A_{s_0,t}}} \leq C.
	\]
\end{lem}
\begin{proof}
	By the many-to-one lemma (Lemma \ref{lem:many_to_one}),
	\[
	\Ec{Z_t \1_{A_{s_0,t}}} 
	\leq t^{-\kappa/4} e^{\vartheta_1 t^{1-\kappa}}
	\Ec{ \exp \left( \int_0^t b(\theta_s) \diff s \right) (\sqrt{2}t-X_t) e^{\sqrt{2} X_t-2t} 
		\ind{\forall s \in [s_0,t], X_s \leq \sqrt{2} s}}.
	\]
	Let $\sigma > 0$ be the constant given by Lemma \ref{lem:branching_rate_upper_bound}.
	Define the events $E \coloneqq \{ \forall s \in [s_0,t], \abs{Y_s} \leq \sigma s \}$ and, for $k \geq 1$, 
	\[
	F_k \coloneqq \{ \exists s \in [2^{k-1} s_0, 2^k s_0) : \abs{Y_s} > \sigma s \}
	\cap \{ \forall s \in [2^k s_0,t] : \abs{Y_s} \leq \sigma s \}.
	\]
	Let $k_t = \min \{ k \geq 1: 2^k s_0 \geq t \}$. We have $E^c \subset \bigcup_{k=1}^{k_t} F_k$, so we use the bound $1 \leq \1_E + \sum_{k=1}^{k_t} \1_{F_k}$.
	
	Using Lemma \ref{lem:branching_rate_upper_bound} and bounding $b(\theta_s) \leq 1$ for $s \leq s_0$, the part on event $E$ is at most
	\begin{align}
	& t^{-\kappa/4} e^{\vartheta_1 t^{1-\kappa}}
	\Ec{ \exp \left( - \beta \int_{s_0}^t \abs{\frac{Y_s}{\sqrt{2} s}}^\alpha (1+f(Y_s,s)) \diff s \right) 
	(\sqrt{2}t-X_t) e^{\sqrt{2}X_t-t} \ind{\forall s \in [s_0,t], X_s \leq \sqrt{2} s}} \nonumber \\
	& \leq C s_0^{-\kappa/4} e^{\vartheta_1 s_0^{1-\kappa}}
	\times \Ec{(\sqrt{2}s_0-X_{s_0}) e^{\sqrt{2}X_{s_0} - s_0}},
	\label{eq:bound_on_E}
	\end{align}
	using independence between $X$ and $Y$, Proposition \ref{prop:bound_on_integrated_G_tilde}.\ref{it:UB_integrated_G_tilde} for the $Y$-contribution and the fact that $((\sqrt{2}t-X_t) e^{\sqrt{2}X_t - t} \ind{\forall s \in [s_0,t], X_s \leq \sqrt{2} s})_{t \geq s_0}$ is a martingale.
	The right-hand side of \eqref{eq:bound_on_E} is a constant depending on $s_0$, so it concludes this part.
	
	We now bound the part on event $F_k$ for some $1 \leq k \leq k_t$.
	Using Lemma \ref{lem:branching_rate_upper_bound} between times $2^k s_0$ and $t$, this part is at most
	\begin{align}
	& t^{-\kappa/4} \exp \left( \vartheta_1 t^{1-\kappa} \right)
	\Ec{ \exp \left( - \beta \int_{2^k s_0}^t \abs{\frac{Y_s}{\sqrt{2} s}}^\alpha (1+f(Y_s,s)) \diff s \right)
		\ind{\exists s \in [2^{k-1} s_0, 2^k s_0) : \abs{Y_s} > \sigma s}} 
	\nonumber \\
	& \quad {} \times \Ec{ (\sqrt{2}t-X_t) e^{\sqrt{2} X_t-t} \ind{\forall s \in [s_0,t], X_s \leq \sqrt{2} s}}
	\nonumber \\
	& \leq C(s_0) (2^k s_0)^{-\kappa/4}  \exp \left( \vartheta_1 (2^k s_0)^{1-\kappa} \right)
	\Pp{\exists s \in [2^{k-1} s_0, 2^k s_0) : \abs{Y_s} > \sigma s},
	\label{eq:bound_on_F_k}
	\end{align}
	proceeding as in \eqref{eq:bound_on_E}.
	Then, this last probability can be bounded by
	\begin{align} 
	\Pp{\max_{s \in [0,2^k s_0]} \abs{Y_s} > \sigma 2^{k-1} s_0} 
	& \leq 
	2 \cdot \Pp{\max_{s \in [0,2^k s_0]} Y_s > \sigma 2^{k-1} s_0}
	\nonumber \\
	& =  2 \cdot \Pp{\abs{Y_{2^k s_0}} > \sigma 2^{k-1} s_0},
	\label{eq:gaussian_bounds}
	\end{align}
	using that $\max_{s\in[0,T]} Y_s$ has the same distribution as $\abs{Y_T}$ for any $T>0$. Together with the tail bound $\P(\abs{Y_T} \geq x) \leq 2e^{-x^2/(2T)}$ for $x,T> 0$, we get that the right-hand side of \eqref{eq:bound_on_F_k} is at most $C(s_0) \exp (\vartheta_1 (2^k s_0)^{1-\kappa} - \frac{\sigma^2}{8} 2^k s_0)$.
	Summing over $k \geq 0$, this is bounded by $C(s_0)$, so it concludes the proof.
\end{proof}

\subsection{Localization of the trajectory of an extremal particle}

In this section, we prove several properties of the trajectory of an extremal particle, that is a particle $u \in \cN(t)$ such that $R_u(t) \geq m(t)+O(1)$. In this first lemma, we show that such a particle typically has a small angle after time $t/2$. This allows us afterwards to restrict ourselves to angles smaller than $\sigma$ given by Lemma \ref{lem:branching_rate_upper_bound}, so that we can bound the branching rate.

\begin{lem} \label{lem:theta}
    For any $\eta>0$, there exists $c>0$ such that, for any $a \geq 0$, for $t$ large enough,
	\begin{align*}
	\Pp{ \exists u \in \cN(t) : 
		R_u(t) \geq m(t)-a, \max_{s\in [t/2,t]} \abs{\theta_u(s)} > \eta }
	\leq e^{-ct}.
	\end{align*}
\end{lem}
\begin{proof} 
	Before diving into the proof, we recall some facts on the polar coordinates $(R_s,\theta_s)_{s\geq 0}$ of the planar Brownian motion.
    Firstly, $(R_s)_{s\geq 0}$ is a 2-dimensional Bessel process starting from $r=(x^2+y^2)^{1/2}$ and its density at time $s >0$ is given by (see \cite[Eq.\@ 4.1.0.6]{borodin_handbook_2002})
	\begin{equation} \label{eq:density_Bessel}
	\P_{(x,y)}(R_s \in \diff z) = \frac{z}{s} e^{-(r^2+z^2)/2s} I_0 \left( \frac{rz}{s} \right) \diff z,
	\end{equation}
	where $I_0$ is the modified Bessel function defined by $I_0(z) = \sum_{k\geq 0} (z/2)^{2k}/(k!)^2$.
	Note that, for $z \geq 0$, $I_0(z) \leq (e^{z/2})^2 = e^z$, so we get the following upper bound for the density of $R_s$: 
	\begin{equation} \label{eq:density_Bessel_UB}
	\P_{(x,y)}(R_s \in \diff z) \leq \frac{z}{s} e^{-(z-r)^2/2s} \diff z.
	\end{equation}
    Concerning the angular coordinate $(\theta_s)_{s\geq 0}$, note that the representation can be chosen to be continuous, except at $s=0$ if the Brownian motion starts from the origin.
    Moreover, we have the following skew-product representation: for any $s \geq 0$ ($s > 0$ if the Brownian motion starts from the origin), there exists a Brownian motion $W$ on $\R$ starting from 0 at time 0 and independent of $R$ such that
    \begin{equation} \label{eq:skew-prod}
        \forall t \geq s, \quad \theta_t = \theta_s + W_{\tau(t)}, \qquad 
        \text{with } \tau(t) \coloneqq \int_s^t R_q^{-2} \diff q
    \end{equation}
    see e.g. \cite[Corollary 19.7]{Kal2021}.
	We now split the proof into three steps.

	\textit{Step 1.} We first prove that an extremal particle needs to have a large radial coordinate on $[t/2,t]$, more precisely: for any $a \geq 0$, for $t$ large enough,
	\begin{equation} \label{eq:angle_step1}
	\Pp{ \exists u \in \cN(t) : 
		R_u(t) \geq m(t)-a, \min_{s\in [t/2,t]} R_u(s) < \frac{\sqrt{2}}{4} t}
	\leq e^{-ct}.
	\end{equation}
    As in the statement of the lemma, the constant $c$ throughout the proof does not depend on $a$, but the threshold involved in the ``$t$ large enough'' statement can.
    By the many-to-one lemma (Lemma \ref{lem:many_to_one}) with the crude bound $b(\theta_s) \leq 1$, we can bound the probability in \eqref{eq:angle_step1} by
    \begin{equation} \label{eq:angle_step1_2}
        e^t \cdot \Pp{ R_t \geq m(t)-a, \min_{s\in [t/2,t]} R_s < \frac{\sqrt{2}}{4} t} 
        = e^t \cdot \Ec{\phi(T_0) \ind{T_0 < t}},
    \end{equation}
	using Markov property at the stopping time $T_0 \coloneqq \inf \{ s \geq t/2 : R_s \leq \sqrt{2}t/4 \}$
    and setting $\phi(s) \coloneqq \P_{\sqrt{2}t/4}(R_{t-s} \geq m(t)-a)$ for $s\in [t/2,t)$, where $(R_q)_{q \geq 0}$ starts from $\sqrt{2}t/4$ at time 0 under $\P_{\sqrt{2}t/4}$.
    Using the upper bound \eqref{eq:density_Bessel_UB}, we get, for $t$ large enough and uniformly in $s\in [t/2,t)$,
    \begin{align*}
        \phi(s) 
        \leq \int_{m(t)-a}^\infty 
        \frac{z}{t-s} e^{-(z-\sqrt{2} t/4)^2/2(t-s)}
        \diff z
        \leq \frac{Ct}{\sqrt{t-s}} e^{-(m(t)-a-\sqrt{2} t/4)^2/(t-s)}
        \leq e^{-9t/8+o(t)}.
    \end{align*}
    Coming back to \eqref{eq:angle_step1_2}, this proves \eqref{eq:angle_step1}.

    \textit{Step 2.} We now prove that the angle of an extremal particle cannot be far from 0 on the whole interval $[t/2,t]$, that is, for any $a \geq 0$, for $t$ large enough,
	\begin{equation} \label{eq:angle_step2}
	\Pp{ \exists u \in \cN(t) : 
		R_u(t) \geq m(t)-a, \min_{s\in [t/2,t]} \abs{\theta_u(s)} > \frac{\eta}{2}
		}
	\leq e^{-ct}.
	\end{equation} 
    By Assumption \ref{assumption1} or \ref{assumption1'}, we have $c(\eta) \coloneqq \sup_{[-\pi,\pi] \setminus [-\eta/2,\eta/2]} b < 1$.
    Therefore, using the many-to-one lemma and bounding $b(\theta_s)$ by $c(\eta)$ for $s \in [t/2,t]$ and by 1 otherwise, we get that the probability in \eqref{eq:angle_step2} is at most
	\[
	\exp \left( \frac{t}{2} + \frac{t}{2} c(\eta) \right)
	\cdot 	
	\Pp{R_t \geq m(t)-a}
	= \exp \left( \frac{t}{2} (c(\eta)-1) + o(t) \right),
	\]
    using that $\Pp{R_t \geq m(t)-a} = e^{-t+o(t)}$ as a consequence of \eqref{eq:density_Bessel_UB} and standard Gaussian bounds.
    This proves \eqref{eq:angle_step2}.
 
	\textit{Step 3.} Finally, it is enough to prove that, for any $a \geq 0$, for $t$ large enough,
	\begin{align}
	& \mathbb{P} \Bigg( \exists u \in \cN(t) : 
		R_u(t) \geq m(t)-a,  \nonumber \\
	& \hspace{1cm}	\min_{s\in [t/2,t]} R_u(s) \geq \frac{\sqrt{2}}{4} t,
		\min_{s\in [t/2,t]} \abs{\theta_u(s)} \leq \frac{\eta}{2},
		\max_{s\in [t/2,t]} \abs{\theta_u(s)} > \eta \bigg)
	 \leq e^{-ct}.
    \label{eq:angle_step3}
	\end{align}
    The key idea is that this event requires an angular displacement of at least $\eta/2$ over the time interval $[t/2,t]$, which has a cost exponential in $t$ for a 2d Brownian motion with radius of order $t$.
	We first apply the many-to-one lemma with the bound $b(\theta_s) \leq 1$ to get that the probability in \eqref{eq:angle_step3} is at most
	\begin{equation} \label{eq:angle_step3_2}
	    e^t \cdot
	    \Pp{ R_t \geq m(t)-a, 
		\min_{s\in [t/2,t]} R_s \geq \frac{\sqrt{2}}{4} t,
		T_1 \leq t, T_2 \leq t },
	\end{equation}
	with $T_1 \coloneqq \inf \{ s \geq t/2 : \abs{\theta_s} \leq \frac{\eta}{2} \}$ and $T_2 \coloneqq \inf \{ s \geq t/2 : \abs{\theta_s} \geq \eta \}$. 
	There are two cases to distinguish, $T_1 < T_2$ and $T_2 < T_1$, but they are treated similarly, so we focus on the case $T_1 < T_2$.
	By Markov property at time $T_1$,
	\begin{equation} \label{eq:step_3}
	\Pp{ R_t \geq m(t)-a, 
		\min_{s\in [t/2,t]} R_s \geq \frac{\sqrt{2}}{4} t,
		T_1 < T_2 \leq t } 
	\leq \Ec{ \chi(T_1,R_{T_1}) \ind{T_1 < t}},
	\end{equation}
	where we set, for any $s \in [t/2,t)$ and $r > 0$,
	\[
	\chi(s,r) \coloneqq \P_{(s,r,\eta/2)} \left( R_t \geq m(t)-a, 
	\min_{q\in [s,t]} R_q \geq \frac{\sqrt{2}}{4} t, T_2 \leq t \right),
	\]
	where under $\P_{(s,r,\eta/2)}$ the process $(R_q,\theta_q)_{q\geq s}$ starts from $(r,\eta/2)$ at time~$s$.
	Now, using the representation of the angle in \eqref{eq:skew-prod} and then using the lower bound on $R$ to upper bound $\tau(t)$, we get
	\begin{align*}
	\chi(s,r) 
	& \leq \P_{(s,r,\eta/2)} \left( R_t \geq m(t)-a, 
	\min_{q\in [s,t]} R_q \geq \frac{\sqrt{2}}{4} t, 
	\max_{p \in [0,\tau(t)]} W_p \geq \frac{\eta}{2} \right) \\
	& \leq \P_{(s,r,\eta/2)} \left( R_t \geq m(t)-a, 
	\max_{p \in [0,4/t]} W_p \geq \frac{\eta}{2} \right) \\
	& \leq 2 \exp \left( - \frac{\eta^2 t}{32} \right) \cdot \P_{(s,r,\eta/2)} \left( R_t \geq m(t)-a \right),
	\end{align*}
    by independence of $R$ and $W$ and using that $\max_{p \in [0,4/t]} W_p$ is distributed as $\abs{W_{4/t}}$ together with a Gaussian tail bound.
	Coming back to \eqref{eq:step_3} and using again Markov property at time $T_1$, the right-hand side of \eqref{eq:step_3} is at most
	\begin{align*}
	2 \exp \left( - \frac{\eta^2 t}{32} \right) \cdot \Pp{R_t \geq m(t)-a}
	= \exp \left( -\frac{\eta^2 t}{32} - t  + o(t) \right)
	\leq e^{-t-ct}.
	\end{align*}
	Proceeding similarly in the case $T_2 < T_1$, we get that \eqref{eq:angle_step3_2} is bounded by $e^{-ct}$, which concludes the proof of \eqref{eq:angle_step3} and hence of the lemma.
\end{proof}

In the next lemma, we prove that the final position of an extremal particle $u$ needs to have a final $y$-coordinate $Y_u(t)$ of order roughly $t^{\kappa/2}$, which is the typical space scale of the Brownian motion weighted by an integral via PDEs studied in Sections \ref{sec:BM_weighted} and \ref{sec:probabilistic_arguments}.
In particular, we deduce $Y_u(t) = o(\sqrt{t})$, which together with $R_u(t) \geq m(t)$ implies that $R_u(t) = X_u(t) + o(1)$, allowing us afterwards to replace the radial coordinate of an extremal particle by its $x$-coordinate. We also note that this lemma, together with the previous lemma, implies Proposition \ref{porism}.

\begin{lem} \label{lem:control_Y}
	For small enough $\eta > 0$, for any $\varepsilon > 0$ and $a \geq 0$, we have
	\begin{align*}
	\Pp{ \exists u \in \cN(t) : 
		R_u(t) \geq m(t)-a, \max_{s\in [t/2,t]} \abs{\theta_u(s)} \leq \eta,
		\abs{Y_u(t)} \geq t^{\frac{\kappa}{2} +\varepsilon} }
	\xrightarrow[t\to\infty]{} 0.
	\end{align*}
\end{lem}
\begin{proof}
	We assume that $\eta \leq \sigma$, where $\sigma$ is given by Lemma \ref{lem:branching_rate_upper_bound}.
	By \eqref{eq:prob_A}, it is enough to work on the event $A_{s_0}$ for some large $s_0$.
	Using $A_{s_0} \subset A_{s_0,t/2}$ and a union bound over $u\in \cN(t)$, we first get
	\begin{align*}
	& \Ppsq{ A_{s_0} \cap \left\{ \exists u \in \cN(t) : 
		R_u(t) \geq m(t)-a, \max_{s\in [t/2,t]} \abs{\theta_u(s)} \leq \eta,
		\abs{Y_u(t)} \geq t^{\frac{\kappa}{2}+\varepsilon} \right\} }{\cF_{t/2}}
	\\
	& \leq \1_{A_{s_0,t/2}} \Ecsq{ \sum_{u \in \cN(t)} 
	\ind{R_u(t) \geq m(t)-a, \max_{s\in [t/2,t]} \abs{\theta_u(s)} \leq \eta,
	\abs{Y_u(t)} \geq t^{\frac{\kappa}{2}+\varepsilon}}}{\cF_{t/2}}.
	\end{align*}
	Then, applying the branching property at time $t/2$ and then the many-to-one lemma (Lemma \ref{lem:many_to_one}), the last expectation equals 
	\begin{align*}
	& \sum_{u \in \cN(t/2)}
	\E_{(t/2,X_u(t/2),Y_u(t/2))} \left[ \sum_{u \in \cN(t)} 
	\ind{R_u(t) \geq m(t)-a, \max_{s\in [t/2,t]} \abs{\theta_u(s)} \leq \eta,
	\abs{Y_u(t)} \geq t^{\frac{\kappa}{2}+\varepsilon}} \right] 
	\\
	& = \sum_{u \in \cN(t/2)}
	\E_{(t/2,X_u(t/2),Y_u(t/2))} \left[ \exp \left( \int_{t/2}^t b(\theta_s) \diff s \right) 
	\ind{R_t \geq m(t)-a, 
	\max_{s\in [t/2,t]} \abs{\theta_s} \leq \eta,
	\abs{Y_t} \geq t^{\frac{\kappa}{2}+\varepsilon}} \right].
	\end{align*}
	Finally, using Lemma \ref{lem:branching_rate_upper_bound} and noting that, on the event in the last indicator function, we have $\abs{Y_t} = \abs{X_t \tan \theta_t} \leq 2 \eta t$, by choosing $\eta$ small enough such that $\tan \eta \leq \sqrt{2} \eta$, we obtain the following bound
	\begin{align}
	& \Ppsq{ A_{s_0} \cap \left\{ \exists u \in \cN(t) : 
		R_u(t) \geq m(t)-a, \max_{s\in [t/2,t]} \abs{\theta_u(s)} \leq \eta,
		\abs{Y_u(t)} \geq t^{\frac{\kappa}{2}+\varepsilon} \right\} }{\cF_{t/2}}
	\nonumber \\
	& \leq e^{t/2} \1_{A_{s_0,t/2}} \sum_{u \in \cN(t/2)}
	\E_{(t/2,X_u(t/2),Y_u(t/2))} \Biggl[ \exp \left( - \beta \int_{t/2}^t \abs{\frac{Y_s}{\sqrt{2} s}}^\alpha (1+f(Y_s,s)) \diff s \right) 
	\nonumber \\
	& \hspace{8.5cm} {} \times \ind{R_t \geq m(t)-a, \abs{Y_t} \in [t^{\kappa/2+\varepsilon}, 2 \eta t]} \Biggr].
	\label{eq:cond_proba}
	\end{align}
	Then, for any $x \leq t/\sqrt{2}-1$ (recall that on $A_{s_0,t/2}$ we have $X_u(t/2) \leq \sqrt{2} \cdot t/2-1$) and for any $y \in \R$, we have to bound
	\begin{align}
	& \E_{(t/2,x,y)} \left[ \exp \left( - \beta \int_{t/2}^t \abs{\frac{Y_s}{\sqrt{2} s}}^\alpha (1+f(Y_s,s)) \diff s \right) 
	\ind{R_t \geq m(t)-a, \abs{Y_t} \in [t^{\kappa/2+\varepsilon},2 \eta t]}
	\right] \nonumber \\
	& = \int_{\abs{z} \in [t^{\kappa/2+\varepsilon},2 \eta t]} 
	\P_{(t/2,x)} \left( \sqrt{X_t^2+z^2} \geq m(t)-a \right) 
	\widetilde{G}(t/2,y;t,z)
	\diff z, \label{eq:in_terms_of_Gtilde}
	\end{align}
	where $\widetilde{G}$ is defined in \eqref{eq:G_tilde}.
    Choosing $\eta < 1/\sqrt{2}$, we have, for $t$ large enough, for any $\abs{z} \leq 2\eta t$, $\abs{z} \leq m(t)-a$, and therefore $\sqrt{(m(t)-a)^2-z^2} \geq m(t)-a - z^2/(m(t)-a)^2 \geq m(t)-a - z^2/(2t)$.
	Hence, for $t$ large enough, for any $x \leq t/\sqrt{2}$ and $\abs{z} \leq 2\eta t$, we get
	\begin{align*}
	\P_{(t/2,x)} &\left( \sqrt{X_t^2+z^2} \geq m(t)-a \right) \\
	& = \P \left( \cN(0,t/2) \geq \sqrt{(m(t)-a)^2-z^2} - x \right) \\
	& \leq \P \left( \cN(0,t/2) \geq m(t)-a- \frac{z^2}{2t}-x \right)  \\
	& \leq \exp \left( - \frac{t}{2}  - \sqrt{2} \overline{x} + \vartheta_1 t^{1-\kappa} - \frac{\overline{x}^2}{t} + \frac{z^2}{t^2} \left( \frac{t}{\sqrt{2}} + \overline{x} \right) + O(\log t) \right),
	\end{align*}
	where we wrote $m(t)-a-x = \frac{t}{\sqrt{2}} + \overline{x} - \frac{\vartheta_1}{\sqrt{2}} t^{1-\kappa} + O(\log t)$ with $\overline{x} \coloneqq t/\sqrt{2}-x \geq 1$,  we  used that $\kappa \in (1/2,1)$, and noted that $m(t)-a-x- z^2/(2t) \geq 0$ if $\eta$ is chosen small enough in order to apply the Gaussian tail bound $\P(\cN(0,v) \geq u) \leq e^{-u^2/2v}$ for $u \geq 0$.
	Furthermore, maximizing in $\overline{x}$, we can bound
	\[
	\frac{z^2}{t^2} \left( \frac{t}{\sqrt{2}} + \overline{x} \right) -  \frac{\overline{x}^2}{t}
	\leq \frac{z^2}{\sqrt{2}t} + \frac{z^4}{4t^3}
	\leq \frac{z^2}{t},
	\]
	by choosing $\eta$ small enough.
	Coming back to \eqref{eq:in_terms_of_Gtilde}, we get
	\begin{align*}
	& \E_{(t/2,x,y)} \left[ \exp \left( - \beta \int_{t/2}^t \abs{\frac{Y_s}{\sqrt{2} s}}^\alpha (1+f(Y_s,s)) \diff s \right) 
	\ind{R_t \geq m(t)-a, \abs{Y_t} \in [t^{\kappa/2+\varepsilon},2 \eta t]}
	\right] \nonumber \\
	& \leq \exp \left( - \frac{t}{2}  - \sqrt{2} \overline{x} + \vartheta_1 t^{1-\kappa} + O(\log t) \right) \int_{\abs{z} \in [t^{\kappa/2+\varepsilon},2 \eta t]} 
	e^{z^2/t} \widetilde{G}(t/2,y;t,z) \diff z \\
	& \leq \exp \left( - \frac{t}{2}  - \sqrt{2} \overline{x} + \vartheta_1 (t/2)^{1-\kappa} + O(\log t) \right)
	\int_{\abs{z} \in [t^{\kappa/2+\varepsilon},2\eta t]} 
	e^{z^2/t - c \abs{z}^{(2+\alpha)/2} t^{-\alpha/2}}
	\diff z,
	\end{align*}
	using Proposition \ref{prop:tail_Gtilde} and that $\kappa = 2\alpha/(2+\alpha)$.
    Note that, if $\eta$ is chosen small enough, one has $z^2/t - c \abs{z}^{(2+\alpha)/2} t^{-\alpha/2} \leq - \frac{1}{2}c \abs{z}^{(2+\alpha)/2} t^{-\alpha/2}$ for $\abs{z} \leq 2 \eta t$.
	Thus, the last integral is a $O(\exp(-c' t^{\varepsilon(2+\alpha)/2}))$.
	Coming back to \eqref{eq:cond_proba}, we get
	\begin{align*}
	& \Ppsq{ A_{s_0} \cap \left\{ \exists u \in \cN(t) : 
		R_u(t) \geq m(t)-a, \max_{s\in [t/2,t]} \abs{\theta_u(s)} \leq \eta,
		\abs{Y_u(t)} \geq t^{\frac{\kappa}{2}+\varepsilon} \right\} }{\cF_{t/2}} \\
	& \leq \exp \left( \vartheta_1 (t/2)^{1-\kappa} -c' t^{\varepsilon(2+\alpha)/2} + O(\log t) \right) \cdot 
	\1_{A_{s_0,t/2}} \sum_{u \in \cN(t/2)}
	e^{\sqrt{2}(X_u(t/2) - \sqrt{2} \cdot t/2)} \\
	& \leq \exp \left( -c' t^{\varepsilon(2+\alpha)/2} + O(\log t) \right) \cdot 
	\1_{A_{s_0,t/2}} Z_{t/2},
	\end{align*}
	using that $1 \leq (\sqrt{2} \cdot t/2 - X_u(t/2))$ on $A_{s_0,t/2}$.
	We conclude by taking the expectation and applying Lemma \ref{lem:Z}.
\end{proof}

In the next lemma, we improve the straight barrier given by the event $A_{s_0}$ to replace it by a curved barrier which includes the same polynomial correction as $m(t)$.
For $s>0$, let 
\begin{equation} \label{eq:def_m+}
m^+(s) \coloneqq \sqrt{2} s - \frac{\vartheta_1}{\sqrt{2}} s^{1-\kappa} + 10 \log s.
\end{equation} 
Note that $m^+(s)$ is above $m(s)$, at a distance of order $\log s$.
The coefficient 10 here has been arbitrarily chosen.

\begin{lem} \label{lem:no_hitting_the_barrier}
	We have
	\[
	\Pp{ \exists u \in \cN(t) : 
		\max_{s \in [t/2,t]} (X_u(s) - m^+(s)) > 0}	
	\xrightarrow[t\to\infty]{} 0.
	\]
\end{lem}

\begin{proof}
    In order to discretize time afterwards, we first argue that no particle is moving more than~1 on a time interval of length $t^{-2}$ between times $t/2$ and $t$: by a union bound over $s$ and over $u \in\cN(s)$ and the many-to-one lemma (bounding the branching rate by 1), we have
    \begin{align*}
        & \Pp{\exists s \in [t/2,t+t^{-2}] \cap t^{-2}\Z, \exists u \in \cN(s) : \max_{r\in[s-t^{-2},s]} \abs{X_u(r) - X_u(s)} > 1} \\
        & \leq \sum_{s \in [t/2,t+t^{-2}] \cap t^{-2}\Z} e^s \cdot 
        \Pp{\max_{r\in[0,t^{-2}]} \abs{B_r} > 1} 
        \leq t^3 \cdot e^t \cdot 4 e^{-t^2/2} 
        \xrightarrow[t\to\infty]{} 0,
    \end{align*}
    using in the second inequality that $\max_{r\in[0,t^{-2}]} B_r$ is distributed as $\lvert B_{t^{-2}} \rvert$ and a classical Gaussian tail bound. It follows that it is sufficent to prove that
    \begin{equation}
	\Pp{ \exists u \in \cN(t) : 
		\max_{s \in [t/2,t]\cap t^{-2}\Z} (X_u(s) - m^+(s)+1) > 0}	
	\xrightarrow[t\to\infty]{} 0. \label{eq:discretized}
	\end{equation}
    Recalling that $\Pp{A_{s_0}^c} \to 0$ as $s_0 \to \infty$ by \eqref{eq:prob_A}, it is enough to prove that
	\begin{equation} \label{eq:adding_A}
	\Pp{ A_{s_0} \cap \left\{ \exists u \in \cN(t) : 
		\max_{s \in [t/2,t]\cap t^{-2}\Z} (X_u(s) - m^+(s)+1) > 0 \right\} } 
	\xrightarrow[t\to\infty]{} 0,
	\end{equation}
   	for any fixed $s_0 > 0$.
	Using a union bound over $s \in [t/2,t] \cap t^{-2}\Z$,
	the left-hand side of \eqref{eq:adding_A} is at most
	\begin{align*}
	& \sum_{s \in [t/2,t] \cap t^{-2}\Z} 
	\Pp{ A_{s_0} \cap \left\{ \exists u \in \cN(s) : X_u(s) > m^+(s)-1 \right\} } \\
	& \leq \sum_{s \in [t/2,t] \cap t^{-2}\Z}  
	\Pp{ A_{s_0} \cap \left\{ \exists u \in \cN(s) : X_u(s) > m^+(s)-1, \max_{r\in [s/2,s]} \abs{\theta_u(r)} \leq \sigma \right\} } 
	+e^{-cs} ,
	\end{align*}
	where $\sigma>0$ is given by Lemma \ref{lem:branching_rate_upper_bound} and $c$ is then given by  Lemma \ref{lem:theta}, noting that $X_u(s) > m^+(s)-1$ implies $R_u(s) > m(s)$.
	Therefore, it is now enough to prove that
	\begin{equation} \label{eq:curved_barrier_new_goal}
	\Pp{ A_{s_0} \cap \left\{ \exists u \in \cN(t) : X_u(t) > m^+(t)-1, \max_{s\in [t/2,t]} \abs{\theta_u(s)} \leq \sigma \right\} } 
	= o \left( t^{-3} \right), 
	\end{equation}
   	as $t \to\infty$, for any fixed $s_0 > 0$.
   	
    We first work conditionally on $\cF_{t/2}$. We first proceed as for the proof of \eqref{eq:cond_proba}: by a union bound over $u\in \cN(t)$, the branching property at time $t/2$, the many-to-one lemma (Lemma \ref{lem:many_to_one}) and Lemma \ref{lem:branching_rate_upper_bound}, we get
    \begin{align}
    & \Ppsq{ A_{s_0} \cap \left\{ \exists u \in \cN(t) : X_u(t) > m^+(t)-1, \max_{s\in [t/2,t]} \abs{\theta_u(s)} \leq \sigma \right\} }{\cF_{t/2}} \nonumber \\
    & \leq e^{t/2} \1_{A_{s_0,t/2}}  \nonumber \\
    & \hspace{0.6cm} \times
    \sum_{u \in \cN(t/2)}
    \E_{(t/2,X_u(t/2),Y_u(t/2))} \Biggl[ \exp \left( - \beta \int_{t/2}^t \abs{\frac{Y_s}{\sqrt{2} s}}^\alpha (1+f(Y_s,s)) \diff s \right) \ind{X_t > m^+(t)-1} \Biggr] \nonumber \\
    & \leq e^{t/2} \1_{A_{s_0,t/2}} \sum_{u \in \cN(t/2)} 
    C \exp\left(\vartheta_1 \left((t/2)^{1-\kappa}-t^{1-\kappa} \right)\right)
    \P_{(t/2,X_u(t/2))} \left( X_t > m^+(t)-1 \right),
    \label{eq:curved_barrier_given_t/2}
    \end{align}
	using the independence between $X$ and $Y$ and applying Proposition~\ref{prop:bound_on_integrated_G_tilde}.\ref{it:UB_integrated_G_tilde} to the $Y$ part.
    For the $X$ contribution, note that, under $\P_{(t/2,X_u(t/2))}$, $X_t$ is Gaussian with mean $X_u(t/2)$ and variance $t/2$. Moreover, note that on event $A_{s_0,t/2}$ we have $X_u(t/2) \leq \sqrt{2} \cdot t/2$ so that $m^+(t) - X_u(t/2) \geq 0$.
    Hence, we can apply the Gaussian tail bound $\P(\cN(0,v) \geq a) \leq e^{-a^2/2v}$ for $a \geq 0$ to get
    \begin{align*}
        & \P_{(t/2,X_u(t/2))}\left( X_t \geq m^+(t)-1 \right) \\
        & \leq \exp\left(-t/2-\sqrt{2}(\sqrt{2}t/2-X_u(t/2)) + \vartheta_1 t^{1-\kappa} - 10 \sqrt{2}\log t +\sqrt{2} \right).
    \end{align*}
    Therefore, the right-hand side of \eqref{eq:curved_barrier_given_t/2} is at most
    \[
    C t^{-10\sqrt{2}} \exp\left(\vartheta_1 (t/2)^{1-\kappa}\right) 
    \1_{A_{s_0,t/2}} 
    \sum_{u \in \cN(t/2)} e^{-\sqrt{2}(\sqrt{2}t/2-X_u(t/2))}
    \leq C t^{-10\sqrt{2}+\kappa/4} \1_{A_{s_0,t/2}} Z_{t/2},
    \]
    using that on $A_{s_0,t/2}$ we have $1 \leq (\sqrt{2}t/2 - X_u(t/2))$. 
    Taking the expectation and applying Lemma \ref{lem:Z} yields \eqref{eq:curved_barrier_new_goal} and concludes the proof.
\end{proof}

We conclude this subsection by the following corollary which gathers the three previous lemmas and is the starting point of the proof of the upper bound of Theorem~\ref{thm:main} in the next subsection.
\begin{cor} \label{cor:localization}
	Recall the definition of $m^+(s)$ in \eqref{eq:def_m+}.
	For any $\eta > 0$ and $a \geq 0$, as $t \to \infty$,
	\begin{align*}
	& \Pp{M_t \geq m(t)+a+1} \\
	& \leq \mathbb{P} \big(\exists u \in \cN(t) : 
		X_u(t) \geq m(t)+a, 
		\max_{s\in [t/2,t]} \abs{\theta_u(s)} \leq \eta,
		\max_{s \in [t/2,t]} X_u(s) - m^+(s) \leq 0\big)
	+ o(1).
	\end{align*}
\end{cor}

\begin{proof}
	Let $\eta > 0$ and $a \geq 0$.
	First note that if the result holds for some $\eta$, then it directly holds for any other $\eta' > \eta$, hence we can assume that $\eta$ is smaller than the one given by Lemma \ref{lem:control_Y}.
	Fix $\varepsilon > 0$ such that $\kappa/2 +\varepsilon < 1/2$. By Lemmas \ref{lem:theta} and \ref{lem:control_Y}, we have
	\begin{align*}
	& \Pp{M_t \geq m(t)+a+1} \\
	& = \Pp{ \exists u \in \cN(t) : 
		R_u(t) \geq m(t)+a+1, 
		\max_{s\in [t/2,t]} \abs{\theta_u(s)} \leq \eta,
		\abs{Y_u(t)} < t^{\frac{\kappa}{2} +\varepsilon} }
	+o(1).
	\end{align*}
	Now consider $u \in \cN(t)$ satisfying the properties in the last probability. We can assume $\eta < \pi/2$ so that $X_u(t) \geq 0$.
	Then, we have $R_u(t) = X_u(t) + O(Y_u(t)^2/X_u(t)) = X_u(t) + o(1)$, because $X_u(t)$ is necessarily of order $t$ while $Y_u(t) = o(\sqrt{t})$. Therefore, we have $X_u(t) \geq m(t)+a$ for $t$ large enough (but deterministic). 
	This proves
	\begin{align*}
	\Pp{M_t \geq m(t)+a+1} 
	& \leq \Pp{ \exists u \in \cN(t) : 
		X_u(t) \geq m(t)+a, 
		\max_{s\in [t/2,t]} \abs{\theta_u(s)} \leq \eta}
	+ o(1).
	\end{align*}
	The conclusion of the corollary then follows directly from Lemma \ref{lem:no_hitting_the_barrier}.
\end{proof}

\subsection{Proof of the upper bound}
\label{sec:proof_UB}

\begin{proof}[Proof of the upper bound in Theorem \ref{thm:main}]
    We prove here that
    \begin{equation}
        \limsup_{a \to \infty} \limsup_{t\to\infty}
        \Pp{M_t \geq m(t)+ a} = 0.
    \end{equation}
    Let $\sigma > 0$ be the constant given by Lemma \ref{lem:branching_rate_upper_bound}.
    By Corollary \ref{cor:localization} and \eqref{eq:prob_A}, it is enough to prove, for any $s_0 >0$,
    \begin{equation} \label{eq:goal_for_E}
        \limsup_{a \to \infty} \limsup_{t\to\infty}
        \Pp{E} = 0,
    \end{equation}
    where we define
    \begin{equation*}
    E \coloneqq A_{s_0,t/2} \cap \left\{ \exists u \in \cN(t) : 
		X_u(t) \geq m(t)+a, 
		\max_{s\in [t/2,t]} \abs{\theta_u(s)} \leq \sigma,
		\max_{s \in [t/2,t]} X_u(s) \leq m^+(s) \right\},
    \end{equation*}
    which depends implicitly on $t$, $a$ and $s_0$.
    We summarize shortly the proof before diving into the details.
    Let $0<\delta<\kappa$ and set $t_0 = t-t^\delta$. The proof consists of three steps: we first work conditionally on $\mathcal{F}_{t_0}$, then on $\mathcal{F}_{t/2}$ and finally bound the non-conditional probability. 
    A key idea for the first step is that the upper barrier at $m^+(s)$ is too crude to get a good bound on the end of the trajectory. 
    Instead, we compare the $x$-coordinate of our BBM on time interval $[t_0,t]$ to a one-dimensional BBM with branching rate 1, for which sharp bounds are known. 
    This comparison is sufficient because $\delta < \kappa$, so the time window $t^\delta$ is smaller than the typical time scale $t^\kappa$ where the fact of having a varying branching rate plays a role%
    \footnote{Alternatively, we could have worked harder in Lemma~\ref{lem:no_hitting_the_barrier} to show that we can put a barrier closer to $m(t)$ at times close to $t$. However, we believe that the strategy presented here would be helpful to establish further results on the extremal particles, such as the convergence of $M_t-m_t$ or of the extremal process, because the comparison with a standard BBM allows for the use of known precise estimates.}. 
    The second step deals with the time interval $[t/2,t_0]$ and consists in calculating the conditional first moment of the bound obtained in the first step, and in bounding it in terms of the pseudo-derivative martingale $Z_{t/2}$.
    The final step, which deals with $[0,t/2]$ is a direct application of Lemma \ref{lem:Z}.

    \textit{Step 1.} Fix some $a \geq 0$. 
    We condition on $\mathcal{F}_{t_0}$ and count the number of particles at time $t_0$ which have a descendant that exceeds $m(t)+a$ at time $t$,
    \begin{align}
        \P(E | \mathcal{F}_{t_0}) 
        & \leq \1_{A_{s_0,t/2}} \sum_{u \in \mathcal{N}(t_0)} 
        \ind{\forall s\in [t/2, t_0], \vert \theta_u(s) \vert \leq \sigma}
        \ind{\forall s\in [t/2, t_0], X_u(s) \leq m^+(s)} 
        \nonumber \\
        & \hspace{3cm} {} \times 
        \P_{(t_0,X_u(t_0), Y_u (t_0))} \left( \max_{v\in \mathcal{N}(t), u\leq v} X_v(t) \geq m(t)+a \right).
        \label{eq:first_conditioning_on_t_0}
    \end{align}
    Next, we can dominate $\max_{v\in \mathcal{N}(t), u\leq v} X_v(t)$ by the maximum of a one-dimensional BBM with branching rate $1$, call $M^{1\mathrm{d}}_s$ its maximum at time $s$,
    \begin{equation}\label{eq:dominating_the_maximum}
        \P_{(t_0,X_u(t_0), Y_u (t_0))} \left( \max_{v\in \mathcal{N}(t), u\leq v} X_v(t) \geq m(t)+a \right) 
        \leq \P_{(0,X_u(t_0))} \left(M^{1\mathrm{d}}_{t^\delta}\geq m(t)+a \right),
    \end{equation}
    where we used that $t-t_0=t^\delta$.
    Now, we use the following bound for the tail of $M^{1\mathrm{d}}_s$, which follows from \cite[Corollary 10]{ArgBovKis2012}: letting $m^{1\mathrm{d}}(s)=\sqrt{2}s - \frac{3}{2\sqrt{2}}\log s$ for $s>1$, there exists $C>0$ such that, for any $x\in \R$ and $s \geq 1$,
    \begin{equation} \label{eq:UB_max_BBM}
        \P \left({M^{1\mathrm{d}}_s \geq m^{1\mathrm{d}}(s) + x} \right)
        \leq C (x\vee 1) e^{-\sqrt{2}x} \exp\left( -\frac{x_+^2}{4s} \right),
    \end{equation}
    where $x_+ = \max(x,0)$. 
    We apply this to \eqref{eq:dominating_the_maximum} with $s=t^\delta$ and 
    \[
    x = m(t) -m^{1\mathrm{d}}(t^\delta)-X_u(t_0)+a
    = \sqrt{2} t_0-X_u(t_0)+a
    - \frac{\vartheta_1}{\sqrt{2}} t^{1-\kappa} 
    -  \left(\frac{3}{2}(1-\delta) - \frac{\kappa}{4} \right) \frac{\log t}{\sqrt{2}}.
    \]
    Moreover, recall $t_0 = t-t^\delta$ and $\delta \in (0,\kappa)$, so we have $t^{1-\kappa} = t_0^{1-\kappa} +o(1)$ and 
    \[
    m(t)-m^{1\mathrm{d}}(t^\delta)+a
    = m^+(t_0) - \left(10\sqrt{2} +\frac{3}{2}(1-\delta) - \frac{\kappa}{4} \right) \frac{\log t}{\sqrt{2}} + a + o(1)
    \leq m^+(t_0).
    \]
    for $t$ large enough depending only on $a$.
    Hence, we get
    \begin{align*}
        \P_{(0,X_u(t_0))} \left(M^{1\mathrm{d}}_{t^\delta}\geq m(t)+a \right) 
        & \leq C \left( (m^+(t_0)- X_u(t_0)) \vee 1 \right)
        e^{- \sqrt{2}(\sqrt{2} t_0-X_u(t_0)+a)+\vartheta_1 t^{1-\kappa} } \\
        & \quad {} \times 
        t^{\frac{3}{2}(1-\delta) -  \frac{\kappa}{4}}
        \exp \left(- \frac{(m(t) -m^{1\mathrm{d}}(t^\delta)-X_u(t_0)+a)_+^2}{4t^\delta} \right).
    \end{align*}
    Note that the last exponential is a non-increasing function of $a$, so we can replace $a$ by $0$.
    Going back to \eqref{eq:first_conditioning_on_t_0}, this yields
    \begin{equation}
        \P(E | \mathcal{F}_{t_0}) 
        \leq C e^{- \sqrt{2}a}
        e^{\vartheta_1 t^{1-\kappa} } 
        t^{\frac{3}{2}(1-\delta) -  \frac{\kappa}{4}}
        \1_{A_{s_0, t/2}} 
        \Upsilon_{t_0},
        \label{eq:first_conditioning_on_t_0_result}
    \end{equation}
    where we set
    \begin{align*}
        \Upsilon_{t_0}
        & \coloneqq
        \sum_{u \in \mathcal{N}(t_0)} 
        (m^+(t_0)- X_u(t_0)+1)
        e^{\sqrt{2}(X_u(t_0)-\sqrt{2} t_0)} 
        \exp \left(- \frac{(m(t) -m^{1\mathrm{d}}(t^\delta)-X_u(t_0))_+^2}{4t^\delta} \right)
        \nonumber \\
        & \hspace{2cm} {} \times 
        \ind{\forall s\in [t/2, t_0], \vert \theta_u(s) \vert \leq \sigma}
        \ind{\forall s\in [t/2, t_0], X_u(s) \leq m^+(s)},
    \end{align*}
    and used that $X_u(t_0) \leq m^+(t_0)$ to bound $(m^+(t_0)- X_u(t_0)) \vee 1 \leq m^+(t_0)- X_u(t_0) + 1$.
    
    \textit{Step 2.} We now aim at bounding $\P(E \vert \mathcal{F}_{t/2})$.
    By \eqref{eq:first_conditioning_on_t_0_result}, it is enough to bound $\E[\Upsilon_{t_0} | \mathcal{F}_{t/2}]$.
    By the branching property at time $t/2$, the many-to-one lemma (Lemma \ref{lem:many_to_one}) and Lemma \ref{lem:branching_rate_upper_bound} (note that $X_u(s) \leq m^+(s)$ ensures that $X_u(s) \leq \sqrt{2} s$), we get
    \begin{align*}
        & \E[\Upsilon_{t_0} | \mathcal{F}_{t/2}] \\
        & \leq \sum_{u \in \mathcal{N}(t/2)} 
        \E_{(t/2,X_u(t/2), Y_u(t/2))}\Biggl[ 
        \exp \left( \left( t_0-\frac{t}{2} \right) - \beta \int_{t/2}^{t_0} \abs{\frac{Y_s}{\sqrt{2} s}}^\alpha (1+f(Y_s,s)) \diff s \right) \\  
        & \hspace{4.7cm} \times 
        (m^+(t_0)-X_{t_0}+1) e^{\sqrt{2}(X_{t_0}-\sqrt{2} t_0)}  \\
         & \hspace{4.7cm} \times 
        \exp \left(- \frac{(m(t) -m^{1\mathrm{d}}(t^\delta)-X_{t_0})_+^2}{4t^\delta} \right) 
        \ind{\forall s\in [t/2, t_0], X_s \leq m^+(s)} 
        \Biggr].
    \end{align*}
	Then, using the independence between $X$ and $Y$ and applying Proposition \ref{prop:bound_on_integrated_G_tilde} to the $Y$ part, we get
    \begin{equation} \label{eq:Upsilon_to_chi}
        \E[\Upsilon_{t_0} | \mathcal{F}_{t/2}] 
        \leq C \exp \left( t_0-\frac{t}{2} + \vartheta_1 \left((t/2)^{1-\kappa}-t_0^{1-\kappa} \right) \right)
        \sum_{u \in \mathcal{N}(t/2)} 
        \chi (\sqrt{2} \cdot t/2 - X_u(t/2)),
    \end{equation}
    where we set, for any $x \in \R$,
    \begin{align*}
        \chi(x) & \coloneqq 
        \E_{(t/2,\sqrt{2} \cdot t/2-x)} \Biggl[ 
        (m^+(t_0)-X_{t_0}+1) e^{\sqrt{2}(X_{t_0}-\sqrt{2} t_0)} 
        \exp \left(- \frac{(m(t) -m^{1\mathrm{d}}(t^\delta)-X_{t_0})_+^2}{4t^\delta} \right) \\
        & \hspace{9.2cm} {} \times 
        \ind{\forall s\in [t/2, t_0], X_s \leq m^+(s)}
        \Biggr].
    \end{align*}
    By Girsanov's theorem and invariance of Brownian motion by reflection, note that, under the measure with density $e^{-\sqrt{2}(X_{t_0}-[\sqrt{2} \cdot t/2-x])+(t_0-t/2)}$ w.r.t.\@ $\P_{(t/2,\sqrt{2} \cdot t/2-x)}$, the process $(\sqrt{2} s -X_s)_{s\geq t/2}$ has the same distribution as $(X_s)_{s\geq t/2}$ under $\P_{(t/2,x)}$.
    Therefore, we get
    \begin{align*}
        \chi(x) 
        & = e^{-\sqrt{2} x -t_0+\frac{t}{2}}
        \E_{(t/2,x)} \Biggl[ (X_{t_0}-f(t_0)+1)  
        \exp \left(- \frac{(m(t)-m^{1\mathrm{d}}(t^\delta)+X_{t_0}-\sqrt{2}t_0)_+^2}{4t^\delta} \right) \\
        & \hspace{9.2cm} \times
        \ind{\forall s\in [t/2, t_0], X_s \geq f(s)}
        \Biggr],
    \end{align*}
    with $f(s) \coloneqq \sqrt{2}s - m^+(s) = \frac{\vartheta_1}{\sqrt{2}} s^{1-\kappa} - 10 \log s$. 
    Note that $\chi(x) = 0$ if $x < f(t/2)$, so we consider now $x \geq f(t/2)$.
    Since $f$ is concave on $[t/2, t_0]$ for $t$ large enough, it stays above the straight line between its endpoint $(t/2,f(t/2))$ and $(t_0,f(t_0))$, that we denote by $f_\mathrm{line}$. Therefore, for any $y \geq f(t_0)$, we have
    \begin{align*}
    \P_{(t/2,x)} \left( \forall s\in [t/2, t_0], X_s \geq f(s) \middle| X_{t_0} = y \right)
    & \leq \P_{(t/2,x)} \left( \forall s\in [t/2, t_0], X_s \geq f_\mathrm{line}(s) \middle| X_{t_0} = y \right) \\
    & \leq \frac{2(x-f(t/2))(y-f(t_0))}{(t_0-t/2)},
    \end{align*}
    by \cite[Lemma 2]{Bra1978}.
    Therefore, integrating w.r.t.\@ $X_{t_0}$ first and noting that the density of $X_{t_0}$ under $\P_{(t/2,x)}$ is upper bounded by $1/\sqrt{2\pi(t_0-t/2)}$, we get
    \begin{align*}
        \chi(x) 
        & \leq \frac{Cx e^{-\sqrt{2} x -t_0+\frac{t}{2}}}{(t_0-t/2)^{3/2}} 
        \int_{f(t_0)}^\infty 
        (y-f(t_0)+1)^2 \exp \left(- \frac{(m(t)-m^{1\mathrm{d}}(t^\delta)+y-\sqrt{2}t_0)_+^2}{4t^\delta} \right) 
        \diff y.
    \end{align*}
    Then, recalling that $t^{1-\kappa} = t_0^{1-\kappa} +o(1)$, we have $m(t) -m^{1\mathrm{d}}(t^\delta)-\sqrt{2} t_0
    = -f(t_0) + O(\log t)$ and therefore, the integral in the last displayed equation is at most $Ct^{3\delta/2}$ for $t$ large enough.
    On the other hand, we have $t_0-t/2 \geq t/3$ for $t$ large enough.
    Combining this, we get 
    \[
        \chi(x) 
        \leq Cx t^{(\delta-1)3/2} e^{-\sqrt{2} x -t_0+\frac{t}{2}}.
    \]
    Coming back to \eqref{eq:first_conditioning_on_t_0_result} and \eqref{eq:Upsilon_to_chi} and using again that $t_0^{1-\kappa} = t^{1-\kappa} + o(1)$, we obtained that, for any $a \geq 0$ for $t$ large enough,
    \begin{align}
        \P(E | \mathcal{F}_{t/2}) 
        & \leq C e^{-\sqrt{2}a}
        e^{\vartheta_1 (\frac{t}{2})^{1-\kappa}} 
        t^{- \frac{\kappa}{4}}
        \1_{A_{s_0, t/2}} 
        \sum_{u \in \mathcal{N}(t/2)} 
        (\sqrt{2} \cdot t/2 - X_u(t/2)) 
        e^{-\sqrt{2} (\sqrt{2} \cdot t/2 - X_u(t/2)) } \nonumber \\
        & = C e^{-\sqrt{2}a} Z_{t/2} \1_{A_{s_0, t/2}}, 
        \label{eq:second_conditioning_on_t/2}
    \end{align}
    recalling the definition of the pseudo-derivative martingale $Z_{t/2}$ in \eqref{eq:def_Zt}.

    \textit{Step 3.} Taking the expectations of \eqref{eq:second_conditioning_on_t/2} and applying Lemma \ref{lem:Z}, we finally get: for any $a \geq 0$, for $t$ large enough,
    \begin{equation*}
        \Pp{E} \leq C e^{-\sqrt{2}a}.
    \end{equation*}
    This implies \eqref{eq:goal_for_E} and therefore concludes the proof.
\end{proof}

\section{Lower bound for the maximum}
\label{sec:lower_bound}

The goal of this section is to show that with high probability there will be particles at time $t$ with a modulus close to $m(t)$ in the sense that
$$\liminf_{a \to \infty} \liminf_{t\to\infty} \Pp{M_t \geq m(t)- a} = 1.$$ 
This shows the lower half of the tightness claimed in Theorem \ref{thm:main}.
The proof relies on a first and second moment calculation on the number of such particles satisfying further path conditions.

\subsection{A lower bound on the branching rate}

Before defining precisely the set of particles we consider, we present a key tool for the proof,  which is a lower bound for the branching rate of particles with appropriate trajectories. 
Recall that Lemma \ref{lem:branching_rate_upper_bound} provided an upper bound for the branching rate.

\begin{lem}\label{lem:lower_bound_branching_rate}
    Let $\eps\in (0,2\kappa-1)$.
    There exist $\sigma,L>0$ such that, for $s_0$ large enough and $t \geq s_0$, on the event $\{\forall s\in [s_0,t], X_s \geq \frac{m(t)}{t} s - s^{(1+\eps)/2} \text{ and } \vert Y_s \vert \leq \sigma s \}$, we have 
	\begin{equation*}
	\forall s\in [s_0,t], \quad  
	b(\theta_s)	
	\geq 1 - \beta \abs{\frac{Y_s}{\sqrt{2} s}}^\alpha (1+f(Y_s,s)),
	\end{equation*}
    where $f=f^+_{L,2-\alpha,(1-\eps)/2} = \left[ L \left( \abs{\frac{y}{s}}^{2-\alpha} + s^{-(1-\eps)/2} \right) \right] \wedge 1 $ which is defined in \eqref{eq:def_f_-_+}.
\end{lem}

We remark that the parameters $a=2-\alpha$ and $b=(1-\eps)/2$ appearing in $f^+_{L,2-\alpha,1-\delta}$ satisfy the conditions of the propositions of Section \ref{sec:probabilistic_arguments}, which are $a > 1-\alpha/2$ and $b>1-\kappa$. 

\begin{proof}[Proof of Lemma \ref{lem:lower_bound_branching_rate}]
    We work on the event appearing in the statement.
    First note that $m(t)/t = \sqrt{2} + O(t^{-\kappa}) \geq \sqrt{2} - t^{-(1-\eps)/2}$ for $s_0$ large enough, because $\kappa > 1/2 > (1-\eps)/2$, and therefore, we have $X_s \geq \sqrt{2} s - 2 s^{(1+\eps)/2}$ for any $s \in [s_0,t]$.
    To obtain a lower bound on the branching rate, we need an upper bound on the angle: using $\arctan(x)\leq x$ for $x\geq 0$ together with $X_s \geq \sqrt{2}s-2 s^{(1+\eps)/2}$ we get
    \begin{equation}\label{eq:doe}
        \abs{\theta_s}
        = \arctan \left( \abs{\frac{Y_s}{X_s}} \right) 
        \leq \abs{\frac{Y_s}{X_s}}
        \leq \abs{\frac{Y_s}{\sqrt{2}s}} 
        \frac{1}{1- \sqrt{2} s^{-(1-\eps)/2}} 
        \leq \abs{\frac{Y_s}{\sqrt{2}s}} \left(1+ 2 s^{-(1-\eps)/2} \right),
    \end{equation}
    for $s_0$ large enough.
    Observe that our assumption on the trajectory also entails that $\abs{\theta_s} \leq C \sigma$ for $s_0$ large enough. 
    By Assumption \ref{assumption2}, $b(\theta) = 1 - \beta \abs{\theta}^\alpha + O(\theta^2)$ as $\theta \to 0$, so there exist $M> 0$ and $\sigma_0 \in (0,\pi/2)$ such that, for any $\abs{\theta} \leq \sigma_0$, 
    \begin{equation*}
        b(\theta) \geq 1- \beta \abs{\theta}^\alpha - M\abs{\theta}^2.
    \end{equation*}
    Take $\sigma \leq C^{-1}$ and combine this with \eqref{eq:doe}, for any $s_0$ large enough we have
    \begin{align*}
        b(\theta_s) 
        &\geq 1- \beta \abs{\frac{Y_s}{\sqrt{2}s}}^\alpha
        \left( 1+ 2 s^{-(1-\eps)/2} \right)^2
        \left( 1 + \frac{M}{\beta} \abs{\frac{Y_s}{\sqrt{2}s}}^{2-\alpha} \right)\\
        &\geq 1- \beta \abs{\frac{Y_s}{\sqrt{2}s}}^\alpha 
        \left( 1+ L \left( \abs{\frac{Y_s}{s}}^{2-\alpha} + s^{-(1-\eps)/2} \right) \right),
    \end{align*}
    for some $L>0$.
    Lastly, note that $L(\lvert Y_s/s \rvert^{2-\alpha} + s^{-(1-\eps)/2}) \leq 1$ when we choose $\sigma$ small enough and $s_0$ large, so that this term equals $f^+_{L,2-\alpha,(1-\eps)/2}$ and the result follows.
\end{proof}

We can now introduce the set of particles we consider for the first and second moment argument.
Recall that $\alpha \in (2/3,2)$ implies $\kappa \in (1/2,1)$. 
Let $\eps\in (0,2\kappa-1)$
and $\sigma$ be small enough so that the conclusions of Lemmas~\ref{lem:branching_rate_upper_bound} and~\ref{lem:lower_bound_branching_rate} are satisfied. 
These two parameters $\eps$ and $\sigma$ are now fixed for the whole section.

For $0 \leq s_0 \leq s \leq t$, we set
\begin{align}
    g_t^+(s) & \coloneqq \frac{m(t)}{t} s - \left( (s-s_0) \wedge(t-s) \right)^{(1-\eps)/2}, \nonumber \\
    g_t^-(s) & \coloneqq \frac{m(t)}{t} s - s^{(1+\eps)/2},  \nonumber \\
    \Gamma_{s_0}(t) & \coloneqq \big\{ u \in \cN(t) : 
    X_u(t) \in [m(t)-1,m(t)],  \nonumber \\
    & \hspace{2.7cm}
    \forall s \in [s_0,t]: \abs{Y_u(s)} \leq \sigma s, X_u(s) \in [g_t^-(s),g_t^+(s)] \big\}. \label{eq:def:g_and_Gamma}
\end{align}
The parameter $s_0$ will be chosen large enough throughout the section.

\subsection{First moment estimate}

\begin{lem} \label{lemma:first_momement_lower_bound}
For $s_0$ large enough, there exists $c > 0$ such that, for any $t$ large enough, for any $x\in [\sqrt{2}s_0 - 2\sqrt{s_0}, \sqrt{2}s_0 - \sqrt{s_0}]$, $y \in [-s_0^{\kappa/2},s_0^{\kappa/2}]$,
    \[
    \E_{(s_0,x,y)}\left[ \abs{\Gamma_{s_0}(t)} \right] \geq c.
    \]
\end{lem}

\begin{proof}
    Throughout the proof we absorb dependencies on $s_0$ into the constants. 
    There are three events that contribute to $\Gamma_{s_0}(t)$, 
    \begin{align*}
        B_1 &= \left\{ X_t \in [m(t) -1, m(t)] \right\}, 
        \hspace{1cm}
        B_2 = \left\{ \forall s\in [s_0,t], X_s \in [g_t^-(s), g_t^+(s)] \right\}, \\
        B_3 &= \left\{  \forall s\in [s_0,t], \vert Y_s \vert \leq \sigma s  \right\},  
    \end{align*}
    we define them for the Brownian motions $(X_s,Y_s)_{s\geq 0}$.
    They arise when applying the many-to-one lemma, Lemma \ref{lem:many_to_one},
    \begin{align}\label{eq:after_many_21}
        \E_{(s_0,x,y)}\left[ \abs{\Gamma_{s_0}(t)} \right] 
        &= \E_{(s_0,x,y)}\left[ \1_{\cap_{i=1}^3 B_i} \exp\left( \int_{s_0}^t b(\theta_s) \diff s \right)\right]. 
    \end{align}
    Because we work on the event $B_2 \cap B_3$, we can apply Lemma \ref{lem:lower_bound_branching_rate}. As a consequence of this, the $X$ and $Y$ contributions decouple and \eqref{eq:after_many_21} is at least
    \begin{equation}\label{eq:eve}
        e^{t-s_0} \P_{(s_0,x)}\left(B_1\cap B_2 \right)
        \E_{(s_0,y)}\left[ \1_{B_3} \exp\left( - \beta\int_{s_0}^t \left\vert \frac{Y_s}{\sqrt{2}s} \right\vert^\alpha \left(1+f(Y_s,s) \right) \diff s \right)\right].
    \end{equation}
    For the $Y$-contribution, we have by Proposition \ref{prop:bound_on_integrated_G_tilde}.\ref{it:LB_integrated_G_tilde} applied with any $\eta \in (0,2-\kappa)$, noting that $s^{(\kappa+\eta)/2} \leq \sigma s$ for any $s \geq s_0$ by choosing $s_0$ large enough,
    \begin{align} \label{eq:Y-contrib}
    	\E_{(s_0,y)}\left[ \1_{B_3} \exp\left( - \beta\int_{s_0}^t \left\vert \frac{Y_s}{\sqrt{2}s} \right\vert^\alpha \left(1+f(Y_s,s) \right) \diff s \right)\right]
    	& \geq c t^{\kappa/4} \exp\left( -\vartheta_1 t^{1-\kappa} \right),
    \end{align}
    where the factors depending on $s_0$ have been absorbed in $c$.

    We now consider the $X$-contribution. By Girsanov's theorem with drift $\lambda = m(t)/t$, we have
    \begin{equation*}
        e^{t-s_0} \P_{(s_0,x)}\left(B_1\cap B_2 \right) 
        = e^{(1-\frac{\lambda^2}{2})(t-s_0)} 
        \E_{(s_0,x-\lambda s_0)}\left[\1_{\hat{B}_1} \1_{\hat{B}_2} e^{-\lambda(X_t-X_{s_0})} \right],
    \end{equation*}
    where, recalling the definitions of $g_t^+$ and $g_t^-$ in \eqref{eq:def:g_and_Gamma},
    \begin{align*}
        \hat{B}_1 = \left\{ X_t \in [-1,0] \right\}, \qquad 
        \hat{B}_2 = \left\{ \forall s\in [s_0,t], X_s \in \left[- s^{(1+\eps)/2}, - \left( (s-s_0) \wedge(t-s) \right)^{(1-\eps)/2} \right] \right\}.
    \end{align*}
    Noting that $1-\frac{\lambda^2}{2} = \vartheta_1 t^{-\kappa} + (\frac{3}{2}-\frac{\kappa}{4}) \frac{\log t}{t} + o(\frac{1}{t})$ 
    and bounding $e^{-\lambda(X_t-X_{s_0})} \geq c$ by using $X_t \in [-1,0]$ and our assumption on $x$, we get
    \begin{equation} \label{eq:hat1}
        e^{t-s_0} \P_{(s_0,x)}\left(B_1\cap B_2 \right) \geq c t^{3/2-\kappa/4} e^{\vartheta_1 t^{1-\kappa}}
        \P_{(s_0,x-\lambda s_0)}\left(\hat{B}_1\cap \hat{B_2}\right).
    \end{equation}
    Shifting time by $s_0$ and integrating w.r.t.\@ the final value $X_{t-s_0}$, we have
    \begin{align}
        & \P_{(s_0,x-\lambda s_0)}\left(\hat{B}_1\cap \hat{B_2}\right) 
        \nonumber \\
        & \geq \int_{-1}^{-1/2} 
        \P_{(0,x-\lambda s_0)}\left( 
        \forall s\in [0,t-s_0], X_s \in I_s
        \middle| X_{t-s_0}=y \right)
        \P_{(0,x-\lambda s_0)}\left( 
        X_{t-s_0} \in \diff y \right),
        \label{eq:hat2}
    \end{align}
    where $I_s = [- (s+s_0)^{(1+\eps)/2}, - \left( s\wedge(t-s_0-s) \right)^{(1-\eps)/2}]$. Note that we restricted ourselves to $y \in [-1,-1/2]$ to be at distance at least $1/2$ from the upper barrier at the end.
    Then, on the one hand, for $t$ large enough, 
    $\P_{(0,x-\lambda s_0)}( X_{t-s_0} \in \diff y) \geq c t^{-1/2} \diff y$.
    On the other hand, note that our lower barrier satisfies
    \begin{align*}
        - (s+s_0)^{(1+\eps)/2} 
        \leq - \frac{1}{2} s_0^{(1+\eps)/2} - \frac{1}{2} s^{(1+\eps)/2}
        \leq - \frac{1}{2} s_0^{(1+\eps)/2} - \frac{1}{2} \left( s\wedge(t-s_0-s) \right)^{(1+\eps)/2}
    \end{align*}
    and by our assumption on $x_0$ we have $x-\lambda s_0 \in [-2\sqrt{s_0},-\sqrt{s_0}/2]$ for $s_0$ large enough, so by 
    \cite[Lemma 2.8]{kim_maximum_2023}, for $t$ large enough,
    \begin{equation*}
        \P_{(0,x-\lambda s_0)}\left( 
        \forall s\in [0,t-s_0], X_s \in I_s
        \middle| X_{t-s_0}=y \right)
        \geq \frac{c}{t} (\lambda s_0 - x) (-y) 
        \geq \frac{c}{t}.
    \end{equation*}
    Combining this with \eqref{eq:hat1} and \eqref{eq:hat2}, the total $X$-contribution is
    \begin{equation}\label{eq:ban}
        e^{t-s_0} \P_{(s_0,x)}\left(B_1\cap B_2 \right) \geq c t^{-\kappa/4} e^{\vartheta_1 t^{1-\kappa}}.
    \end{equation}
    Going back to \eqref{eq:eve}, combining \eqref{eq:ban} with \eqref{eq:Y-contrib}, we get the desired result.
\end{proof}

\subsection{Second moment estimate}

\begin{lem}\label{lemma:second_moment_estimate}
For $s_0$ large enough, there exists $C > 0$ such that, for any $t$ large enough, for any $x\in [\sqrt{2}s_0 - 2\sqrt{s_0}, \sqrt{2}s_0 - \sqrt{s_0}]$, $y \in \R$,
    \[
    \E_{(s_0,x,y)} \left[ \abs{\Gamma_{s_0}(t)}^2 \right] \leq C.
    \]
\end{lem}

\begin{proof}
Throughout the proof, we absorb dependencies on $s_0$ into the constants. 
First, note that it is enough to prove $\E_{(s_0,x,y)}[ \abs{\Gamma_{s_0}(t)} (\abs{\Gamma_{s_0}(t)}-1)] \leq C$.
Indeed, we then have 
\begin{align*}
    \E_{(s_0,x,y)} \left[ \abs{\Gamma_{s_0}(t)}^2 \right] 
    &= \E_{(s_0,x,y)} \left[ \abs{\Gamma_{s_0}(t)} \right] 
    + \E_{(s_0,x,y)} \left[ \abs{\Gamma_{s_0}(t)} (\abs{\Gamma_{s_0}(t)}-1) \right] \\
    &\leq \E_{(s_0,x,y)} \left[ \abs{\Gamma_{s_0}(t)}^2 \right]^{1/2} + C,
\end{align*}
which yields the result because $x^2 \leq x + y$ implies $x \leq 1 + y^{1/2}$ for any $x,y \geq 0$.

To compute $\E_{(s_0,x,y)}[ \abs{\Gamma_{s_0}(t)} (\abs{\Gamma_{s_0}(t)}-1)]$, we want to apply the many-to-two lemma.
Recall that, under $\P_{(s_0,x,y)}$, $\xi^{1}$ and $\xi^{2,r}$ denote two Brownian motions on $\R^2$ starting from $(x,y)$ at time $s_0$, which are coinciding up to time $r$ and then moving independently.
Let $\xi_{1,t}^{i}$ and $\xi_{2,t}^{i}$ denote the $X$ and $Y$-coordinate of $\xi_{t}^{i}$ respectively, and define the events for $i\in\{1,2\}$,
\begin{align*}
    B_1^{(i)} &= \left\{ \xi^{i}_{1,t} \in [m(t) -1, m(t)] \right\},
    \hspace{1cm}
    B_2^{(i)} = \left\{ \forall s\in [s_0,t], \xi^{i}_{1,s} \in [g_t^-(s), g_t^+(s)] \right\}, \\
    B_3^{(i)} &= \left\{  \forall s\in [s_0,t], \lvert \xi^{i}_{2,s} \rvert \leq \sigma s  \right\}.
\end{align*}
Then, the many-to-two lemma (Lemma~\ref{lem:many_to_two}) yields
\begin{align}
    & \E_{(s_0,x,y)} \left[ \abs{\Gamma_{s_0}(t)} (\abs{\Gamma_{s_0}(t)}-1) \right] \nonumber \\
    & = \int_{s_0}^t \E_{(s_0,x,y)} \left[ \1_{\cap_{i=1}^2 \cap_{j=1}^3 B_j^{(i)}} b(\xi_t^1) 
    \exp\left(  \int_{s_0}^t b(\xi_s^1) \diff s  
    + \int_{r}^t b(\xi_s^{2,r}) \diff s \right) \right] 2 \diff r.
    \label{eq:many_to_two_used}
\end{align}
Choosing $s_0$ large enough, we have $g_t^+(s) \leq \sqrt{2} s$ for any $s \in [s_0,t]$, and therefore Lemma~\ref{lem:branching_rate_upper_bound} can be applied on the event $B_2^{(i)} \cap B_3^{(i)}$ to bound $b(\xi_s^i) \leq 1-\beta \vert \xi^i_{2,s} /{\sqrt{2}s}\vert^\alpha(1+f(\xi^i_{2,s},s))$, where $f$ satisfies the assumptions of Proposition~\ref{prop:bound_on_integrated_G_tilde}.\ref{it:UB_integrated_G_tilde}. 
This then completely decouples the $X$ and $Y$-coordinates of \eqref{eq:many_to_two_used}. 
We now bound the indicators of $B_3^{(i)}$ by $1$ and use $b(\xi_t^1)\leq 1$ to bound \eqref{eq:many_to_two_used} by
\begin{align}
    &\int_{s_0}^t e^{(t-s_0)}e^{(t-r)}\P_{(s_0,x)}\left(\bigcap_{i=1}^2\bigcap_{j=1}^2 B_j^{(i)} \right) \nonumber \\
    & \hspace{.5cm} \times \E_{(s_0,x,y)} \left[  \exp\left(  -\beta\int_{s_0}^t \abs{\frac{\xi^1_s}{\sqrt{2}s}} (1+f(\xi^1_s,s)) \diff s 
    -\beta \int_{r}^t \abs{\frac{\xi^{2,r}_s}{\sqrt{2}s}} (1+f(\xi^{2,r}_s,s)) \diff s \right) 
    \right] \diff r. 
    \label{eq:many_to_two_bound}
\end{align}
Recalling the definition of $\widetilde{G}$ from \eqref{eq:G_tilde}, the expectation in \eqref{eq:many_to_two_bound} equals
\begin{align}
    \int_\R \widetilde{G}(s_0,y;r,\hat{y})\bigg[\int_{\R} \widetilde{G}(r,\hat{y};t,z) \diff z \bigg]^2 \diff \hat{y} 
    & \leq C t^{\kappa/2} r^{-\kappa/4} e^{\vartheta_1 r^{1-\kappa}-2\vartheta_1 t^{1-\kappa}}
    \label{eq:from_y}
\end{align}
using Proposition~\ref{prop:bound_on_integrated_G_tilde}.\ref{it:UB_integrated_G_tilde} twice and incorporating the $s_0$ dependencies into $C$.
We now focus on the probability in \eqref{eq:many_to_two_bound}.
Integrating first w.r.t.\@ $\xi^{1,r}_r = \xi^{2,r}_r$ and weakening the barrier, we get
\begin{align}
    &\P_{(s_0,x)}\left(\bigcap_{i=1}^2\bigcap_{j=1}^2 B_j^{(i)} \right) \nonumber \\
    & \hspace{1cm} \leq \int_{-\infty}^{g_t^+(r)} 
    \P_{(s_0,x)}\left( B_2^{[s_0,r]} \middle| X_r = y \right) 
    \P_{(r,y)}\left(B_1 \cap  B_2^{[r,t]} \right)^2
    \frac{e^{-(y-x)^2/[2(r-s_0)]}}{\sqrt{2\pi (r-s_0)}} \diff y,
    \label{eq:from_x}
\end{align}
where $B_1 = \{ X_t \in [m(t)-1,m(t)] \}$ and $B_2^I = \{ \forall s\in I, X_s \leq \lambda s \}$ with $\lambda = m(t)/t$.
By \cite[Lemma 2]{Bra1978}, we have
\begin{equation*}
    \P_{(s_0,x)}\left( B_2^{[s_0,r]} \middle| X_r = y \right) 
    \leq \frac{2(\lambda s_0 - x)(\lambda r - y)}{r-s_0} \wedge 1
    \leq \frac{C (\overline{y}+1)}{(r-s_0+1)},
\end{equation*}
writing $\overline{y} = \lambda r - y$ and using the assumption on $x$.
On the other hand, by Girsanov's theorem, we have
\begin{align*}
    \P_{(r,y)}\left(B_1 \cap  B_2^{[r,t]} \right)
    & = \E_{(s_0,y-\lambda r)}\left[\ind{X_t \in [-1,0]} 
    \ind{\forall s \in [r,t], X_r \leq 0} 
    e^{-\lambda(X_t-X_{s_0}) -\frac{\lambda^2}{2} (t-r)} \right] \\
    & \leq \frac{C (\overline{y}+1)}{(t-r+1)^{3/2}} 
    e^{- \lambda \overline{y} -\frac{\lambda^2}{2} (t-r)},
\end{align*}
where we used $e^{-\lambda(X_t-X_{s_0})} \leq e^{- \lambda \overline{y}}$ and bounded the remaining probability by first integrating w.r.t.\@ $X_t$ and then applying \cite[Lemma 2]{Bra1978} again.
Coming back to \eqref{eq:from_x}, using also 
\[
    e^{-(y-x)^2/[2(r-s_0)]} 
    = e^{-(\lambda(r-s_0) -\overline{y}-x + \lambda s_0)^2/[2(r-s_0)]}
    \leq e^{-\frac{\lambda^2}{2}(r-s_0) + \lambda(\overline{y}+x -\lambda s_0)}
    \leq C e^{-\frac{\lambda^2}{2} r + \lambda \overline{y}},
\]
and changing $y$ for $\overline{y}$ in the integral, we obtain
\begin{align*}
    \P_{(s_0,x)}&\left(\bigcap_{i=1}^2\bigcap_{j=1}^2 B_j^{(i)} \right) \\
    & \leq \frac{C e^{-\lambda^2 t + \frac{\lambda^2}{2}r}}{(r-s_0+1)\sqrt{r-s_0} (t-r+1)^3}
    \int_{((r-s_0) \wedge(t-r))^{(1-\eps)/2}}^\infty e^{-\lambda \overline{y}} (\overline{y}+1)^3 \diff \overline{y} \\
    & \leq \frac{C t^{3-\frac{\kappa}{2}} e^{-2t + 2\vartheta_1 t^{-\kappa}} e^{r-\vartheta_1 r t^{-\kappa}-(\frac{3}{2}-\frac{\kappa}{4}) \frac{r}{t} \log(t)}}{(r-s_0+1)\sqrt{r-s_0} (t-r+1)^3}
    e^{-((r-s_0) \wedge (t-r))^{(1-\eps)/2}},
\end{align*}
using $\frac{\lambda^2}{2} = 1 - \vartheta_1 t^{-\kappa} - (\frac{3}{2}-\frac{\kappa}{4}) \frac{\log t}{t} + o(\frac{1}{t})$ for the prefactor, and $e^{-\lambda \overline{y}} (\overline{y}+1)^3 \leq C e^{-\overline{y}}$ for the integral.
Combining this with \eqref{eq:from_y},
we get that \eqref{eq:many_to_two_bound} is at most
\begin{equation}
    C t^3 
    \int_{s_0}^t r^{-\kappa/4} \frac{e^{\vartheta_1 (r^{1-\kappa}-r t^{-\kappa})-(\frac{3}{2}-\frac{\kappa}{4}) \frac{r}{t} \log(t)}}{(r-s_0+1)\sqrt{r-s_0} (t-r+1)^3}
    e^{- ((r-s_0) \wedge (t-r))^{(1-\eps)/2}}
    \diff r. 
    \label{eq:many_to_two_bound2}
\end{equation}
Then, using that $s \geq s_0 \mapsto (\log s)/s$ is decreasing for $s_0$ large enough, we have $e^{\frac{\kappa}{4} \frac{r}{t} \log(t)} \leq r^{\kappa/4}$ and $e^{-\frac{3}{2} \frac{r}{t} \log(t)} = t^{-3/2} e^{\frac{3}{2} \frac{t-r}{t} \log(t)} \leq t^{-3/2} (t-r)^{3/2}$.
Moreover, there exists $C(\kappa)>0$ such that $r^{1-\kappa}-r t^{-\kappa} \leq C (r \wedge (t-r))^{1-\kappa}$ for any $r \in [s_0,t]$: this is direct for $r \leq t/2$ and, for $r \geq t/2$, one has 
\[
r^{1-\kappa}-r t^{-\kappa} 
= r^{1-\kappa} \bigg(1-\big(1-\frac{t-r}{t}\big)^\kappa\bigg)
\leq C(\kappa) t^{1-\kappa} \frac{t-r}{t} 
\leq C(\kappa) (t-r)^{1-\kappa}.
\]
It follows that $e^{\vartheta_1 (r^{1-\kappa}-r t^{-\kappa})- ((r-s_0) \wedge (t-r))^{(1-\eps)/2}} \leq C$, where we recall that $C$ can depend on $s_0$. 
Applying these facts to \eqref{eq:many_to_two_bound2} shows
\begin{equation*}
    \E_{(s_0,x,y)} \left[ \abs{\Gamma_{s_0}(t)} (\abs{\Gamma_{s_0}(t)}-1) \right] 
    \leq C t^{3/2}
    \int_{s_0}^t \frac{1}{(r-s_0+1)\sqrt{r-s_0} (t-r+1)^{3/2}}
    \diff r
    \leq C,
\end{equation*}
which concludes the proof.
\end{proof}

\subsection{Proof of the lower bound}
\label{subsec:LB}

We are now in a position to show the lower bound of Theorem \ref{thm:main}. 
To be precise, we show that, for any $\delta>0$, there exists $x\in \R$ such that
\begin{equation}\label{eq:lower_bound_goal}
    \limsup_{t \to \infty} \P \left(M_t \leq m(t) - x \right) \leq \delta.
\end{equation}

\begin{proof}[Proof of the lower bound in Theorem \ref{thm:main}]
    We use the moment bounds on $\vert \Gamma_K(t) \vert$ to show that there are particles near $m(t)$ at time $t$. 
    For $s_0 \geq 0$, we consider the box 
    \begin{equation} \label{eq:def_B_s_0}
    \cB_{s_0} = [\sqrt{2}s_0 - 2\sqrt{s_0}, \sqrt{2}s_0 - \sqrt{s_0}] \times [-s_0^{\kappa/2},s_0^{\kappa/2}].
    \end{equation}
    In the sequel, we consider $s_0$ large enough to satisfy Lemmas \ref{lemma:first_momement_lower_bound} and \ref{lemma:second_moment_estimate} as well as other properties mentioned below.

    Combining Lemmas \ref{lemma:first_momement_lower_bound} and \ref{lemma:second_moment_estimate}, as well as the Cauchy-Schwarz inequality, there exist $t_0>s_0$ and $c_0 >0$ such that for all $t \geq t_0$ and $(x,y) \in \cB_{s_0}$,
    \begin{align*}
        \P_{(s_0,x,y)}\left( M_t \geq m(t) -1 \right) 
        \geq \P_{(s_0,x,y)}\left(\vert \Gamma_K(t) \vert \geq 1 \right)
        \geq \frac{\E_{(s_0,x,y)}[\vert \Gamma_K(t) \vert]^2}{\E_{(s_0,x,y)}[\vert \Gamma_K(t) \vert^2]}\geq c_0 >0.
    \end{align*}
    From this it follows that we also have, for all $t \geq t_0-s_0$ and $(x,y) \in \cB_{s_0}$,
    \begin{equation} \label{eq:lower_bound_box}
        \P_{(0,x,y)}\left( M_t \geq m(t) -1 \right) 
        \geq \P_{(s_0,x,y)}\left(M_{t+s_0} \geq m(t+s_0) -1 \right) 
        \geq c_0 > 0,
    \end{equation}
    using that $m(t+s_0) \geq m(t)$, if $s_0$ is chosen large enough.

    Let $\delta >0$. By Lemma \ref{lem:box} below, for every $N \in \N$, there is $r>0$ such that 
    \begin{equation} \label{eq:claim_box}
        \P_{(0,0,0)}\left(\# \left\{u\in \mathcal{N}(r) : (X_u(r),Y_u(r)) \in \cB_{s_0} \right\} \geq N \right) \geq 1-\delta/2. 
    \end{equation}
    Hence, restricting ourselves to the event in \eqref{eq:claim_box}, applying the branching property at time $r$ and then \eqref{eq:lower_bound_box} to the BBMs starting from particles in $\cB_{s_0}$ at time $r$, we get, for any $t \geq t_0-s_0+r$,
    \begin{equation*}
        \P_{(0,0,0)}\left(M_t \geq m(t-r) - 1\right) \geq (1-\delta/2)\left(1- (1-c_0)^N \right) \geq 1-\delta,
    \end{equation*}
    for $N=N(\delta, c_0)$ large enough. 
    And therefore we have
    \begin{equation*}
        \P_{(0,0,0)}\left( M_t \geq m(t) - \sqrt{2}r - 1\right) \geq 1-\delta,
    \end{equation*}
    for $t$ large enough such that $m(t-r) \geq m(t) - \sqrt{2}r$.
    This proves \eqref{eq:lower_bound_goal}.
\end{proof}

\begin{lem} \label{lem:box}
    For $s_0$ large enough, recalling the definition of $\cB_{s_0}$ in \eqref{eq:def_B_s_0}, we have $\P_{(0,0,0)}$-a.s.
    \[
        \# \left\{u\in \mathcal{N}(t) : (X_u(t),Y_u(t)) \in \cB_{s_0} \right\}
        \xrightarrow[t\to\infty]{} \infty
    \]    
\end{lem}

\begin{proof}
    As $s_0 \to \infty$, the angular coordinate of points in $\cB_{s_0}$ is uniformly going to 0, so, using Assumption \ref{assumption2}, we can choose $s_0$ large enough so that the branching rate is at least $1/2$ on $\cB_{s_0}$.
    We also introduce a family of square boxes, indexed by $L>0$:
    \begin{equation} \label{eq:def_B_s_0_L}
    \cB_{s_0}^L = \left( \sqrt{2}s_0 - \frac{3}{2} \sqrt{s_0} + [-L,L] \right) \times [-L,L].
    \end{equation}
    We also assume that $s_0$ is large enough such that $\cB_{s_0}^4 \subset \cB_{s_0}$. We now fix $s_0$.
    
    We define a spine $(\xi_t)_{t\geq 0}$ with $\xi_t \in \cN(t)$, starting with the root and then choosing one of the two children uniformly at random at each branching point. We now make the two following claims
    \begin{itemize}
        \item \textit{Claim 1:} The spine branches infinitely many times in the box $\cB_{s_0}^2$ almost surely.
        \item \textit{Claim 2:} There exists $c>0$ such that, for all $(x,y) \in \cB_{s_0}^2$,
        \begin{equation} \label{eq:claim2}
            \P_{(0,x,y)} \left( \# \left\{u\in \mathcal{N}(t) : (X_u(t),Y_u(t)) \in \cB_{s_0}^4 \right\} \xrightarrow[t\to\infty]{} \infty \right)
            \geq c.
        \end{equation}
    \end{itemize}
    Combining these two claims together with the branching property along the spine and the second Borel--Cantelli lemma yields the result.

    We now prove the claims. For Claim 1, note that the trajectory of the spine is a two-dimensional Brownian motion, which is recurrent, so a.s.\@ the spine enters $\cB_{s_0}^1$, then leaves $\cB_{s_0}^2$, then enters $\cB_{s_0}^1$, then leaves $\cB_{s_0}^2$, and so on infinitely many times. Because the branching rate is at least $1/2$ on $\cB_{s_0}^2$ as mentioned at the beginning of the proof, the probability that the spine branches between an entry time of $\cB_{s_0}^1$ and the next exit time of $\cB_{s_0}^2$ is bounded away from 0 uniformly in the entry point in $\cB_{s_0}^1$. Therefore, Claim 1 follows from the second Borel--Cantelli lemma.

    We now deal with Claim 2. Using that the branching rate is at least $1/2$ on $\cB_{s_0}^4$, it is enough to prove the claim for a BBM with branching rate $1/2$ and where particles are killed when leaving $\cB_{s_0}^4$. But this branching Markov process is supercritical, because the largest eigenvalue of the Laplacian with Dirichlet boundary condition on $\cB_{s_0}^4$ is $-\frac{1}{2} (\frac{\pi^2}{4^2})$, which is larger than the opposite of the branching rate. The fact of being supercritical implies the claim as shown in \cite{Sev1958} or \cite{Wat1965}. More precisely, let $u_1(x,y)$ be the extinction probability when starting from $(x,y) \in \cB_{s_0}^4$. Then, \cite[Theorem 2.1]{Wat1965} implies that $u_1$ is continuous on $\cB_{s_0}^4$ and \cite[Theorem 2.2]{Wat1965} shows that $u_1 <1$ on the interior of $\cB_{s_0}^4$, so there exists $c >0$ such that, for all $(x,y) \in \cB_{s_0}^2$, $u_1(x,y) \leq 1-c$. But \cite[Lemma 2.1]{Wat1965} proves that $1-u_1(x,y)$ equals the probability in \eqref{eq:claim2}.
\end{proof}

\section{Proof of Proposition~\ref{porism}} 
\label{sec:porism}

\begin{proof}[Proof of Proposition~\ref{porism}]
	Part~\ref{it:Y}. This is a consequence of Lemmas \ref{lem:theta} and \ref{lem:control_Y}.
	
	Part~\ref{it:angle}. Lemma \ref{lem:theta} also implies that, for any $a \geq 0$,
	\begin{equation} \label{eq:angle_crude}
	\Pp{ \exists u \in \cN(t) : 
		R_u(t) \geq m(t)-a \text{ and } \abs{\theta_u(t)} > \frac{\pi}{2} }
	\xrightarrow[t\to\infty]{} 0.
	\end{equation}
	But, if $u\in\cN(t)$ satisfies $R_u(t) \geq m(t)-a$ and $\abs{\theta_u(t)} \geq t^{-\frac{2}{2+\alpha} +\varepsilon}$, then we have either $\abs{\theta_u(t)} > \frac{\pi}{2}$, or, in the opposite case, 
	\[
	\abs{Y_u(t)} 
	\geq R_u(t) \sin \abs{\theta_u(t)} 
	\geq R_u(t) \cdot \frac{2}{\pi} \abs{\theta_u(t)} 
	\geq (m(t)-a) \cdot \frac{2}{\pi} t^{-\frac{2}{2+\alpha} +\varepsilon}
	\geq t^{\frac{\alpha}{2+\alpha} + \frac{\varepsilon}{2}},
	\]
	for $t$ large enough. Therefore, Part~\ref{it:angle} follows from \eqref{eq:angle_crude} and Part~\ref{it:Y}.
	
	Part~\ref{it:X}. Fix $\varepsilon \in (0,\frac{2}{2+\alpha}-\frac{1}{2})$, which is non-empty because $\alpha<2$.
	Let $a \geq 0$. By Theorem \ref{thm:main} and Part~\ref{it:angle}, we have
	\[
	\Pp{ \left\{ \abs{M_t-m(t)} \leq a \right\} \cap
	\left\{ \exists u \in \cN(t) : 
		R_u(t) \geq m(t)-a,
		\abs{\theta_u(t)} \geq t^{-\frac{2}{2+\alpha} +\varepsilon} \right\} }
	\xrightarrow[t,a \to \infty]{} 1,
	\]
	where in the limit we first take $t \to \infty$ and then $a \to \infty$.
	We now work on the event in the probability above.
	Then, if $v$ denote the particle with the maximal displacement, then $R_v(t) = M_t \geq m(t)-a$ so $\abs{\theta_v(t)} \leq t^{-\frac{2}{2+\alpha} +\varepsilon}$ and thus
	\[
		X_v(t) = R_v(t) \cos \theta_v(t) 
		\geq R_v(t) \left( 1 - \frac{\theta_v(t)^2}{2} \right)
		\geq R_v(t) - (m(t)+a) \cdot \frac{t^{-\frac{4}{2+\alpha} +2\varepsilon}}{2},
	\]
	using that $R_v(t) = M_t \leq m(t)+a$.
	By our choice of $\varepsilon$, we have $t^{-\frac{4}{2+\alpha} +2\varepsilon} = o(t^{-1})$ so, for any $\delta>0$, $X_v(t) \geq R_v(t) - \delta$ for $t$ large enough.
	Altogether, this proves that
	\[
		\P \left( \max_{u\in \cN(t)} X_u(t) \geq M_t-\delta \right) \xrightarrow[t \to \infty]{} 1.
	\]
	Noting that $\max_{u\in \cN(t)} X_u(t) \leq M_t$, Part~\ref{it:X} follows.
\end{proof}

\begin{rem} \label{rem:optimal_exponent}
	We explain here why the exponent $\alpha/(2+\alpha) = \kappa/2$ appearing in Proposition~\ref{porism}.\ref{it:Y} is optimal (and therefore so is the exponent appearing in Proposition~\ref{porism}.\ref{it:angle}). 
	The key reason is that the $Y$-coordinate of an extremal particle is governed by the kernel $\widetilde{G}$ studied in Section~\ref{sec:probabilistic_arguments}, which is itself close to the kernel $G$ studied in Section~\ref{sec:BM_weighted}, and the density $G(s,x;t,\cdot)$ is supported by points on scale $t^{\kappa/2}$, see Proposition~\ref{prop:estimate_G}.
	
	More precisely, it follows from Proposition~\ref{prop:estimate_G} that there exist $K,C,c > 0$ such that, for any $s \geq K$, $t \geq s + Ks^\kappa$, $x \in \R$ and $\eta>0$,
	\begin{equation} \label{eq:new_bound_G}
		\int_{\abs{y} \leq \eta t^{\kappa/2}} G(s,x;t,y) \diff y
		\leq C \eta (t/s)^{\kappa/4} \exp\left(\vartheta_1 (s^{1-\kappa}-t^{1-\kappa})\right).
	\end{equation}
	Note the additional factor $\eta$ compared to Corollary \ref{cor:estimate_G}.\ref{it:UB_integrated_G}: when $\eta$ is small this integral is negligible w.r.t.\@ the integral on the full line.
	Then, under the same assumptions as in Proposition~\ref{prop:bound_on_integrated_G_tilde}.\ref{it:UB_integrated_G_tilde}, by following its proof, one can check that the same holds with $\widetilde{G}$ instead of $G$.
	
	With this bound, one can then prove that, for any $a \geq 0$ and $\varepsilon>0$,
	\begin{equation} \label{eq:lower_bound_extremal_Y}
	\Pp{ \exists u \in \cN(t) : 
		R_u(t) \geq m(t)-a,
		\abs{Y_u(t)} \leq t^{\kappa/2 -\varepsilon} }
	\xrightarrow[t\to\infty]{} 0,
	\end{equation}
	hence showing optimality of the exponent $\kappa/2$ appearing in Proposition~\ref{porism}.\ref{it:Y}.
	We now sketch the proof of \eqref{eq:lower_bound_extremal_Y}.
	By the same ideas as in the proof of Corollary~\ref{cor:localization}, it is enough to prove $\P(E') \to 0$ as $t\to\infty$, where 
	\begin{align*}
	&E' \coloneqq A_{s_0,t/2} \cap \biggl\{ \exists u \in \cN(t) : 
	X_u(t) \geq m(t)-a, 
	\max_{s\in [t/2,t]} \abs{\theta_u(s)} \leq \sigma,  \\
	&\hspace{6cm} \max_{s \in [t/2,t]} X_u(s) \leq m^+(s),
	\abs{Y_u(t)} \leq t^{\kappa/2 -\varepsilon} \biggr\}.
	\end{align*}
	The bound of $\P(E')$ follows the steps of the proof of the upper bound in Theorem~\ref{thm:main} in Section~\ref{sec:proof_UB}.
	Note that the difference between $E'$ and the event $E$ considered there is that $+a$ has been replaced by $-a$ and the constraint on $\abs{Y_u(t)}$ has been added.
	In Step 1 (between times $t_0 = t-t^\delta$ and $t$), one can argue that the event where $\{\abs{Y_u(t) - Y_u(t_0)} > 2(t-t_0)\}$ would have a negligible contribution (using that 2 is larger than the speed $\sqrt{2}$ of the maximum of a standard 1d BBM). On the opposite event, we can replace the control $\abs{Y_u(t)} \leq t^{\kappa/2 -\varepsilon}$ by $\abs{Y_u(t_0)} \leq 3t^{\kappa/2 -\varepsilon}$ if we choose $\delta= \kappa/2 -\varepsilon$ (we can assume that $\varepsilon$ is small enough such that $\delta>0$).
	Then, Step 2 (between times $t/2$ and $t_0$) is similar except for the $Y$-contribution which is replaced by
	\begin{align*}
	\E_{(t/2,X_u(t/2), Y_u(t/2))}\Biggl[ 
	\exp \left( - \beta \int_{t/2}^{t_0} \abs{\frac{Y_s}{\sqrt{2} s}}^\alpha (1+f(Y_s,s)) \diff s \right) 
	\1_{\abs{Y_u(t_0)} \leq 3t^{\kappa/2 -\varepsilon}}
	\Biggr].
	\end{align*}
	This is bounded using \eqref{eq:new_bound_G} (with $\widetilde{G}$ instead of $G$), resulting in the gain of a factor $t^{-\varepsilon}$. thus, a bound similar to \eqref{eq:second_conditioning_on_t/2} holds for $E'$, with $-a$ replaced by $+a$ and an additional factor $t^{-\varepsilon}$ and the conclusion follows.
\end{rem}

\section*{Acknowledgements}

We are grateful to the Centre de Recherche Mathématiques (CRM) in Montréal for hosting the semester ``Probability and PDE'' in Fall 2023, as well as the organizers of this program. The project of this paper came about during this semester, based on a suggestion of Nina Gantert, to whom we would like to express our warmest thanks. Julien Berestycki and Michel Pain  acknowledge support from the Simons Foundation over this period. Michel Pain is also grateful for the support of the  PEPS JCJC grant n°246644 and the MITI interdisciplinary
program 80PRIME GEx-MBB. David Geldbach is also grateful for the support of EPSRC grant EP/W523781/1. David Geldbach would also like to express his gratitude to mathoverflow users Mathias Rousset and Thomas Kojar for help with his questions.
Finally, we are grateful to two anonymous referees whose remarks have improved the presentation of this paper.

\addcontentsline{toc}{section}{References}
\bibliographystyle{abbrv}
\bibliography{biblio}

@article {AidBerBruShi2013,
    AUTHOR = {A\"{\i}d\'{e}kon, E. and Berestycki, J. and Brunet, \'{E}. and Shi, Z.},
     TITLE = {Branching {B}rownian motion seen from its tip},
   JOURNAL = {Probab. Theory Related Fields},
  FJOURNAL = {Probability Theory and Related Fields},
    VOLUME = {157},
      YEAR = {2013},
    NUMBER = {1-2},
     PAGES = {405--451},
      ISSN = {0178-8051},
   MRCLASS = {60J80 (60G55 60G70)},
  MRNUMBER = {3101852},
MRREVIEWER = {J\'{a}nos Engl\"{a}nder},
       DOI = {10.1007/s00440-012-0461-0},
       URL = {https://doi.org/10.1007/s00440-012-0461-0},
}

@article{alban_1_2025,
	title = {From 1 to infinity: {The} log-correction for the maximum of variable speed branching {Brownian} motion},
	volume = {30},
	shorttitle = {From 1 to infinity},
	doi = {10.1214/25-EJP1287},
	journal = {Electronic Journal of Probability},
	author = {Alban, Alexander and Bovier, Anton and Gros, Annabell and Hartung, Lisa},
	year = {2025},
	pages = {1--46},
}

@article {ArgBovKis2012,
    AUTHOR = {Arguin, Louis-Pierre and Bovier, Anton and Kistler, Nicola},
     TITLE = {Poissonian statistics in the extremal process of branching
              {B}rownian motion},
   JOURNAL = {Ann. Appl. Probab.},
  FJOURNAL = {The Annals of Applied Probability},
    VOLUME = {22},
      YEAR = {2012},
    NUMBER = {4},
     PAGES = {1693--1711},
      ISSN = {1050-5164},
   MRCLASS = {60J80 (60G70 82B44)},
  MRNUMBER = {2985174},
MRREVIEWER = {Leonid V. Bogachev},
       DOI = {10.1214/11-AAP809},
       URL = {https://doi.org/10.1214/11-AAP809},
}

@article {ArgBovKis2013,
    AUTHOR = {Arguin, Louis-Pierre and Bovier, Anton and Kistler, Nicola},
     TITLE = {The extremal process of branching {B}rownian motion},
   JOURNAL = {Probab. Theory Related Fields},
  FJOURNAL = {Probability Theory and Related Fields},
    VOLUME = {157},
      YEAR = {2013},
    NUMBER = {3-4},
     PAGES = {535--574},
      ISSN = {0178-8051},
   MRCLASS = {60J80 (60G70 60J65)},
  MRNUMBER = {3129797},
MRREVIEWER = {Anthony G. Pakes},
       DOI = {10.1007/s00440-012-0464-x},
       URL = {https://doi.org/10.1007/s00440-012-0464-x},
}

@article {AroBes1967,
	AUTHOR = {Aronson, D. G. and Besala, P.},
	TITLE = {Parabolic equations with unbounded coefficients},
	JOURNAL = {J. Differential Equations},
	FJOURNAL = {Journal of Differential Equations},
	VOLUME = {3},
	YEAR = {1967},
	PAGES = {1--14},
	ISSN = {0022-0396},
	MRCLASS = {35.62},
	MRNUMBER = {208160},
	MRREVIEWER = {R. Barrar},
	DOI = {10.1016/0022-0396(67)90002-2},
	URL = {https://doi.org/10.1016/0022-0396(67)90002-2},
}

@book {BenBroWei2020,
    AUTHOR = {Bennewitz, Christer and Brown, Malcolm and Weikard, Rudi},
     TITLE = {Spectral and scattering theory for ordinary differential
              equations. {V}ol. {I}: {S}turm-{L}iouville equations},
    SERIES = {Universitext},
 PUBLISHER = {Springer, Cham},
      YEAR = {2020},
     PAGES = {ix+379},
      ISBN = {978-3-030-59088-8; 978-3-030-59087-1},
   MRCLASS = {34-02 (28A35 33E30 34B24 34Lxx 46N20 47N20 81U20)},
  MRNUMBER = {4199125},
       DOI = {10.1007/978-3-030-59088-8},
       URL = {https://doi.org/10.1007/978-3-030-59088-8},
}

@article {BerBruGraMytRoqRyz2022,
    AUTHOR = {Berestycki, Julien and Brunet, \'{E}ric and Graham, Cole and Mytnik, Leonid and Roquejoffre, Jean-Michel and Ryzhik, Lenya},
     TITLE = {The distance between the two {BBM} leaders},
   JOURNAL = {Nonlinearity},
  FJOURNAL = {Nonlinearity},
    VOLUME = {35},
      YEAR = {2022},
    NUMBER = {4},
     PAGES = {1558--1609},
      ISSN = {0951-7715},
   MRCLASS = {60J80 (35K57)},
  MRNUMBER = {4385443},
       DOI = {10.1088/1361-6544/ac4a8e},
       URL = {https://doi.org/10.1088/1361-6544/ac4a8e},
}

@article {BerBruHarHar2010,
    AUTHOR = {Berestycki, J. and Brunet, \'{E}. and Harris, J. W. and Harris, S. C.},
     TITLE = {The almost-sure population growth rate in branching {B}rownian
              motion with a quadratic breeding potential},
   JOURNAL = {Statist. Probab. Lett.},
  FJOURNAL = {Statistics \& Probability Letters},
    VOLUME = {80},
      YEAR = {2010},
    NUMBER = {17-18},
     PAGES = {1442--1446},
      ISSN = {0167-7152},
   MRCLASS = {60J80 (60J65)},
  MRNUMBER = {2669786},
MRREVIEWER = {Maroussia N. Slavtchova-Bojkova},
       DOI = {10.1016/j.spl.2010.05.011},
       URL = {https://doi.org/10.1016/j.spl.2010.05.011},
}

@article {BerBruHarHarRob2015,
    AUTHOR = {Berestycki, Julien and Brunet, \'{E}ric and Harris, John W. and
              Harris, Simon C. and Roberts, Matthew I.},
     TITLE = {Growth rates of the population in a branching {B}rownian
              motion with an inhomogeneous breeding potential},
   JOURNAL = {Stochastic Process. Appl.},
  FJOURNAL = {Stochastic Processes and their Applications},
    VOLUME = {125},
      YEAR = {2015},
    NUMBER = {5},
     PAGES = {2096--2145},
      ISSN = {0304-4149},
   MRCLASS = {60J80 (60F10 60J65 92D25)},
  MRNUMBER = {3315624},
       DOI = {10.1016/j.spa.2014.12.008},
       URL = {https://doi.org/10.1016/j.spa.2014.12.008},
}

@article {BerKimLubMalZei2024,
    AUTHOR = {Berestycki, Julien and Kim, Yujin H. and Lubetzky, Eyal and
              Mallein, Bastien and Zeitouni, Ofer},
     TITLE = {The extremal point process of branching {B}rownian motion in
              {$\Bbb R^d$}},
   JOURNAL = {Ann. Probab.},
  FJOURNAL = {The Annals of Probability},
    VOLUME = {52},
      YEAR = {2024},
    NUMBER = {3},
     PAGES = {955--982},
      ISSN = {0091-1798},
   MRCLASS = {60J80 (60G15 60G70 60J60)},
  MRNUMBER = {4736697},
       DOI = {10.1214/23-aop1677},
       URL = {https://doi.org/10.1214/23-aop1677},
}

@article{BocHar2014,
    AUTHOR = {Bocharov, Sergey and Harris, Simon C.},
     TITLE = {Branching {B}rownian motion with catalytic branching at the origin},
   JOURNAL = {Acta Appl. Math.},
  FJOURNAL = {Acta Applicandae Mathematicae},
    VOLUME = {134},
      YEAR = {2014},
     PAGES = {201--228},
      ISSN = {0167-8019},
   MRCLASS = {60J80 (60F15 60J65)},
  MRNUMBER = {3273694},
MRREVIEWER = {Zakhar Kabluchko},
       DOI = {10.1007/s10440-014-9879-y},
       URL = {https://doi.org/10.1007/s10440-014-9879-y},
}

@article {BocHar2016,
    AUTHOR = {Bocharov, Sergey and Harris, Simon C.},
     TITLE = {Limiting distribution of the rightmost particle in catalytic
              branching {B}rownian motion},
   JOURNAL = {Electron. Commun. Probab.},
  FJOURNAL = {Electronic Communications in Probability},
    VOLUME = {21},
      YEAR = {2016},
     PAGES = {Paper No. 70, 12},
   MRCLASS = {60J55 (60J65 60J80)},
  MRNUMBER = {3564217},
MRREVIEWER = {Yueyun Hu},
       DOI = {10.1214/16-ECP22},
       URL = {https://doi.org/10.1214/16-ECP22},
}

@article{BocWan2019,
    AUTHOR = {Bocharov, Sergey and Wang, Li},
     TITLE = {Branching {B}rownian motion with spatially homogeneous and point-catalytic branching},
   JOURNAL = {J. Appl. Probab.},
  FJOURNAL = {Journal of Applied Probability},
    VOLUME = {56},
      YEAR = {2019},
    NUMBER = {3},
     PAGES = {891--917},
      ISSN = {0021-9002},
   MRCLASS = {60J80 (60F15)},
  MRNUMBER = {4015642},
MRREVIEWER = {P\'{e}ter Kevei},
       DOI = {10.1017/jpr.2019.51},
       URL = {https://doi.org/10.1017/jpr.2019.51},
}

@book{borodin_handbook_2002,
    AUTHOR = {Borodin, Andrei N. and Salminen, Paavo},
     TITLE = {Handbook of {B}rownian motion---facts and formulae},
    SERIES = {Probability and its Applications},
   EDITION = {Second},
 PUBLISHER = {Birkh\"{a}user Verlag, Basel},
      YEAR = {2002},
     PAGES = {xvi+672},
      ISBN = {3-7643-6705-9},
   MRCLASS = {60-00 (60H05 60J25 60J55 60J60 60J65)},
  MRNUMBER = {1912205},
MRREVIEWER = {S\'{a}ndor Cs\"{o}rg\H{o}},
       DOI = {10.1007/978-3-0348-8163-0},
       URL = {https://doi.org/10.1007/978-3-0348-8163-0},
}

@article {BovKur2004,
    AUTHOR = {Bovier, Anton and Kurkova, Irina},
     TITLE = {Derrida's generalized random energy models. {II}. {M}odels
              with continuous hierarchies},
   JOURNAL = {Ann. Inst. H. Poincar\'{e} Probab. Statist.},
  FJOURNAL = {Annales de l'Institut Henri Poincar\'{e}. Probabilit\'{e}s et
              Statistiques},
    VOLUME = {40},
      YEAR = {2004},
    NUMBER = {4},
     PAGES = {481--495},
      ISSN = {0246-0203},
   MRCLASS = {82B44 (60G70 60K35)},
  MRNUMBER = {2070335},
MRREVIEWER = {Giambattista Giacomin},
       DOI = {10.1016/j.anihpb.2003.09.003},
       URL = {https://doi.org/10.1016/j.anihpb.2003.09.003},
}

@article{BovHar2014,
    AUTHOR = {Bovier, Anton and Hartung, Lisa},
     TITLE = {The extremal process of two-speed branching {B}rownian motion},
   JOURNAL = {Electron. J. Probab.},
  FJOURNAL = {Electronic Journal of Probability},
    VOLUME = {19},
      YEAR = {2014},
     PAGES = {no. 18, 28},
   MRCLASS = {60J80 (60G70 60K35)},
  MRNUMBER = {3164771},
       DOI = {10.1214/EJP.v19-2982},
       URL = {https://doi.org/10.1214/EJP.v19-2982},
}

@article {BovHar2015,
    AUTHOR = {Bovier, Anton and Hartung, Lisa},
     TITLE = {Variable speed branching {B}rownian motion 1. {E}xtremal
              processes in the weak correlation regime},
   JOURNAL = {ALEA Lat. Am. J. Probab. Math. Stat.},
  FJOURNAL = {ALEA. Latin American Journal of Probability and Mathematical
              Statistics},
    VOLUME = {12},
      YEAR = {2015},
    NUMBER = {1},
     PAGES = {261--291},
   MRCLASS = {60J80 (60G70)},
  MRNUMBER = {3351476},
MRREVIEWER = {In\'{e}s Maria del Puerto},
}

@article {BovHar2020,
    AUTHOR = {Bovier, Anton and Hartung, Lisa},
     TITLE = {From 1 to 6: a finer analysis of perturbed branching {B}rownian motion},
   JOURNAL = {Comm. Pure Appl. Math.},
  FJOURNAL = {Communications on Pure and Applied Mathematics},
    VOLUME = {73},
      YEAR = {2020},
    NUMBER = {7},
     PAGES = {1490--1525},
      ISSN = {0010-3640},
   MRCLASS = {60J80 (60J65)},
  MRNUMBER = {4156608},
MRREVIEWER = {Xiaowen Zhou},
       DOI = {10.1002/cpa.21893},
       URL = {https://doi.org/10.1002/cpa.21893},
}

@article {Bra1978,
    AUTHOR = {Bramson, Maury D.},
     TITLE = {Maximal displacement of branching {B}rownian motion},
   JOURNAL = {Comm. Pure Appl. Math.},
  FJOURNAL = {Communications on Pure and Applied Mathematics},
    VOLUME = {31},
      YEAR = {1978},
    NUMBER = {5},
     PAGES = {531--581},
      ISSN = {0010-3640},
   MRCLASS = {60J80},
  MRNUMBER = {494541},
MRREVIEWER = {S\o ren Asmussen},
       DOI = {10.1002/cpa.3160310502},
       URL = {https://doi.org/10.1002/cpa.3160310502},
}

@article{Bra1983,
    AUTHOR = {Bramson, Maury},
     TITLE = {Convergence of solutions of the {K}olmogorov equation to travelling waves},
   JOURNAL = {Mem. Amer. Math. Soc.},
  FJOURNAL = {Memoirs of the American Mathematical Society},
    VOLUME = {44},
      YEAR = {1983},
    NUMBER = {285},
     PAGES = {iv+190},
      ISSN = {0065-9266},
   MRCLASS = {60J65 (35K55 92A10)},
  MRNUMBER = {705746},
MRREVIEWER = {Kazuaki Taira},
       DOI = {10.1090/memo/0285},
       URL = {https://doi.org/10.1090/memo/0285},
}

@article {CarHu2014,
    AUTHOR = {Carmona, Philippe and Hu, Yueyun},
     TITLE = {The spread of a catalytic branching random walk},
   JOURNAL = {Ann. Inst. Henri Poincar\'{e} Probab. Stat.},
  FJOURNAL = {Annales de l'Institut Henri Poincar\'{e} Probabilit\'{e}s et
              Statistiques},
    VOLUME = {50},
      YEAR = {2014},
    NUMBER = {2},
     PAGES = {327--351},
      ISSN = {0246-0203},
   MRCLASS = {60K37},
  MRNUMBER = {3189074},
MRREVIEWER = {Matthias Meiners},
       DOI = {10.1214/12-AIHP529},
       URL = {https://doi.org/10.1214/12-AIHP529},
}

@article {ChaUri2011,
    AUTHOR = {Chaumont, Lo\"{\i}c and Uribe Bravo, Ger\'{o}nimo},
     TITLE = {Markovian bridges: weak continuity and pathwise constructions},
   JOURNAL = {Ann. Probab.},
  FJOURNAL = {The Annals of Probability},
    VOLUME = {39},
      YEAR = {2011},
    NUMBER = {2},
     PAGES = {609--647},
      ISSN = {0091-1798},
   MRCLASS = {60J25 (60J65)},
  MRNUMBER = {2789508},
MRREVIEWER = {Abhay G. Bhatt},
       DOI = {10.1214/10-AOP562},
       URL = {https://doi.org/10.1214/10-AOP562},
}

@article {CorHarLou2019,
    AUTHOR = {Cortines, Aser and Hartung, Lisa and Louidor, Oren},
     TITLE = {The structure of extreme level sets in branching {B}rownian motion},
   JOURNAL = {Ann. Probab.},
  FJOURNAL = {The Annals of Probability},
    VOLUME = {47},
      YEAR = {2019},
    NUMBER = {4},
     PAGES = {2257--2302},
      ISSN = {0091-1798},
   MRCLASS = {60J80 (60G15 60G70)},
  MRNUMBER = {3980921},
       DOI = {10.1214/18-AOP1308},
       URL = {https://doi.org/10.1214/18-AOP1308},
}

@article{HarLouWu2024,
    author = {Hartung, Lisa and Louidor, Oren and Wu, Tianqi},
    title = {On the Growth of the Extremal and Cluster Level Sets in Branching {B}rownian Motion},
    note = {arXiv:2405.17634},
    year = {2024}
}

@article{gao_laplace_2003,
    AUTHOR = {Gao, Fuchang and Hannig, Jan and Lee, Tzong-Yow and Torcaso,
              Fred},
     TITLE = {Laplace transforms via {H}adamard factorization},
   JOURNAL = {Electron. J. Probab.},
  FJOURNAL = {Electronic Journal of Probability},
    VOLUME = {8},
      YEAR = {2003},
     PAGES = {no. 13, 20},
      ISSN = {1083-6489},
   MRCLASS = {60G15},
  MRNUMBER = {1998764},
MRREVIEWER = {H. Vishnu Hebbar},
       DOI = {10.1214/EJP.v8-151},
       URL = {https://doi.org/10.1214/EJP.v8-151},
}

@article {FanZei2012a,
    AUTHOR = {Fang, Ming and Zeitouni, Ofer},
     TITLE = {Branching random walks in time inhomogeneous environments},
   JOURNAL = {Electron. J. Probab.},
  FJOURNAL = {Electronic Journal of Probability},
    VOLUME = {17},
      YEAR = {2012},
     PAGES = {no. 67, 18},
   MRCLASS = {60G50 (60J80)},
  MRNUMBER = {2968674},
MRREVIEWER = {Elcio Lebensztayn},
       DOI = {10.1214/EJP.v17-2253},
       URL = {https://doi.org/10.1214/EJP.v17-2253},
}

@article{FanZei2012b,
    AUTHOR = {Fang, Ming and Zeitouni, Ofer},
     TITLE = {Slowdown for time inhomogeneous branching {B}rownian motion},
   JOURNAL = {J. Stat. Phys.},
  FJOURNAL = {Journal of Statistical Physics},
    VOLUME = {149},
      YEAR = {2012},
    NUMBER = {1},
     PAGES = {1--9},
      ISSN = {0022-4715},
   MRCLASS = {60J65 (60F10 60J80)},
  MRNUMBER = {2981635},
MRREVIEWER = {Xiaowen Zhou},
       DOI = {10.1007/s10955-012-0581-z},
       URL = {https://doi.org/10.1007/s10955-012-0581-z},
}

@incollection{hardy_spine_2009,
    AUTHOR = {Hardy, Robert and Harris, Simon C.},
     TITLE = {A spine approach to branching diffusions with applications to {$ L^p$}-convergence of martingales},
 BOOKTITLE = {S\'{e}minaire de {P}robabilit\'{e}s {XLII}},
    SERIES = {Lecture Notes in Math.},
    VOLUME = {1979},
     PAGES = {281--330},
 PUBLISHER = {Springer, Berlin},
      YEAR = {2009},
   MRCLASS = {60J80 (60F25 60G44 60J60)},
  MRNUMBER = {2599214},
MRREVIEWER = {Jos\'{e} Villa-Morales},
       DOI = {10.1007/978-3-642-01763-6\_11},
       URL = {https://doi.org/10.1007/978-3-642-01763-6_11},
}

@article {HarHar2009,
    AUTHOR = {Harris, J. W. and Harris, S. C.},
     TITLE = {Branching {B}rownian motion with an inhomogeneous breeding
              potential},
   JOURNAL = {Ann. Inst. Henri Poincar\'{e} Probab. Stat.},
  FJOURNAL = {Annales de l'Institut Henri Poincar\'{e} Probabilit\'{e}s et
              Statistiques},
    VOLUME = {45},
      YEAR = {2009},
    NUMBER = {3},
     PAGES = {793--801},
      ISSN = {0246-0203},
   MRCLASS = {60J80 (60J65)},
  MRNUMBER = {2548504},
MRREVIEWER = {John P. Lehoczky},
       DOI = {10.1214/08-AIHP300},
       URL = {https://doi.org/10.1214/08-AIHP300},
}

@article{harris_many--few_2017,
    AUTHOR = {Harris, Simon C. and Roberts, Matthew I.},
     TITLE = {The many-to-few lemma and multiple spines},
   JOURNAL = {Ann. Inst. Henri Poincar\'{e} Probab. Stat.},
  FJOURNAL = {Annales de l'Institut Henri Poincar\'{e} Probabilit\'{e}s et Statistiques},
    VOLUME = {53},
      YEAR = {2017},
    NUMBER = {1},
     PAGES = {226--242},
      ISSN = {0246-0203},
   MRCLASS = {60J80},
  MRNUMBER = {3606740},
MRREVIEWER = {Weijuan Chu},
       DOI = {10.1214/15-AIHP714},
       URL = {https://doi.org/10.1214/15-AIHP714},
}

@article{harris_branching_2016,
    AUTHOR = {Harris, S. C. and Hesse, M. and Kyprianou, A. E.},
     TITLE = {Branching {B}rownian motion in a strip: survival near criticality},
   JOURNAL = {Ann. Probab.},
  FJOURNAL = {The Annals of Probability},
    VOLUME = {44},
      YEAR = {2016},
    NUMBER = {1},
     PAGES = {235--275},
      ISSN = {0091-1798},
   MRCLASS = {60J80 (60E10)},
  MRNUMBER = {3456337},
MRREVIEWER = {Bastien Mallein},
       DOI = {10.1214/14-AOP972},
       URL = {https://doi.org/10.1214/14-AOP972},
}

@book {Kal2021,
    AUTHOR = {Kallenberg, Olav},
     TITLE = {Foundations of modern probability},
    SERIES = {Probability Theory and Stochastic Modelling},
    VOLUME = {99},
      NOTE = {Third edition},
 PUBLISHER = {Springer, Cham},
      YEAR = {2021},
     PAGES = {xii+946},
      ISBN = {978-3-030-61871-1; 978-3-030-61870-4},
   MRCLASS = {60-01 (60A10 60G05)},
  MRNUMBER = {4226142},
MRREVIEWER = {Myron Hlynka},
       DOI = {10.1007/978-3-030-61871-1},
       URL = {https://doi.org/10.1007/978-3-030-61871-1},
}

@book {KarShr1991,
	AUTHOR = {Karatzas, Ioannis and Shreve, Steven E.},
	TITLE = {Brownian motion and stochastic calculus},
	SERIES = {Graduate Texts in Mathematics},
	VOLUME = {113},
	EDITION = {Second},
	PUBLISHER = {Springer-Verlag, New York},
	YEAR = {1991},
	PAGES = {xxiv+470},
	ISBN = {0-387-97655-8},
	MRCLASS = {60J65 (35K99 35R60 60G44 60H10 60J60)},
	MRNUMBER = {1121940},
	DOI = {10.1007/978-1-4612-0949-2},
	URL = {https://doi.org/10.1007/978-1-4612-0949-2},
}

@article{kim_maximum_2023,
    AUTHOR = {Kim, Yujin H. and Lubetzky, Eyal and Zeitouni, Ofer},
     TITLE = {The maximum of branching {B}rownian motion in {$\Bbb{R}^d$}},
   JOURNAL = {Ann. Appl. Probab.},
  FJOURNAL = {The Annals of Applied Probability},
    VOLUME = {33},
      YEAR = {2023},
    NUMBER = {2},
     PAGES = {1315--1368},
      ISSN = {1050-5164},
   MRCLASS = {60J80 (60J65 60J70)},
  MRNUMBER = {4564433},
MRREVIEWER = {Bastien Mallein},
       DOI = {10.1214/22-aap1848},
       URL = {https://doi.org/10.1214/22-aap1848},
}

@article {LalSel1987,
	AUTHOR = {Lalley, S. P. and Sellke, T.},
	TITLE = {A conditional limit theorem for the frontier of a branching {B}rownian motion},
	JOURNAL = {Ann. Probab.},
	FJOURNAL = {The Annals of Probability},
	VOLUME = {15},
	YEAR = {1987},
	NUMBER = {3},
	PAGES = {1052--1061},
	ISSN = {0091-1798},
	MRCLASS = {60J65 (60J80)},
	MRNUMBER = {893913},
	MRREVIEWER = {Luis G. Gorostiza},
	URL = {http://links.jstor.org/sici?sici=0091-1798(198707)15:3<1052:ACLTFT>2.0.CO;2-0&origin=MSN},
}

@article {LiuSch2023,
    AUTHOR = {Liu, Jiaqi and Schweinsberg, Jason},
     TITLE = {Particle configurations for branching {B}rownian motion with an inhomogeneous branching rate},
   JOURNAL = {ALEA Lat. Am. J. Probab. Math. Stat.},
  FJOURNAL = {ALEA. Latin American Journal of Probability and Mathematical Statistics},
    VOLUME = {20},
      YEAR = {2023},
    NUMBER = {1},
     PAGES = {731--803},
   MRCLASS = {60J80},
  MRNUMBER = {4585518},
       DOI = {10.30757/alea.v20-28},
       URL = {https://doi.org/10.30757/alea.v20-28},
}

@article {MaiZei2016,
    AUTHOR = {Maillard, Pascal and Zeitouni, Ofer},
     TITLE = {Slowdown in branching {B}rownian motion with inhomogeneous variance},
   JOURNAL = {Ann. Inst. Henri Poincar\'{e} Probab. Stat.},
  FJOURNAL = {Annales de l'Institut Henri Poincar\'{e} Probabilit\'{e}s et Statistiques},
    VOLUME = {52},
      YEAR = {2016},
    NUMBER = {3},
     PAGES = {1144--1160},
      ISSN = {0246-0203},
   MRCLASS = {60J80 (58J65 60J65)},
  MRNUMBER = {3531703},
MRREVIEWER = {Bastien Mallein},
       DOI = {10.1214/15-AIHP675},
       URL = {https://doi.org/10.1214/15-AIHP675},
}

@article {Mal2015a,
    AUTHOR = {Mallein, Bastien},
     TITLE = {Maximal displacement of a branching random walk in time-inhomogeneous environment},
   JOURNAL = {Stochastic Process. Appl.},
  FJOURNAL = {Stochastic Processes and their Applications},
    VOLUME = {125},
      YEAR = {2015},
    NUMBER = {10},
     PAGES = {3958--4019},
      ISSN = {0304-4149},
   MRCLASS = {60K37 (60J80)},
  MRNUMBER = {3373310},
MRREVIEWER = {Chunmao Huang},
       DOI = {10.1016/j.spa.2015.05.011},
       URL = {https://doi.org/10.1016/j.spa.2015.05.011},
}

@article {Mal2015b,
    AUTHOR = {Mallein, Bastien},
     TITLE = {Maximal displacement of {$d$}-dimensional branching {B}rownian
              motion},
   JOURNAL = {Electron. Commun. Probab.},
  FJOURNAL = {Electronic Communications in Probability},
    VOLUME = {20},
      YEAR = {2015},
     PAGES = {no. 76, 12},
   MRCLASS = {60J65 (60J70 60J80)},
  MRNUMBER = {3417448},
MRREVIEWER = {Xiaowen Zhou},
       DOI = {10.1214/ECP.v20-4216},
       URL = {https://doi.org/10.1214/ECP.v20-4216},
}

@article{OzEng2019,
  title = {Optimal survival strategy for branching Brownian motion in a Poissonian trap field},
  volume = {55},
  ISSN = {0246-0203},
  url = {http://dx.doi.org/10.1214/18-AIHP937},
  DOI = {10.1214/18-aihp937},
  number = {4},
  journal = {Annales de l’Institut Henri Poincaré,  Probabilités et Statistiques},
  publisher = {Institute of Mathematical Statistics},
  author = {\"{O}z,  Mehmet and Engl\"{a}nder,  János},
  year = {2019},
  month = Nov 
}

@article {MytRoqRyz2022,
    AUTHOR = {Mytnik, Leonid and Roquejoffre, Jean-Michel and Ryzhik, Lenya},
     TITLE = {Fisher-{KPP} equation with small data and the extremal process of branching {B}rownian motion},
   JOURNAL = {Adv. Math.},
  FJOURNAL = {Advances in Mathematics},
    VOLUME = {396},
      YEAR = {2022},
     PAGES = {Paper No. 108106, 58},
      ISSN = {0001-8708},
   MRCLASS = {60J80 (35R60 60F05)},
  MRNUMBER = {4370468},
MRREVIEWER = {Krzysztof Joachim Bartoszek},
       DOI = {10.1016/j.aim.2021.108106},
       URL = {https://doi.org/10.1016/j.aim.2021.108106},
}

@article {RobSch2021,
    AUTHOR = {Roberts, Matthew I. and Schweinsberg, Jason},
     TITLE = {A {G}aussian particle distribution for branching {B}rownian
              motion with an inhomogeneous branching rate},
   JOURNAL = {Electron. J. Probab.},
  FJOURNAL = {Electronic Journal of Probability},
    VOLUME = {26},
      YEAR = {2021},
     PAGES = {Paper No. 103, 76},
   MRCLASS = {60J80 (92D15 92D25)},
  MRNUMBER = {4290505},
MRREVIEWER = {Ekaterina Vladimirovna Bulinskaya},
       DOI = {10.1214/21-ejp673},
       URL = {https://doi.org/10.1214/21-ejp673},
}

@article {Sev1958,
    AUTHOR = {Sevastyanov, B. A.},
     TITLE = {Branching stochastic processes for particles diffusing in a
              bounded domain with absorbing boundaries},
   JOURNAL = {Teor. Veroyatnost. i Primenen.},
  FJOURNAL = {Akademiya Nauk SSSR. Teoriya Veroyatnoste\u{\i} i ee Primeneniya},
    VOLUME = {3},
      YEAR = {1958},
     PAGES = {121--136},
      ISSN = {0040-361X},
   MRCLASS = {60.00 (82.00)},
  MRNUMBER = {97867},
MRREVIEWER = {T. E. Harris},
}

@article {StaBerMal2021,
    AUTHOR = {Stasi\'{n}ski, Roman and Berestycki, Julien and Mallein, Bastien},
     TITLE = {Derivative martingale of the branching {B}rownian motion in dimension {$d \geq 1$}},
   JOURNAL = {Ann. Inst. Henri Poincar\'{e} Probab. Stat.},
  FJOURNAL = {Annales de l'Institut Henri Poincar\'{e} Probabilit\'{e}s et Statistiques},
    VOLUME = {57},
      YEAR = {2021},
    NUMBER = {3},
     PAGES = {1786--1810},
      ISSN = {0246-0203},
   MRCLASS = {60J80 (60F17 60G44 60G50)},
  MRNUMBER = {4291461},
       DOI = {10.1214/20-aihp1131},
       URL = {https://doi.org/10.1214/20-aihp1131},
}

@book {Tes2014,
	AUTHOR = {Teschl, Gerald},
	TITLE = {Mathematical methods in quantum mechanics},
	SERIES = {Graduate Studies in Mathematics},
	VOLUME = {157},
	EDITION = {Second},
	NOTE = {With applications to Schr\"{o}dinger operators},
	PUBLISHER = {American Mathematical Society, Providence, RI},
	YEAR = {2014},
	PAGES = {xiv+358},
	ISBN = {978-1-4704-1704-8},
	MRCLASS = {81-02 (47N50 81Q10 81Q12 81U10)},
	MRNUMBER = {3243083},
	MRREVIEWER = {Rupert L. Frank},
	DOI = {10.1090/gsm/157},
	URL = {https://doi.org/10.1090/gsm/157},
}

@article {Wat1965,
    AUTHOR = {Watanabe, Shinzo},
     TITLE = {On the branching process for {B}rownian particles with an absorbing boundary},
   JOURNAL = {J. Math. Kyoto Univ.},
  FJOURNAL = {Journal of Mathematics of Kyoto University},
    VOLUME = {4},
      YEAR = {1965},
     PAGES = {385--398},
      ISSN = {0023-608X},
   MRCLASS = {60.67 (60.62)},
  MRNUMBER = {178505},
MRREVIEWER = {P. E. Ney},
       DOI = {10.1215/kjm/1250524667},
       URL = {https://doi.org/10.1215/kjm/1250524667},
}

@book {Zet2005,
	AUTHOR = {Zettl, Anton},
	TITLE = {Sturm-{L}iouville theory},
	SERIES = {Mathematical Surveys and Monographs},
	VOLUME = {121},
	PUBLISHER = {American Mathematical Society, Providence, RI},
	YEAR = {2005},
	PAGES = {xii+328},
	ISBN = {0-8218-3905-5},
	MRCLASS = {34-02 (34B20 34B24 34L05 47E05)},
	MRNUMBER = {2170950},
	MRREVIEWER = {Mikl\'{o}s Horv\'{a}th},
	DOI = {10.1090/surv/121},
	URL = {https://doi.org/10.1090/surv/121},
}

\end{document}